\documentclass[preprint,3p,times,12pt]{elsarticle} 

\usepackage{amssymb}
\usepackage{lipsum}
\usepackage{amsmath}
\usepackage{cool}
\usepackage{mathtools}
\usepackage{dirtytalk}
\usepackage{cuted}
\usepackage{relsize}
\usepackage[font=small]{caption}
\usepackage{subcaption}
\usepackage{ifthen}
\usepackage[normalem]{ulem}
\usepackage{xcolor}
\usepackage{url}
\usepackage[inkscapearea=page]{svg}
\usepackage{tabularx}
\usepackage{hyperref}
\usepackage{mathrsfs}
\usepackage{enumitem}
\usepackage{multirow}
\usepackage{stmaryrd}

\newboolean{AnnotateChanges}
\setboolean{AnnotateChanges}{true} 

\ifthenelse{\boolean{AnnotateChanges}}{%
\definecolor{annotateColor}{RGB}{0,0,255}
}{%
\definecolor{annotateColor}{RGB}{0,0,0} 
}

\newcommand{\added}[2][]{%
    \ifthenelse{\boolean{AnnotateChanges}}{%
        \textcolor{annotateColor}{#2}\textbf{\textcolor{red}{#1}}}{#2}}
        
\newcommand{\deleted}[2][]{%
    \ifthenelse{\boolean{AnnotateChanges}}{%
        \textcolor{annotateColor}{\sout{#2}}\textbf{\textcolor{red}{#1}}}{}}
        
\newcommand{\replaced}[3][]{%
    \ifthenelse{\boolean{AnnotateChanges}}{%
        \textcolor{annotateColor}{#2\sout{#3}}\textbf{\textcolor{red}{#1}}}{#2}}

\usepackage{stfloats}

\usepackage{amsthm}
\newtheorem{theorem}{Theorem}
\newtheorem{assumption}{Assumption}

\newtheorem{remark}{Remark}

\newtheorem{lemma}{Lemma}
\newtheorem{weak}{Weak form}

\usepackage{tikz}

\usepackage{lineno}

\journal{}

\newcommand{\bTheta}{\boldsymbol{\Theta}}
\newcommand{\bnabla}{\boldsymbol{\nabla}}
\newcommand{\bn}{\boldsymbol{n}}
\newcommand{\bt}{\boldsymbol{t}}
\newcommand{\bm}{\boldsymbol{m}}

\usepackage{algorithm}
\usepackage[noend]{algpseudocode}

\algnewcommand\algorithmicinput{\textbf{Input:}}
\algnewcommand\Input{\item[\algorithmicinput]}
\algnewcommand\algorithmicoutput{\textbf{Output:}}
\algnewcommand\Output{\item[\algorithmicoutput]}
\algnewcommand\algorithmicassert{\textbf{Assert:}}
\algnewcommand\Assert{\item[\algorithmicassert]}

\begin{document}

\begin{frontmatter}

\title{Level-set topology optimisation with unfitted finite elements and automatic shape differentiation}

\author[qut]{Zachary J. Wegert\corref{cor}}
\ead{zach.wegert@hdr.qut.edu.au}
\author[monash]{Jordi Manyer}
\ead{jordi.manyer@monash.edu}
\author[qut]{Connor Mallon}
\ead{connor.mallon@monash.edu}
\author[monash]{Santiago Badia}
\ead{santiago.badia@monash.edu}
\author[qut]{Vivien J. Challis\corref{cor}}
\ead{vivien.challis@qut.edu.au}

\cortext[cor]{Corresponding author}

\affiliation[qut]{organization={School of Mathematical Sciences, Queensland University of Technology},
            addressline={2 George St}, 
            city={Brisbane},
            postcode={4000}, 
            state={Queensland},
            country={Australia}}

\affiliation[monash]{organization={School of Mathematics, Monash University},
            addressline={Wellington Rd}, 
            city={Clayton},
            postcode={3800}, 
            state={Victoria},
            country={Australia}}

\begin{abstract}
In this paper we develop automatic shape differentiation techniques for unfitted discretisations and link these to recent advances in shape calculus for unfitted methods. 
We extend existing analytic shape calculus results to the case where the domain boundary intersects with the boundary of the background domain. We further show that we can recover these analytic derivatives to machine precision regardless of the mesh size using the developed automatic shape differentiation techniques, drastically reducing the burden associated with the analytic derivation of these quantities. In addition, we show that we can also recover the symmetric shape Hessian. We implement these techniques for both serial and distributed computing frameworks in the Julia package GridapTopOpt and the wider Gridap ecosystem. As part of this implementation we propose a novel graph-based approach for isolated volume detection. We demonstrate the applicability of the unfitted automatic shape differentiation framework and our implementation by considering the three-dimensional minimum compliance topology optimisation of a linear elastic wheel and of a linear elastic structure in a fluid-structure interaction problem with Stokes flow. The implementation is general and allows GridapTopOpt to solve a wider range of problems on unstructured meshes without analytic calculation of shape derivatives and avoiding issues that arise when material properties are smoothed at the domain boundary. The software is open source and available at \url{https://github.com/zjwegert/GridapTopOpt.jl}.
\end{abstract}



\end{frontmatter}


\section{Introduction}\label{sec: Introduction}
Shape and topology optimisation are important computational techniques with a wide range of industrial applications and a substantial theoretical foundation \cite{TopOptMonograph, DeatonGrandhi2013,TopOptReviewSigmund,10.1016/bs.hna.2020.10.004_978-0-444-64305-6_2021}. Generally speaking, there are two main branches of techniques in topology optimisation: \textit{density} methods \cite{Bendsoe89, Rozvanyetal1992} and \textit{level-set} methods \cite{10.1016/S0045-7825(02)00559-5_2003,10.1016/j.jcp.2003.09.032_2004}. In the former, design variables are typically material densities of elements or nodes in a mesh, while, in the latter, the boundary of the domain is implicitly tracked via a level-set function and updated using an evolution equation.

In conventional level-set methods, the underlying partial differential equations are solved over a domain that is immersed in a static background domain by `smudging' the boundary over the interface with a smooth Heaviside function \cite{10.1016/S0045-7825(02)00559-5_2003,10.1016/j.jcp.2003.09.032_2004,10.1016/bs.hna.2020.10.004_978-0-444-64305-6_2021}. This allows integration to be relaxed over the whole computational domain. While this is computationally efficient and suitable for many topology optimisation problems, it can lead to undesirable computational artefacts particularly in the case of nonlinear behaviour or solid-fluid interaction. In addition, differentiation with respect to the level-set function in a smoothed boundary regime cannot fully capture the shape derivative of a functional under a boundary perturbation \cite{GridapTopOpt}. Unfitted finite element methods are an alternative approach that embed a complex computational domain into a simple bounding domain without introducing additional degrees of freedom. This works by defining an \textit{in} and \textit{out} region based on the sign of the level-set function. The boundary separating the regions is the intersection of the zero isosurface of the level-set function and the background mesh, determined using an appropriate method such as marching tetrahedra \citep{10.1002/nme.4823_2015,Badia_Verdugo_Martin_2018}. The boundary mesh can be used to integrate transmission conditions across the interface; for example, in the case of fluid-structure interaction. Unfitted methodologies are a promising way to address the above issues with conventional level-set methods because they provide a precise description of the boundary as a $D-1$ dimensional manifold. However, the small cut-cell problem affects unfitted finite element methods; the resulting discrete system is ill-posed when the portion of a cut cell in the domain interior is small \cite{dePrenter2023}. This issue has been addressed by using ghost penalty stabilisation (CutFEM) \cite{10.1002/nme.4823_2015} and discrete extension operators (AgFEM) \cite{Badia_Verdugo_Martin_2018}.  CutFEM has been successfully applied to several topology optimisation problems, including those involving linear elasticity, fluid flow, and acoustics \cite{10.1016/j.cma.2017.09.005_2018,10.1016/j.cma.2017.03.007_2017,10.1002/nme.5621_2018}. 

In recent work by \citet{Berggren_2023}, a shape calculus for unfitted discretisations was proposed based on the concept of boundary-face dilation. This can be regarded as a discretise-then-differentiate approach because it directly relies on parameterising the dilated region under a perturbation of a continuous and piecewise linear level-set function. Using this technique, \citet{Berggren_2023} rigorously derived the directional shape derivatives for general volume and surface integrals in two and three dimensions, and provided the minimum regularity requirements for these results. One of the underlying assumptions of Berggren's work is that the fictitious domain is entirely enclosed in the background domain. Although this situation is common in unfitted finite element methods \citep[e.g.,][]{10.1002/nme.4823_2015}, it is less standard in general topology optimisation problems where boundary conditions are typically applied on subsets of the boundary of the background domain. In this work, we therefore consider the extension of the results given by \citet{Berggren_2023} to the case where the domain boundary intersects the background domain boundary. In addition, we derive the directional shape derivative for general flux integrals.

The derivation of first- and second-order shape derivatives is difficult and error-prone. To this end, automatic shape differentiation, which is the application of automatic differentiation techniques to shape derivatives, has been proposed in the literature to help reduce this burden \cite{Schmidt_2018,Schmidt_Schutte_Walther_2018,Ham_Mitchell_Paganini_Wechsung_2019,Neofytou_Rios_Bujny_Menzel_Kim_2024}. These techniques have now been implemented in several software packages \cite[e.g.,][]{Dokken_Mitusch_Funke_2020,Gangl_Sturm_Neunteufel_Schöberl_2021}. The conventional methodology relies on deformations of a mesh to propagate dual numbers through the desired integrals. Although the conventional automatic shape differentiation discussed above has been successfully applied to several topology optimisation problems in the aforementioned work, it relies on having a body-fitted mesh throughout the optimisation procedure. Naturally, this poses the question of whether shape derivatives can be recovered in the context of unfitted discretisations and to what level of precision.

The automatic shape differentiation for unfitted discretisations was first proposed by \citet{mallon2024neurallevelsettopology}. They proposed the use of forward-mode differentiation to automatically compute shape derivatives as perturbations of a level-set function. However, it is not clear how the derivatives computed in their work relate to the analytic shape derivatives \cite[e.g.,][]{10.1016/j.jcp.2003.09.032_2004} or the directional shape derivatives in \cite{Berggren_2023}. In this work, we further develop the mathematics behind automatic differentiation for unfitted discretisations and link it to the directional derivatives in \cite{Berggren_2023}. We will show that the results given by \citet{Berggren_2023} can be recovered to machine precision regardless of mesh size. The developed unfitted automatic shape differentiation techniques can be used to avoid computing complicated mesh-related quantities that appear in directional derivatives of surface integrals. In addition, the technique can also compute derivatives of expressions for which there is not yet a rigorous mathematical counterpart in the framework of \citet{Berggren_2023} (e.g., shape Hessians).

We implement the analytic directional shape derivatives in \cite{Berggren_2023} and automatic shape differentiation for unfitted discretisations in the Julia package GridapTopOpt \cite{GridapTopOpt} and the wider Gridap-package ecosystem \cite{Badia2020,Verdugo2022}. GridapTopOpt provides efficient implementations of C\'ea's method \cite{10.1051/m2an/1986200303711_1986} (also known as the adjoint method) for differentiating functionals that depend on the solutions to PDEs. As a result, our implementation can be used to solve general PDE-constrained optimisation problems with a syntax that is near one-to-one with the mathematical notation. In addition, we leverage GridapDistributed \cite{Badia2022} to implement analytic derivatives, automatic shape differentiation, and unfitted evolution and reinitialisation in both serial and memory-distributed computing frameworks. These extensions of GridapTopOpt \cite{GridapTopOpt} allow us to solve an even wider range of problems on unstructured background triangulations. In addition, compatibility of our framework within the Gridap package ecosystem makes it possible to solve a significant number of cutting-edge topology optimisation problems involving complicated physical phenomena such as those posed on non-standard function spaces such as $H(\operatorname{div})$ and $H(\operatorname{curl})$.

To demonstrate the applicability of the unfitted automatic shape differentiation framework, we consider two example optimisation problems. In the first, we consider the minimum compliance optimisation of a three-dimensional linear elastic wheel. In the second problem, we consider topology optimisation of a three-dimensional linear elastic structure in a fluid-structure interaction problem with Stokes flow. We solve both computational examples in a memory-distributed manner using a ghost penalty stabilisation \cite{Burman2010}. Note that the computational framework and theoretical results given in Section \ref{sec: Shape derivatives for unfitted methods} and \ref{sec: Automatic differentiation} can be readily applied to other unfitted finite element methods.

The remainder of the paper is as follows. In Section \ref{sec: Shape derivatives for unfitted methods}, we recall the main results given in \cite{Berggren_2023} and extend them to the cases mentioned above. In Section \ref{sec: Automatic differentiation}, we discuss automatic differentiation for unfitted discretisations and compare the results to those in Section \ref{sec: Shape derivatives for unfitted methods}. In Section \ref{sec: Level-set topology optimisation}, we discuss the implementation of unfitted level-set topology optimisation. In Section \ref{sec: Examples}, we give two example results computed using the unfitted framework. Finally, in Section \ref{sec: Conclusions} we present our concluding remarks.

\section{Directional shape derivatives for unfitted methods}\label{sec: Shape derivatives for unfitted methods}

\subsection{Notation and previous results}

In the following, we recall the notation and main results of \citet{Berggren_2023}. We refer to this work and references therein for further discussion of shape calculus in the context of unfitted discretisations. For the purpose of restating these results, we assume that the unfitted domain does not intersect the boundary of the background domain. 

Suppose that we consider a background domain $D\subset\mathbb{R}^d$ with a symplectic triangulation $\mathscr{T}_h$. We use $\mathscr{S}_h$ to denote the faces of the triangulation and $\mathscr{E}_h$ to denote the subfaces. We define a level-set function $\phi\in W_h$, where $W^h=\{v\in C^0(\bar\Omega):v|_K\in P^1(K),\forall K\in\mathscr{S}_h\}$, with the property that $\phi$ partitions $D$ in the standard manner:
\begin{equation}\label{eqn: lsf def}
\begin{cases}\phi(\boldsymbol{x})<0&\text{if } \boldsymbol{x} \in \Omega, \\ \phi(\boldsymbol{x})=0 &\text{if } \boldsymbol{x} \in \partial \Omega, \\\phi(\boldsymbol{x})>0 &\text{if } \boldsymbol{x} \in D \backslash \bar{\Omega}.\end{cases}
\end{equation}
We now consider perturbations of $\Omega$ under variations of $\phi$. Namely, we consider domains
\begin{equation}
    \Omega_t=\{\boldsymbol{x}\in D:\phi_t(\boldsymbol{x})<0\}{}
\end{equation}
under perturbations
\begin{equation}\label{eqn: perturb}
    \phi_t(\boldsymbol{x})=\phi(\boldsymbol{x})+tw(\boldsymbol{x})
\end{equation}
where $t\geq0$ and $w$ is a Lagrangian basis function for $W_h$. Finally, we define the perturbed region $E_t=\Omega\setminus\bar\Omega_t$. Parameterising this domain is a fundamental ingredient in formulating the directional semiderivatives in \cite{Berggren_2023}. Note that the semiderivatives are one-sided derivatives that coincide under certain assumptions discussed below.  

\begin{figure}[!t]
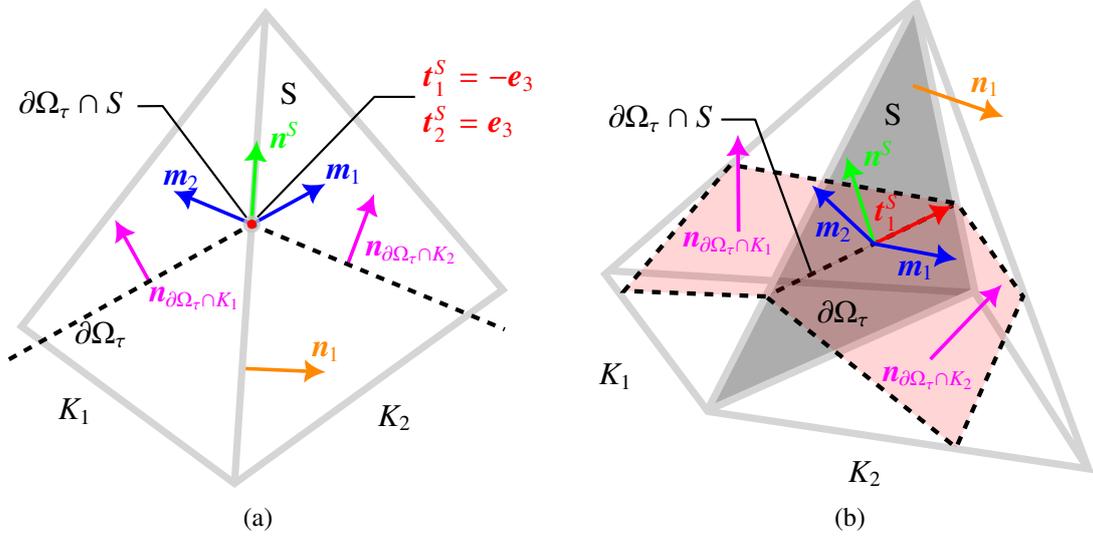

    \centering
    \begin{subfigure}{0.4\textwidth}
        \centering
        \def\svgwidth{\textwidth}
        \input{Figure1a}
        \caption{}        
    \end{subfigure}
    \hspace{1cm}
    \begin{subfigure}{0.4\textwidth}
        \centering
        \def\svgwidth{\textwidth}
        \input{Figure1b}
        \caption{}
    \end{subfigure}
    \caption{An illustration of geometric quantities required for the resolution of Equation \eqref{eqn: 6.36 og} in two dimensions (Fig. a) and three dimensions (Fig. b) where $K_1$ and $K_2$ are two neighbouring elements of $\mathscr{T}_h$ that are cut by the interface, $S=\bar{K}_1\cap \bar{K}_2$ is the facet shared by $K_1$ and $K_2$, $\bn_1$ is the normal to $S$ in the direction of $K_2$ and similarly for $\bn_2$, $\bn_{\partial\Omega_\tau\cap K_k}$ is the outward normal restricted to $\partial\Omega_\tau\cap K_k$, $\bt^S_k$ is the tangent to $\partial\Omega_\tau\cap S$, and $\bm_k$ is the co-normal defined by $\bm_k=\bt^S_k\times\bn_{\partial\Omega_\tau\cap K_k}$. In Figure (a) we use the convention that $\bt_1^S=-\boldsymbol{e}_3$ and $\bt_2^S=\boldsymbol{e}_3$. (Adapted from \cite{Berggren_2023})}
    \label{fig:geometric-setup}
\end{figure}

Before we state the main theorems in \cite{Berggren_2023}, we recall the following assumptions:
\begin{assumption}\label{assemption 6.3} There is a $t_\mathrm{max}>0$ such that for each $K\in\mathscr{T}_h$, if $\partial\Omega_t\cap K$ is either empty or non-empty for some $t\in(0,t_\mathrm{max}]$, it has the same property for each $t\in(0,t_\mathrm{max}]$.
\end{assumption}
\begin{assumption}\label{assemption 6.4} The boundary $\partial\Omega$ does not intersect any mesh points of the triangulation $\mathscr{T}_h$.
\end{assumption}
\noindent The first assumption ensures that $E_t$ has a unique parameterisation while the second ensures agreement of semiderivatives in the limits $t\rightarrow0^+$ and $t\rightarrow0^-$.

\begin{theorem}[\citet{Berggren_2023}]\label{Berggren theorem 6.7}
Under the perturbation in Equation \eqref{eqn: perturb}, the directional semiderivative of volume integral
\begin{equation}
    J_1(\phi)=\int_{\Omega(\phi)} f~\mathrm{d}\boldsymbol{x},
\end{equation}
for $f \in C^0(\bar{\mathscr{T}}_h)$, satisfies

\begin{equation}\label{vol result}
\mathrm{d}J_1(\phi ; w)=\lim _{t \rightarrow 0^{+}} \frac{1}{t}\left(J\left(\phi_t\right)-J(\phi)\right)=-\int_{\partial \Omega} f \frac{w}{\left|\partial_{\bn} \phi\right|}~\mathrm{d}s
\end{equation}

\begin{enumerate}[label=(\roman*)]
    \item If Assumption \ref{assemption 6.3} is violated, then $f$ and $\partial_{\bn} \phi$ are the limits of these functions from the interior of $\Omega$ whenever these quantities possess jump discontinuities on $\partial \Omega$.
    \item If Assumption \ref{assemption 6.3} is satisfied, the semiderivatives $t \rightarrow 0^{-}$and $t \rightarrow 0^{+}$agree.
\end{enumerate}
\end{theorem}

\begin{theorem}[\citet{Berggren_2023}]\label{Berggren theorem 6.11}
Consider the perturbation in Equation \eqref{eqn: perturb} according to Assumption \ref{assemption 6.3} of a domain respecting Assumption \ref{assemption 6.4}. The semiderivative of boundary integral
\begin{equation}\label{eqn: boundary integral}
    J_2(\phi)=\int_{\partial\Omega(\phi)}f~\mathrm{d}s
\end{equation}
for $f \in C^1\left(\bar{\mathscr{T}}_h\right)$ satisfies
\begin{align}
\mathrm{d}J_2(\phi, w) & = -\int_{\partial \Omega} \frac{\partial f}{\partial n} \frac{w}{\left|\partial_{\bn} \phi\right|}~\mathrm{d}s-\sum_{S \in \mathscr{S}_h} \int_{\partial \Omega\cap S} \bn^S \cdot \llbracket f \boldsymbol{m} \rrbracket \frac{w}{\left|\partial_{\bn^S} \phi\right|} \,\mathrm{d}\gamma,\label{eqn: 6.36 og}
\end{align}
where $\boldsymbol{n}^S$ is a unit vector located in $S$, outward-directed from $\Omega$, and orthogonal to $\partial \Omega \cap S$ for $d=3$. Moreover, $\llbracket \boldsymbol{f} \boldsymbol{m} \rrbracket=f_1 \boldsymbol{m}_1+f_2 \boldsymbol{m}_2$. Here, $\boldsymbol{m}_k$ is the co-normal for $\tau=0$ defined by 
Equation \eqref{eqn:co-normal definition}, and $f_k$ is the limit $f$ on $S$ defined by $f_k(\boldsymbol{x})=\lim _{\epsilon \rightarrow 0^{+}} f\left(\boldsymbol{x}-\epsilon \boldsymbol{m}_k\right)$ for $\boldsymbol{x} \in \partial \Omega \cap S$.
\end{theorem}
In the above, $\bn$ is the outward normal of $\partial\Omega$, $\bm_k$ is the co-normal defined by
\begin{equation}\label{eqn:co-normal definition}
    \bm_k=\bt^S_k\times\bn_{\partial\Omega_\tau\cap K_k}
\end{equation}
for $\tau\in[0,t]$, where $\bt^S_k$ is the tangent to $\partial\Omega_\tau\cap S$ for $S\in\mathscr{S}_h$ and $\bn_{\partial\Omega_\tau\cap K_k}$ is the outward normal restricted to $\partial\Omega_\tau\cap K_k$. An illustration of these quantities is given in Figure \ref{fig:geometric-setup} and a detailed discussion can be found in \cite{Berggren_2023}.

\subsection{Extension to non-empty intersections of $\partial\Omega$ and $\partial D$}

Let us now consider an extension to the case where $\partial\Omega$ has a non-empty intersection with $\partial D$ as shown in Figure \eqref{fig:extension-setup}.

\begin{figure}[!t]
    \centering
    \def\svgwidth{0.35\textwidth}
\begingroup%
  \makeatletter%
  \providecommand\color[2][]{%
    \errmessage{(Inkscape) Color is used for the text in Inkscape, but the package 'color.sty' is not loaded}%
    \renewcommand\color[2][]{}%
  }%
  \providecommand\transparent[1]{%
    \errmessage{(Inkscape) Transparency is used (non-zero) for the text in Inkscape, but the package 'transparent.sty' is not loaded}%
    \renewcommand\transparent[1]{}%
  }%
  \providecommand\rotatebox[2]{#2}%
  \newcommand*\fsize{\dimexpr\f@size pt\relax}%
  \newcommand*\lineheight[1]{\fontsize{\fsize}{#1\fsize}\selectfont}%
  \ifx\svgwidth\undefined%
    \setlength{\unitlength}{148.18583852bp}%
    \ifx\svgscale\undefined%
      \relax%
    \else%
      \setlength{\unitlength}{\unitlength * \real{\svgscale}}%
    \fi%
  \else%
    \setlength{\unitlength}{\svgwidth}%
  \fi%
  \global\let\svgwidth\undefined%
  \global\let\svgscale\undefined%
  \makeatother%
  \begin{picture}(1,1.15233301)%
    \lineheight{1}%
    \setlength\tabcolsep{0pt}%
    \put(0,0){\includegraphics[width=\unitlength,page=1]{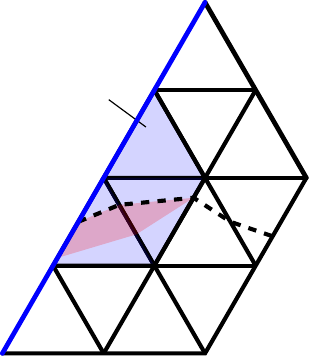}}%
    \put(0.26378277,0.61386762){\color[rgb]{0,0,0}\makebox(0,0)[lt]{\lineheight{0}\smash{\begin{tabular}[t]{l}$\boldsymbol{x}_w$\end{tabular}}}}%
    \put(0.17177157,0.82421991){\color[rgb]{0,0,0}\makebox(0,0)[lt]{\lineheight{0}\smash{\begin{tabular}[t]{l}$\operatorname{supp}w$\end{tabular}}}}%
    \put(0,0){\includegraphics[width=\unitlength,page=2]{Figure2-pdf.pdf}}%
    \put(0.7554765,0.4518076){\color[rgb]{0,0,0}\makebox(0,0)[lt]{\lineheight{0}\smash{\begin{tabular}[t]{l}$\partial\Omega$\end{tabular}}}}%
    \put(0.29169043,0.12266058){\color[rgb]{0,0,0}\makebox(0,0)[lt]{\lineheight{0}\smash{\begin{tabular}[t]{l}$\Omega$\end{tabular}}}}%
    \put(0.61304584,0.94462529){\color[rgb]{0,0,0}\makebox(0,0)[lt]{\lineheight{0}\smash{\begin{tabular}[t]{l}$\partial D$\end{tabular}}}}%
    \put(0.2755417,0.38182724){\color[rgb]{0,0,0}\makebox(0,0)[lt]{\lineheight{0}\smash{\begin{tabular}[t]{l}$E_t$\end{tabular}}}}%
  \end{picture}%
\endgroup%

    \caption{A visualisation of intersections between $\partial\Omega$ and $\partial D$ considered in this work.}
    \label{fig:extension-setup}
\end{figure}

The result of Theorem \ref{Berggren theorem 6.7} and the related results in \cite{Berggren_2023} remain unchanged for the non-empty intersections considered here. This is because the computation in these results is performed locally on each element $K\in\mathscr{T}_h$ and does not require information from adjacent elements. This is quite different from the situation in Theorem \ref{Berggren theorem 6.11}, where the jump of quantities over element interfaces is required. In particular, Lemma 6.10 from \cite{Berggren_2023} must be extended as follows: 

\begin{lemma}\label{lemma: ext lemma 6.10} Let $t\in(0,t_{\mathrm{max}}]$ according to Assumption \ref{assemption 6.3} and consider the perturbation in Equation \eqref{eqn: perturb}. Moreover, let $f\in C^1(\overline{\mathscr{T}}_h)$ and $\bTheta\in L^\infty(E_t)^d$ such that $\bTheta\rvert_{E_t\cap K}\in C^1(\overline{E_t\cap K})^d$ for each $K\in\mathscr{T}_h$. In addition, suppose that we have non-empty $\partial \Omega\cap\partial D$. Then, we have
\begin{align}
    \sum_{K\in\mathscr{T}_h}\int_{E_t\cap K}\left[(\bnabla\cdot\bTheta)f+\bTheta\cdot\bnabla f\right]\,\mathrm{d}V&=\int_{\partial E_t\setminus\partial D}\bn\cdot\bTheta f\,\mathrm{d}s+\sum_{S\in\mathscr{S}_h\setminus\partial D}\int_0^t\int_{\partial\Omega_\tau\cap S}\bn^S\cdot\llbracket f\bTheta\times\bt^S\rrbracket\frac{w}{\lvert\partial_{\bn^S}\phi_\tau\rvert}\,\mathrm{d}\gamma\,\mathrm{d}\tau\nonumber\\
    &\hspace{0.5cm}+\sum_{S\in\mathscr{S}_h\cap\partial D}\int_0^t\int_{\partial\Omega_\tau\cap S}\bn^S\cdot(f\bTheta\times\bt^S)\frac{w}{\lvert\partial_{\bn^S}\phi_\tau\rvert}\,\mathrm{d}\gamma\,\mathrm{d}\tau.
\end{align}
\end{lemma}
\begin{proof}
For the case where $\partial \Omega\cap\partial D$ is empty, the proof is as in Lemma 6.10. from \cite{Berggren_2023}.

For the case of non-empty $\partial \Omega\cap\partial D$, we assume that $E_t=\Omega\setminus\overline{\Omega}_t$ is nonempty and denote cells $K\in\mathscr{T}_h$ with faces $S\in\mathscr{S}_h$. Integrating $(\bnabla\cdot\bTheta)f$ on $E_t\cap K$ by parts gives
\begin{equation}
    \int_{E_t\cap K}\left[(\bnabla\cdot\bTheta)f+\bTheta\cdot\bnabla f\right]\,\mathrm{d}V=\int_{\partial(E_t\cap K)}f\bn\cdot\bTheta\,\mathrm{d}s,\quad K\in\mathscr{T}_h\label{int by parts b4 sum}
\end{equation}
and summing the resulting terms yields
\begin{align}
    &\sum_{K\in\mathscr{T}_h}\int_{E_t\cap K}\left[(\bnabla\cdot\bTheta)f+\bTheta\cdot\bnabla f\right]\,\mathrm{d}V\nonumber\\
    &=\int_{\partial E_t\setminus\partial D}f\bn\cdot\bTheta \,\mathrm{d}s+\sum_{S\in\mathscr{S}_h\setminus\partial D}\int_{E_t\cap S}f_1\bn_1\cdot\bTheta_1+f_2\bn_2\cdot\bTheta_2\,\mathrm{d}s+\sum_{S\in\mathscr{S}_h\cap\partial D}\int_{E_t\cap S}f\bn\cdot\bTheta\,\mathrm{d}s.\label{int by parts1}
\end{align}
When $E_t\cap\partial D$ is empty, the last term of Equation \eqref{int by parts1} vanishes and the remaining terms are as in Equation (6.32) from \cite{Berggren_2023}, so the proof is complete. 

For non-empty $E_t\cap\partial D$, application of Lemma 6.9 from \cite{Berggren_2023} to the last two integrals in Equation \eqref{int by parts1} gives
\begin{equation}
\int_{E_t\cap S}[f_1\bn_1\cdot\bTheta_1+f_2\bn_2\cdot\bTheta_2]\,\mathrm{d}s = \int_0^t\int_{\partial\Omega_\tau\cap S}\left(f_1\bn_1\cdot\bTheta_1+f_2\bn_2\cdot\bTheta_2\right)\frac{w}{\lvert\partial_{\bn^S}\phi_\tau\rvert}\,\mathrm{d}\gamma\,\mathrm{d}\tau,\quad S\in\mathscr{S}_h\setminus\partial D, \label{eq: 8}
\end{equation}
and
\begin{equation}
\int_{E_t\cap S}f\bn\cdot\bTheta\,\mathrm{d}s = \int_0^t\int_{\partial\Omega_\tau\cap S}\left(f\bn\cdot\bTheta\right)\frac{w}{\lvert\partial_{\bn^S}\phi_\tau\rvert}\,\mathrm{d}\gamma\,\mathrm{d}\tau,\quad S\in\mathscr{S}_h\cap\partial D.\label{eq: 9}
\end{equation}
For Equation \eqref{eq: 8} we have 
\begin{equation}
\left.f_1\bn_1\cdot\bTheta_1+f_2\bn_2\cdot\bTheta_2 = \bn^S\cdot\llbracket f\bTheta\times\bt^S\rrbracket\right\rvert_S\label{eq: 10}
\end{equation}
by Equation (6.34) from \cite{Berggren_2023}. On the other hand, for Equation \eqref{eq: 9} the term $f\bn\cdot\bTheta$ is single valued and has no jump. This is the same case for $\bt^S$. As a result, we have
\begin{equation}
    f\bn\cdot\bTheta = f\left(\bt^S\times\bn^S\right)\cdot\bTheta = \left.\bn^S\cdot\left(f\bTheta\times\bt^S\right)\right\rvert_S,\label{eq: 11}
\end{equation}
where we use that $\bn=\bt^S\times\bn^S$ and properties of the triple product in the first and second equalities, respectively.

Substituting Equations \eqref{eq: 10} and \eqref{eq: 11} into Equations \eqref{eq: 8} and \eqref{eq: 9}, and substituting the result into Equation \eqref{int by parts1} completes the proof.
\end{proof}

The extension of Theorem \ref{Berggren theorem 6.11} to the case of non-empty $\partial\Omega\cap\partial D$ is then as follows:

\begin{theorem}\label{thrm: ext theorem 6.11}
Consider the perturbation in Equation \eqref{eqn: perturb} according to Assumption \ref{assemption 6.3} of a domain respecting Assumption \ref{assemption 6.4}. Suppose in addition that we have non-empty $\partial \Omega\cap\partial D$. The semiderivative of Equation \eqref{eqn: boundary integral} for $f \in C^1\left(\bar{T}_h\right)$ satisfies 
\begin{align}
\mathrm{d} J_2(\phi, w) =-\int_{\partial \Omega} \frac{\partial f}{\partial n} \frac{w}{\left|\partial_{\bn} \phi\right|}~\mathrm{d}s-\sum_{S \in \mathscr{S}_h\setminus\partial D} \int_{\partial \Omega\cap S} \bn^S \cdot \llbracket f \boldsymbol{m} \rrbracket \frac{w}{\left|\partial_{\bn^S} \phi\right|} \,\mathrm{d}\gamma-\sum_{S \in \mathscr{S}_h\cap\partial D} \int_{\partial \Omega\cap S} \bn^S \cdot (f \boldsymbol{m}) \frac{w}{\left|\partial_{\bn^S} \phi\right|} \,\mathrm{d}\gamma.\label{eqn: 6.36 ext}
\end{align}
\end{theorem}

Using Lemma \ref{lemma: ext lemma 6.10}, the proof of the extension in Theorem \ref{thrm: ext theorem 6.11} is analogous to the one given in \cite{Berggren_2023}. 
Note that assumption \ref{assemption 6.4} restricts the analysis to the case that $\Omega$ is not part of $\partial D$, namely $\phi(\boldsymbol{x})\neq0$ on $\partial D$, except for measure-zero intersections (i.e., point intersections for $d=2$ or surface intersections for $d=3$). 

\subsection{Directional derivative of flux integrals}
The following result considers the directional derivative of flux integrals of the form
\begin{equation}\label{eqn: J3}
J_3(\phi)=\int_{\partial\Omega(\phi)}\boldsymbol{f}\cdot\boldsymbol{n}~\mathrm{d}s.
\end{equation}
These are necessary for problems involving loads over $\partial\Omega$ (e.g., fluid-structure interaction).

\begin{lemma}
    Consider the perturbation in Equation \eqref{eqn: perturb} according to Assumption \ref{assemption 6.3} of a domain respecting Assumption \ref{assemption 6.4} and where the intersection $\partial\Omega\cap\partial D$ is empty. The semiderivative of Equation \eqref{eqn: J3} for $\boldsymbol{f} \in \left[C^1\left(\bar{\mathscr{T}}_h\right)\right]^d$ satisfies
\begin{equation}
    \mathrm{d}J_3(\phi ; w)=-\int_{\partial \Omega} \boldsymbol{\nabla}\cdot\boldsymbol{f} \frac{w}{\left|\partial_{\bn} \phi\right|}~\mathrm{d}s.
\end{equation}
\end{lemma}
\begin{proof}
    Since $\boldsymbol{f} \in \left[C^1\left(\bar{\mathscr{T}}_h\right)\right]^d$ and the intersection $\partial\Omega\cap\partial D$ is empty, Equation \eqref{eqn: J3} can be rewritten using the divergence theorem as
    \begin{equation}
J_3(\phi)=\int_{\Omega(\phi)}\boldsymbol{\nabla}\cdot\boldsymbol{f}~\mathrm{d}s.
    \end{equation}
    Application of Theorem \ref{Berggren theorem 6.7} yields the desired result.
\end{proof}

\begin{remark}
We note that extension of this result to the case of non-empty $\partial\Omega\cap\partial D$ as well as more general integrands (e.g., $g(\boldsymbol{n})$) requires further mathematical development. The latter has been considered in the context of domain transformations \cite[e.g.,][]{10.1007/s00466-017-1383-6_2017}, however, the analog in the framework of \cite{Berggren_2023} is not immediately clear.
\end{remark}

\section{Automatic differentiation}\label{sec: Automatic differentiation}

In this section, we consider automatic differentiation for unfitted discretisations. We first recall some basic properties of dual numbers. Then we consider how dual numbers propagate through quadrature for unfitted discretisations and verify that we exactly recover the semiderivatives from the previous section.

\subsection{Dual numbers}
Forward-mode automatic differentiation, in which the function and its derivative are computed concurrently, is based on the properties of dual numbers to evaluate derivatives exactly \cite{Baydin_Pearlmutter_Radul_Siskind_2018}. A dual number is defined as $a+b\varepsilon$ with the property $\varepsilon^2=0$. As a result, the Taylor series of a differentiable function $f(a+b\varepsilon)$ about $a$ is
\begin{equation}\label{eqn: taylor duals}
    f(a+b\varepsilon)=f(a)+b\varepsilon f'(a)
\end{equation}
where the higher order terms vanish thanks to $\varepsilon^2=0$. This allows us to recover the derivative of $f$ exactly via the dual component of $f(a+b\varepsilon)$.

\subsection{Automatic differentiation for unfitted discretisations}
We now demonstrate how dual numbers propagate through quadrature on a cut triangulation. In particular, we consider piecewise linear cuts defined via a level-set function (e.g., Figure \ref{fig:fig 3}).

\begin{figure}[!t]
    \centering
    \includegraphics[width=0.5\linewidth]{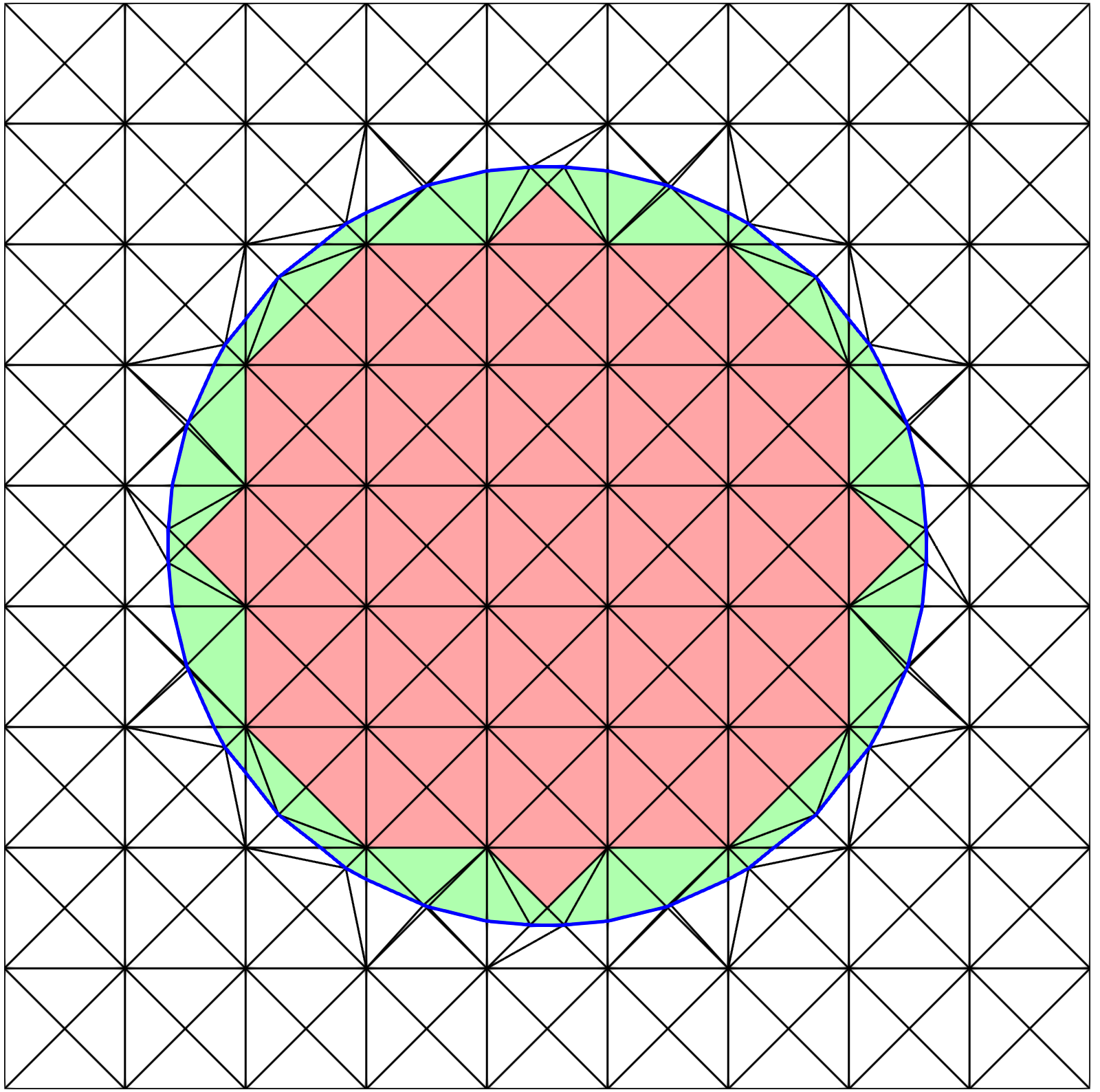}
    \caption{An example subtriangulation of a domain $\Omega(\phi)$ defined by $\phi(\boldsymbol{x})\leq0$. The green cells, denoted $\Omega_{h,\text{CUT}}(\phi)$, are those that include the interface $\Gamma(\phi)$. The red cells, denoted $\Omega_{h,\text{IN}}$, are the remaining internal cells.}
    \label{fig:fig 3}
\end{figure}

Consider the functional
\begin{equation}
    J_1(\phi)=\int_{\Omega(\phi)} f~\mathrm{d}\boldsymbol{x}
\end{equation}
where, for the purpose of this demonstration, we let $\Omega(\phi)\subset D\subset\mathbb{R}^2$ as in Figure \ref{fig:fig 3}, and $f$ is sufficiently differentiable.

Suppose that we generate a fully-symplectic subtriangulation of the background mesh $\mathscr{T}_h$ at the interface defined by $\phi(\boldsymbol{x})=0$. This can be accomplished using the algorithms described in \cite{10.1002/nme.4823_2015} and the references therein. We denote this subtriangulation as $\Omega_h(\phi)=\Omega_{h,\text{CUT}}\cup\Omega_{h,\text{IN}}$ where $\Omega_{h,\text{CUT}}$ and $\Omega_{h,\text{IN}}$ are the simplices that include the interface ($\phi(\boldsymbol{x})=0$) and the bulk ($\phi(\boldsymbol{x})<0$), respectively. Figure \ref{fig:fig 3} shows an example visualisation of this subtriangulation. Finally, we let $K(\phi)\in\Omega_{h,\text{CUT}}$ and $K\in\Omega_{h,\text{IN}}$ denote a cut subcell and bulk element, respectively. We assume Assumption \ref{assemption 6.3} and \ref{assemption 6.4} are satisfied, as such the topology of $\Omega_{h,\text{CUT}}$ and $\Omega_{h,\text{IN}}$ are invariant under small perturbations of $\phi$ while the geometry of elements of $\Omega_{h,\text{CUT}}$ can vary, hence the notational dependence of $K(\phi)\in\Omega_{h,\text{CUT}}$. 

The functional $J_1$ can now be rewritten as
\begin{align}
    J_1(\phi)=\sum_{K(\phi)\in\Omega_{h,\text{CUT}}}\int_{K(\phi)} f~\mathrm{d}\boldsymbol{x}+\sum_{K\in\Omega_{h,\text{IN}}}\int_{K} f~\mathrm{d}\boldsymbol{x}.
\end{align}
Let us focus our attention on the integral over $K(\phi)$ and assume without loss of generality that the element is cut as in Figure \ref{fig: 4}.
\begin{figure}
    \centering
    \def\svgwidth{0.65\textwidth}
    \input{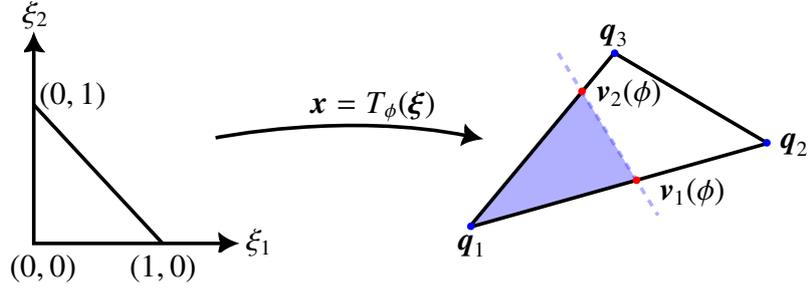}
    \caption{A visualisation of the map $T_\phi:\hat{K}\rightarrow K(\phi)$ from the reference element $\hat{K}$ to a cut subcell $K(\phi)$. The intersection points $\boldsymbol{v}_1(\phi)$ and $\boldsymbol{v}_2(\phi)$ of the cut subcell are computed using $\phi(\boldsymbol{q}_1)$, $\phi(\boldsymbol{q}_2)$, and $\phi(\boldsymbol{q}_3)$ via interpolation.} 
    \label{fig: 4}
\end{figure}
In addition, suppose that we have $T_\phi:\hat{K}\rightarrow K(\phi)$ that maps the first-order reference element to the cut subcell. For the cut subcell in Figure \ref{fig: 4}, the mapping $T_\phi$ is defined as
\begin{equation}
    T_\phi(\boldsymbol{\xi})=J_{K(\phi)}\xi+\boldsymbol{q}_1
\end{equation}
where $\boldsymbol{q}_1=(x_{\boldsymbol{q}_1},y_{\boldsymbol{q}_1})$ and $J_{K(\phi)}$ is the Jacobian of $T$ defined as
\begin{equation}
    J_{K(\phi)}=\begin{pmatrix}
        x_{\boldsymbol{v}_1}(\phi)-x_{\boldsymbol{q}_1}&x_{\boldsymbol{v}_2}(\phi)-x_{\boldsymbol{q}_1}\\
        y_{\boldsymbol{v}_1}(\phi)-y_{\boldsymbol{q}_1}&y_{\boldsymbol{v}_2}(\phi)-y_{\boldsymbol{q}_1}
    \end{pmatrix}.
\end{equation}
The above coordinates $\boldsymbol{v_1}$ and $\boldsymbol{v_2}$ depend on $\phi$ via the following interpolants
\begin{align}
    \boldsymbol{v}_1(\phi)&=\boldsymbol{q}_1+\frac{\lvert\phi(\boldsymbol{q}_1)\rvert}{\lvert\phi(\boldsymbol{q}_1)\rvert+\lvert\phi(\boldsymbol{q}_2)\rvert}(\boldsymbol{q}_2-\boldsymbol{q}_1),\label{eqn: interp 1}\\
    \boldsymbol{v}_2(\phi)&=\boldsymbol{q}_1+\frac{\lvert\phi(\boldsymbol{q}_1)\rvert}{\lvert\phi(\boldsymbol{q}_1)\rvert+\lvert\phi(\boldsymbol{q}_3)\rvert}(\boldsymbol{q}_3-\boldsymbol{q}_1)\label{eqn: interp 2}.
\end{align}
At this point, our map $T_\phi$ is well defined and we can proceed with our quadrature rule:
\begin{align}
    \int_{K(\phi)} f~\mathrm{d}\boldsymbol{x}=\int_{\hat{K}} f(T_\phi(\boldsymbol{\xi}))\lvert J_{K(\phi)}\rvert~\mathrm{d}\boldsymbol{\xi}=\sum_{i=1}^N\omega_if(T_\phi(\boldsymbol{\xi}_i))\lvert J_{K(\phi)}\rvert.
\end{align}
Thus, our original functional $J_1$ can be written as
\begin{align}\label{eqn: quadrature with phi depend}
    J_1(\phi)=\sum_{K(\phi)\in\Omega_{h,\text{CUT}}}\sum_{i=1}^N\omega_if(T_\phi(\boldsymbol{\xi}_i))\lvert J_{K(\phi)}\rvert+\sum_{K\in\Omega_{h,\text{IN}}}\sum_{i=1}^N\omega_if(T(\boldsymbol{\xi}_i))\lvert J_{K}\rvert
\end{align}
where the latter integral has been rewritten using a standard quadrature rule.

Now that the dependence of the quadrature rule in Equation \eqref{eqn: quadrature with phi depend} on $\phi$ is well defined, we can consider computing the directional derivative of $J_1$ with respect to $\phi$ in the direction $w$ using dual numbers. 

Let $F(s)=J_1(\phi+sw)$ and consider the Taylor series of $F(t+\varepsilon)$ about $t=0$. Equation \eqref{eqn: taylor duals} gives
\begin{equation}
    F(\varepsilon) = F(0) + \varepsilon F'(t)\rvert_{t=0},
\end{equation}
In terms of $J_1$ this becomes
\begin{equation}
    J_1(\phi+\varepsilon w)=J_1(\phi)+\varepsilon\left.\frac{\mathrm{d}}{\mathrm{d}t}\right\rvert_{t=0}J_1(\phi+t w)=J_1(\phi)+\varepsilon~\mathrm{d}J_1(\phi,w),
\end{equation}
where the last equality is purely notational. Thus, we recover the directional derivative as the dual component of $J_1(\phi+\varepsilon w)$. Importantly, we capture the change in shape due to the propagation of dual numbers through the map $T_\phi$ and its Jacobian $\lvert J_{K(\phi)}\rvert$ .

From an algorithmic point of view, computing the derivative using automatic differentiation can be summarised as follows:
\begin{enumerate}
    \item Construct the dual numbers corresponding to the level-set values at each node (i.e., replace $\phi_i$ by $\phi_i+\varepsilon_i$ where $\varepsilon_i$ is the corresponding dual number at node $i$).
    \item Compute the nodal coordinates of the cut subcells using $\phi_i+\varepsilon_i$ in Equation \eqref{eqn: interp 1} and \eqref{eqn: interp 2}.
    \item Construct the map $T_{\phi_i+\varepsilon_i}$ and the corresponding Jacobian $J_{K(\phi_i+\varepsilon_i)}$.
    \item Compute the integral using the quadrature rule in Equation \eqref{eqn: quadrature with phi depend} and assemble the vector of dual components. 
\end{enumerate}

\begin{remark}
    Note that the above demonstrative example is verbatim for the case of $J_2$ and $J_3$ in Equation \eqref{eqn: boundary integral} and \eqref{eqn: J3}, respectively, with the exception that the normal in $J_3$ should be calculated using the nodal coordinates resulting from Step (2) of the above summary. This ensures that the dual numbers are propagated correctly in the calculation of the normal vector. 
\end{remark}

\begin{remark}
    We further note that, among other things, the above technique can also be used to evaluate derivatives of functionals that depend on additional fields (e.g., solutions to finite element problems) or whose rigorous mathematical counterpart is not yet known. For example, 
    \begin{equation}
        J_4(\phi)=\int_{\partial\Omega(\phi)} g(\boldsymbol{n})~\mathrm{d}s.
    \end{equation}
    For functionals that depend on the solutions to PDEs, C\'ea's method can be utilised \citep{10.1051/m2an/1986200303711_1986}.
\end{remark}

\subsection{Implementation}\label{sec: ad implem}
We implement the analytic directional derivatives discussed in Section \ref{sec: Shape derivatives for unfitted methods} and automatic differentiation methods discussed above in the Julia package GridapTopOpt \cite{GridapTopOpt} using the machinery provided by GridapEmbedded \cite{GridapEmbedded} and Gridap \cite{Badia2020,Verdugo2022}. GridapEmbedded provides algorithms and abstract methods for computation on subtriangulations, while Gridap provides the underlying finite element method. We leverage GridapDistributed \cite{Badia2022} to implement both the analytic derivatives and automatic differentiation for unfitted discretisations in both serial and distributed computing frameworks. The latter readily follows from the serial implementation, as all computation is cell-wise. As such, the serial implementation can be mapped over each local partition.

It should be noted that for shape differentiation of functionals that depend on the solutions to PDEs, C\'ea's method \citep{10.1051/m2an/1986200303711_1986} is utilised in conjunction with automatic shape differentiation. We refer the reader to Section 2.6 of \citet{GridapTopOpt} for further discussion.

\subsection{Verification}
In this section, we verify that automatic differentiation for unfitted discretisations can exactly compute their analytic counterparts. To show this, we consider the four level-set functions shown in Figure \ref{fig: fig 5}. 

\begin{figure}[!t]
    \centering
    \begin{subfigure}{0.43\textwidth}
        \centering
        \includegraphics[width=\textwidth]{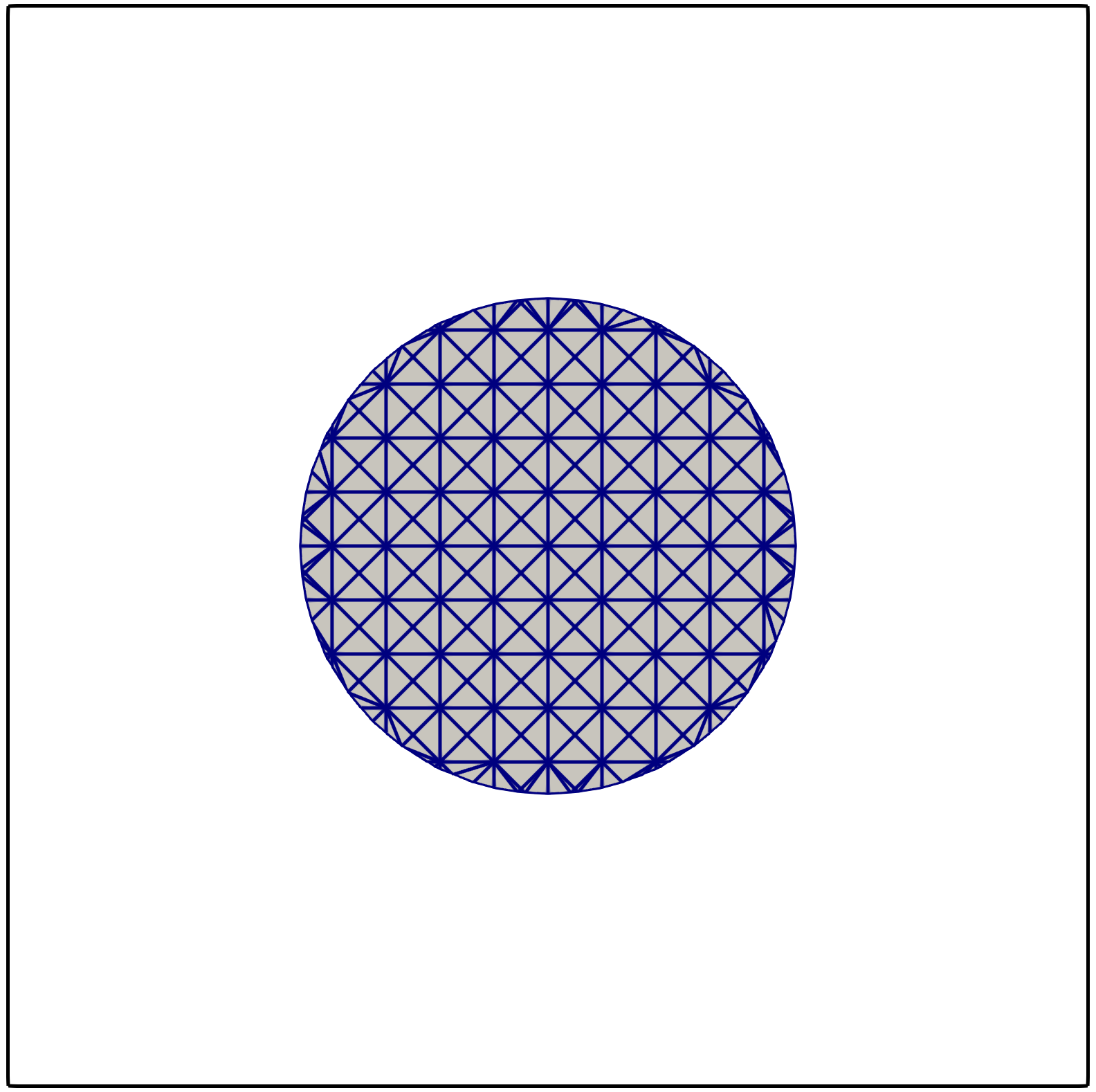}
        \caption{$\sqrt{(x-0.5)^2+(y-0.5)^2}-0.23\leq0$}
        \label{5a}
    \end{subfigure}
    \hspace{1cm}
    \begin{subfigure}{0.43\textwidth}
        \centering
        \includegraphics[width=\textwidth]{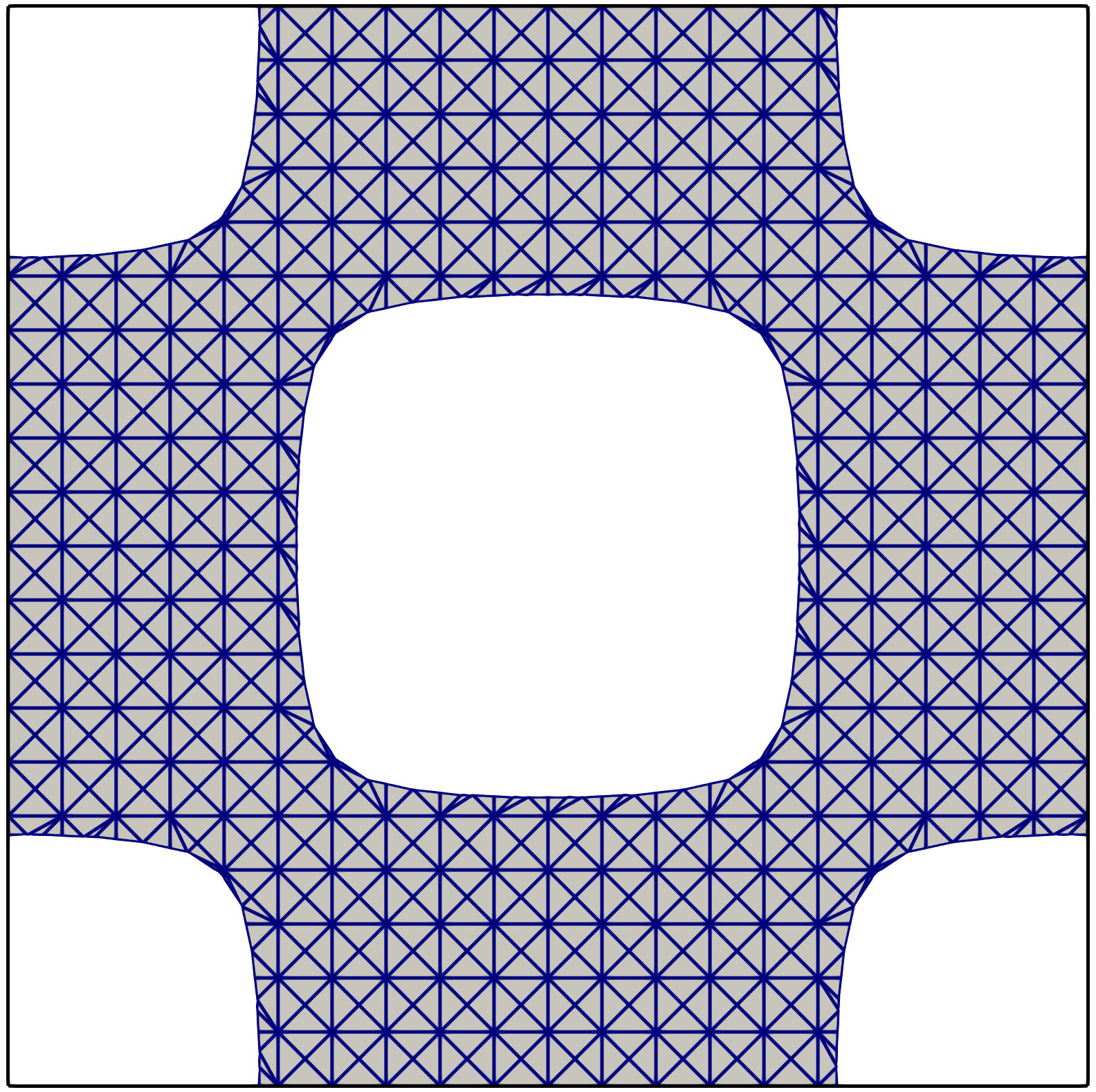}
        \caption{$\cos(2\pi x)\cos(2\pi y)-0.11\leq0$}
        \label{5b}
    \end{subfigure}
    \begin{subfigure}{0.43\textwidth}
        \centering
        \includegraphics[width=\textwidth]{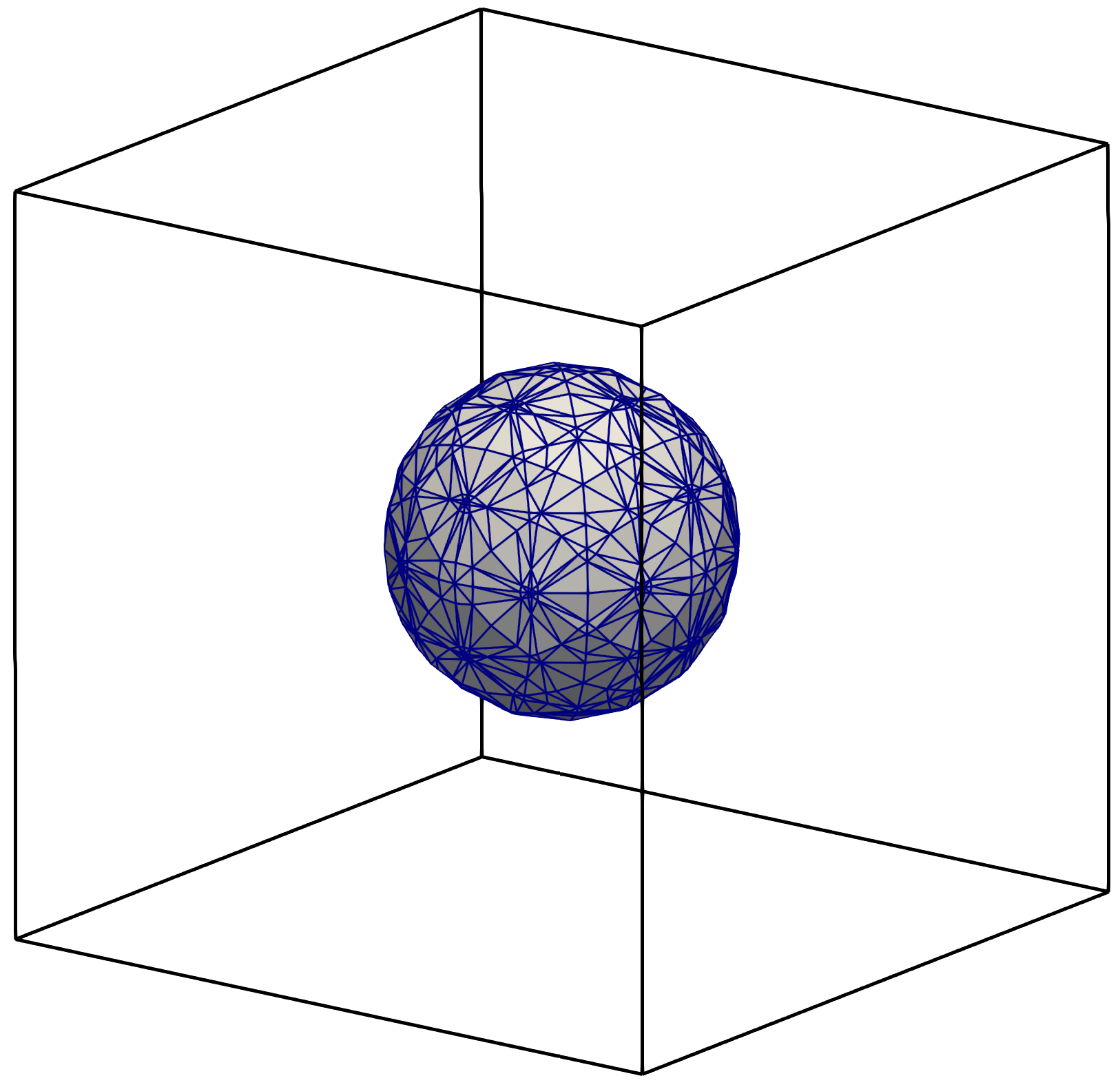}
        \caption{$\sqrt{(x-0.5)^2+(y-0.5)^2+(z-0.5)^2}-0.23\leq0$}
        \label{5c}
    \end{subfigure}
    \hspace{1cm}
    \begin{subfigure}{0.43\textwidth}
        \centering
        \includegraphics[width=\textwidth]{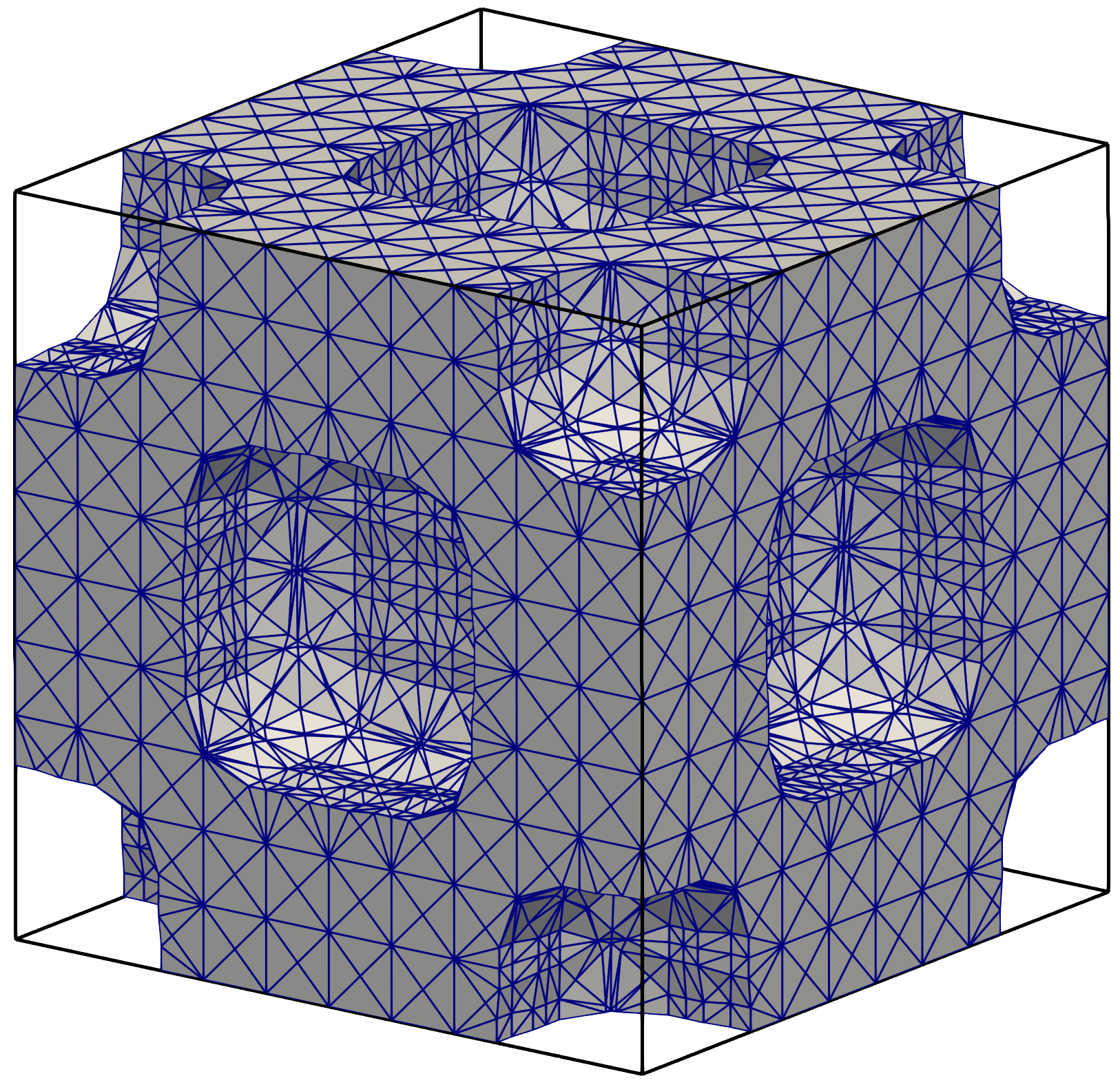}
        \caption{$\cos(2\pi x)\cos(2\pi y)\cos(2\pi z)-0.11\leq0$}
        \label{5d}
    \end{subfigure}
    \caption{A visualisation of the geometries and discretisations used for validation of automatic differentiation}
    \label{fig: fig 5}
\end{figure}
\begin{table}[!b]
\small
\centering
\caption{Verification of directional derivatives resulting from automatic differentiation against the exact expressions from Section \ref{sec: Shape derivatives for unfitted methods} and finite differences. In the above, we take $f(x,y)=\cos(x+y)$ and $g(\boldsymbol{n})=\lVert\boldsymbol{n}-\boldsymbol{n}_g\rVert^2$ with $\boldsymbol{n}_g=\frac{\boldsymbol{\nabla}g}{\rVert\boldsymbol{\nabla}g\lVert}$ and $g(x,y)=x - \frac{1}{10}\sin\left(\frac{\pi y}{3}\right)$. The latter is based on an expression in \cite[Sec. 6.3.,][]{10.1007/s00466-017-1383-6_2017}. Note that we use `N/A' for directional derivatives that cannot readily be calculated analytically.}
\label{tab:verify}
\begin{tabular}{c|c|c|c}
F & Level-set function & $\lVert\mathrm{d}F_{\mathrm{AD}}-\mathrm{d}F_{\mathrm{exact}}\rVert_{\infty}$ & $\lVert\mathrm{d}F_{\mathrm{AD}}-\mathrm{d}F_{\mathrm{FDM}}\rVert_{\infty}$ \\ \hline
\multirow{4}{*}{$\displaystyle\int_{\Omega(\phi)} f~\mathrm{d}\boldsymbol{x}$} & Fig. \ref{5a} & {3.47$\times10^{-17}$} & {4.44$\times10^{-10}$} \\
 & Fig. \ref{5b} & {7.81$\times10^{-18}$} & {2.09$\times10^{-11}$} \\
 & Fig. \ref{5c} & {3.90$\times10^{-18}$} & {4.22$\times10^{-11}$} \\
 & Fig. \ref{5d} & {1.73$\times10^{-18}$} & {7.04$\times10^{-12}$} \\ \hline
\multirow{4}{*}{$\displaystyle\int_{\partial\Omega(\phi)}f~\mathrm{d}s$} & Fig. \ref{5a} & {3.47$\times10^{-16}$} & {3.59$\times10^{-8}$} \\
 & Fig. \ref{5b} & {3.05$\times10^{-16}$} & {1.03$\times10^{-9}$} \\
 & Fig. \ref{5c} & {2.64$\times10^{-16}$} & {9.39$\times10^{-10}$} \\
 & Fig. \ref{5d} & {2.92$\times10^{-16}$} & {1.62$\times10^{-10}$} \\ \hline
\multirow{4}{*}{$\displaystyle\int_{\partial\Omega(\phi)}\boldsymbol{f}\cdot\boldsymbol{n}~\mathrm{d}s$} & Fig. \ref{5a} & $4.30\times10^{-16}$ & $1.52\times10^{-9}$ \\
 & Fig. \ref{5b} & N/A & $6.78\times10^{-11}$ \\
 & Fig. \ref{5c} & $3.12\times10^{-17}$ & $1.44\times10^{-10}$ \\
 & Fig. \ref{5d} & N/A & $1.63\times10^{-11}$ \\ \hline
\multirow{4}{*}{$\displaystyle\int_{\partial\Omega(\phi)} g(\boldsymbol{n})~\mathrm{d}s$} & Fig. \ref{5a} & N/A & $9.35\times10^{-8}$ \\
 & Fig. \ref{5b} & N/A & $2.35\times10^{-9}$ \\
 & Fig. \ref{5c} & N/A & $2.50\times10^{-9}$ \\
 & Fig. \ref{5d} & N/A & $6.09\times10^{-10}$
\end{tabular}
\end{table}

Table \ref{tab:verify} shows the results when comparing the analytic expressions, automatic differentiation, and finite differences. These demonstrate that the derivatives computed using the exact expressions from Section \ref{sec: Shape derivatives for unfitted methods} exactly correspond to the derivatives computed using automatic differentiation. This correspondence is to within machine precision regardless of mesh size because the exact expressions in Section \ref{sec: Shape derivatives for unfitted methods} are derived using a discretise-then-differentiate approach. 

The implementation discussed in Section \ref{sec: ad implem} can also be leveraged for computation of shape Hessians for unfitted discretisations. For example, we computed the Hessian of $F(\phi)=\int_{\partial\Omega(\phi)} \lVert\boldsymbol{n}-\boldsymbol{n}_g\rVert^2~\mathrm{d}s$ using automatic differentiation and finite differences and found that the absolute error was $1.85\times10^{-5}$. This suggests an interesting avenue of mathematical development for the analytic methods proposed in \cite{Berggren_2023}.

\section{Level-set topology optimisation}\label{sec: Level-set topology optimisation}

In the following, we discuss our implementation of level set-based topology optimisation in the framework of unfitted finite elements.

\subsection{Boundary evolution}
In this work, we use the approach proposed by \citet{10.1002/nme.4823_2015} to update the level-set function. Namely, we evolve the level-set function by solving a transport equation
\begin{equation}\label{eqn: transport}
    \begin{cases}
        \displaystyle\pderiv{\phi(t,\boldsymbol{x})}{t}+\boldsymbol{\beta}\cdot\boldsymbol{\nabla}\phi(t,\boldsymbol{x})=0,\\
        \phi(0,\boldsymbol{x})=\phi_0(\boldsymbol{x}),\\
        \boldsymbol{x}\in D,~t\in(0,T),
    \end{cases}
\end{equation}
where $\boldsymbol{\beta}$ is a velocity field. In this work $\boldsymbol{\beta}$ is computed as $\boldsymbol{\beta}=\boldsymbol{n}g_\Omega$,
where $g_\Omega$ is the regularised sensitivity resulting from the Hilbertian extension-regularisation approach \cite{10.1016/bs.hna.2020.10.004_978-0-444-64305-6_2021}. This approach projects a directional derivative $\mathrm{d}J(\phi;w)$ of a functional $J(\phi)$, computed using the approaches in Section \ref{sec: Shape derivatives for unfitted methods} and \ref{sec: Automatic differentiation}, onto a Hilbert space $H$ on $D$, typically with additional regularity. This involves solving an identification problem \cite{10.1016/bs.hna.2020.10.004_978-0-444-64305-6_2021}:
\begin{weak}
For Find $g_\Omega\in H$ such that 
\begin{equation}\label{eqn: hilb extension wf}
\langle g_\Omega,w\rangle_H=\mathrm{d}J(\phi;w),~\forall w\in H,
\end{equation}
where $\langle\cdot,\cdot\rangle_H$ is the inner product on $H$.
\end{weak}
In addition to naturally extending the sensitivity from $\partial\Omega$ onto the bounding domain $D$ with $H$-regularity, this approach also ensures that the solution $g_\Omega$ is a descent direction for $J(\Omega)$.

The transport equation in Equation \eqref{eqn: transport} is solved using an interior penalty approach and Crank-Nicolson for the discretisation in time \cite{10.1016/j.cma.2017.09.005_2018,Burman_Elfverson_Hansbo_Larson_Larsson_2017,Burman_Fernández_2009}. The weak formulation of this problem is:
\begin{weak}
For all $t\in(0,T)$, find $\phi\in W^h$ such that
\begin{equation}\label{eqn: evolution wf}
    \int_D\left[v\pderiv{\phi}{t}+v\boldsymbol{\beta}\cdot\boldsymbol{\nabla}\phi\right]~\mathrm{d}\boldsymbol{x}+\sum_{F\in\mathscr{S}_h}\int_Fc_eh_F^2\lvert\boldsymbol{n}_F\cdot\boldsymbol{\beta}\rvert\llbracket \boldsymbol{n}_F\cdot\boldsymbol{\nabla}\phi\rrbracket\llbracket\boldsymbol{n}_F\cdot\boldsymbol{\nabla}v\rrbracket~\mathrm{d}s=0,~\forall v\in W^h,
\end{equation}
where $\mathscr{S}_h$ is the set of interior mesh facets, $h_F$ is the average element diameters of the elements sharing a facet, $\boldsymbol{n}_F$ is the normal to the facet $F$, $\llbracket v\rrbracket = v^+-v^-$ is the jump in a function $v$ over the facet $F$, and $c_e$ is a stabilisation coefficient. 
\end{weak}

We include $\lvert\boldsymbol{n}_F\cdot\boldsymbol{\beta}\rvert$ in the interior penalty term as proposed by \citet{Burman_Fernández_2009}. In the context of topology optimisation, this ensures that non-designable regions with $\boldsymbol{\beta}=\boldsymbol{0}$ are not locally over-stabilised by the interior penalty term. The example in Figure \ref{fig:evolution_fig} shows the evolution of an interface towards a non-designable region highlighted in grey with (green) and without (red) the $\lvert\boldsymbol{n}_F\cdot\boldsymbol{\beta}\rvert$ term. This shows that when the velocity magnitude is not included the interface will encroach into the non-designable domain, and this will continue over the entire evolution process. Note that this term was omitted from the formulation in \cite{10.1016/j.cma.2017.09.005_2018,Burman_Elfverson_Hansbo_Larson_Larsson_2017}, likely because non-designable regions were not considered. 

In our numerical examples, we take the stabilisation coefficient to be $c_e=0.01$.

\begin{figure}[!t]
    \centering
    \begin{subfigure}{0.475\linewidth}
        \centering
        \includegraphics[width=\linewidth]{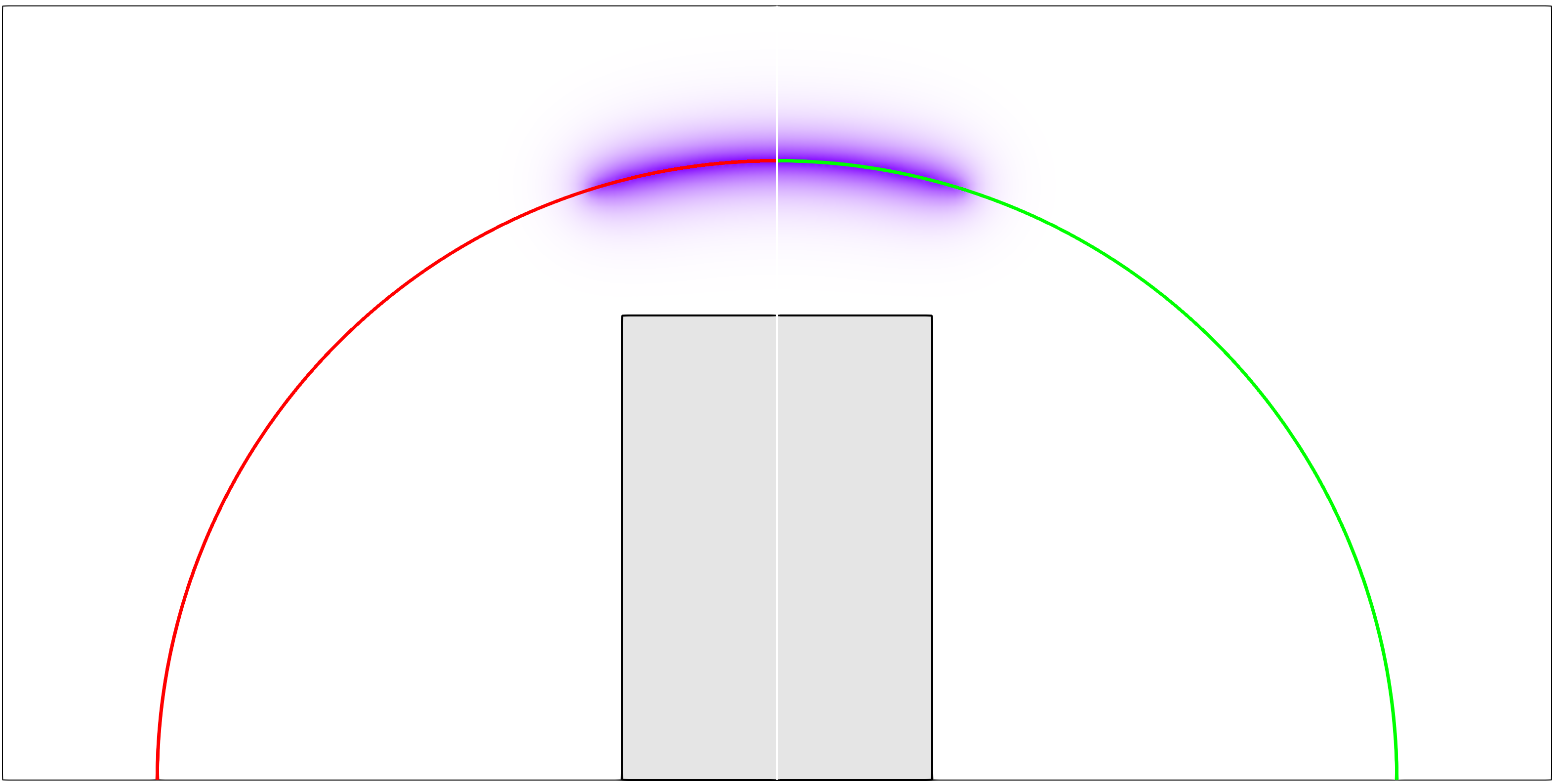}
        \caption{}
    \end{subfigure}
    \begin{subfigure}{0.475\linewidth}
        \centering
        \includegraphics[width=\linewidth]{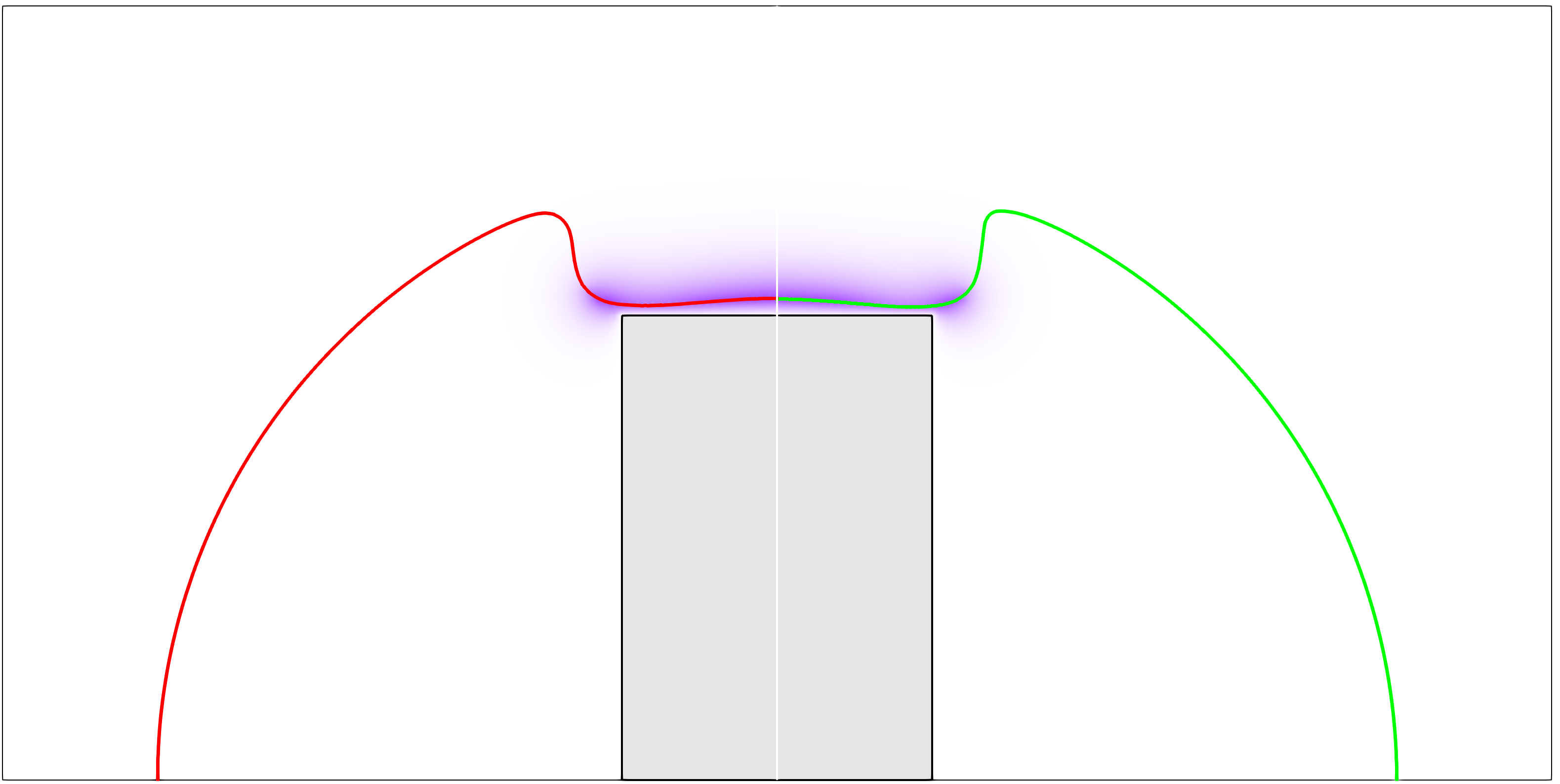}
        \caption{}
    \end{subfigure}
    \begin{subfigure}{0.475\linewidth}
        \centering
        \includegraphics[width=\linewidth]{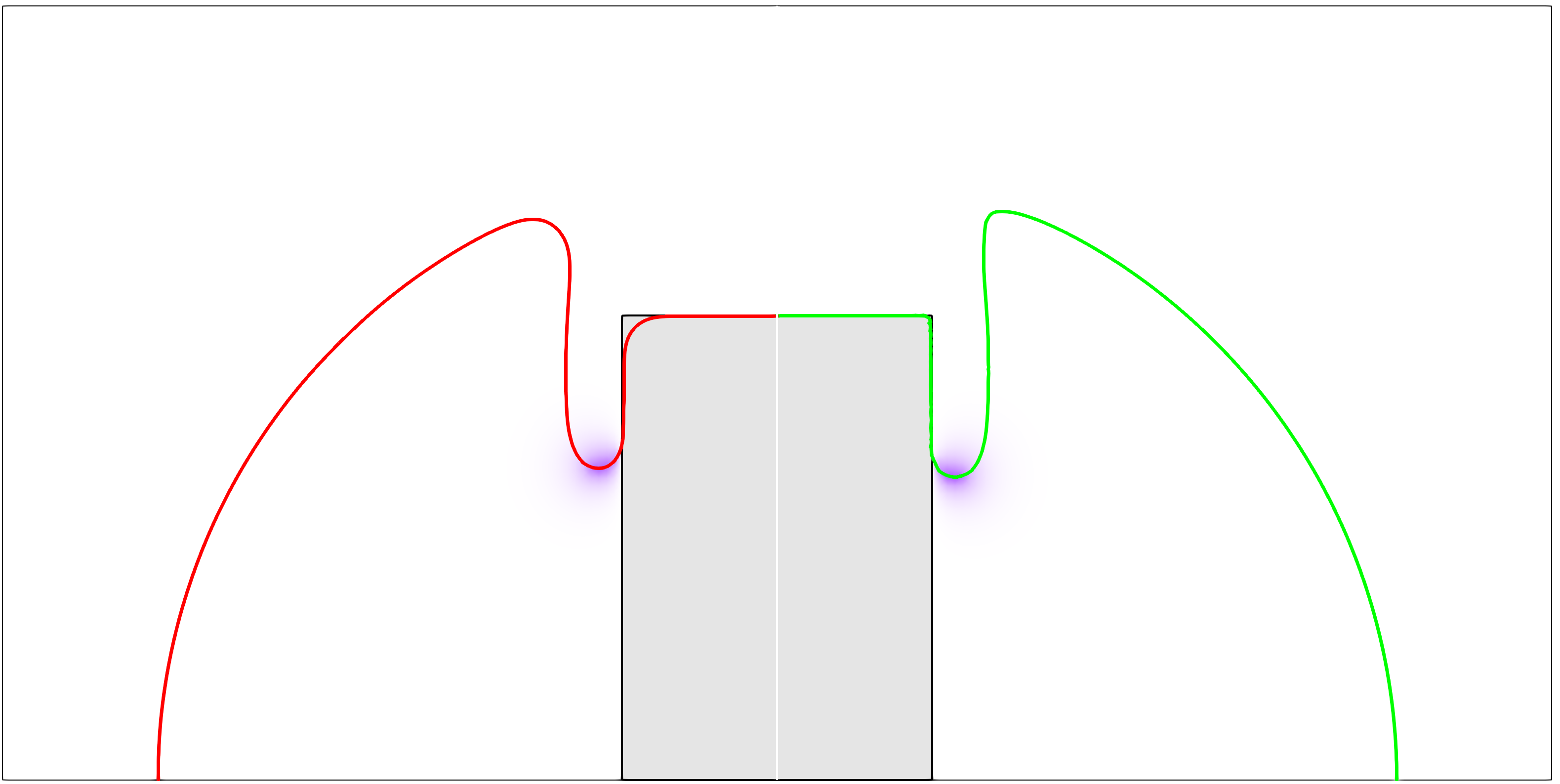}
        \caption{}
    \end{subfigure}
    \begin{subfigure}{0.475\linewidth}
        \centering
        \includegraphics[width=\linewidth]{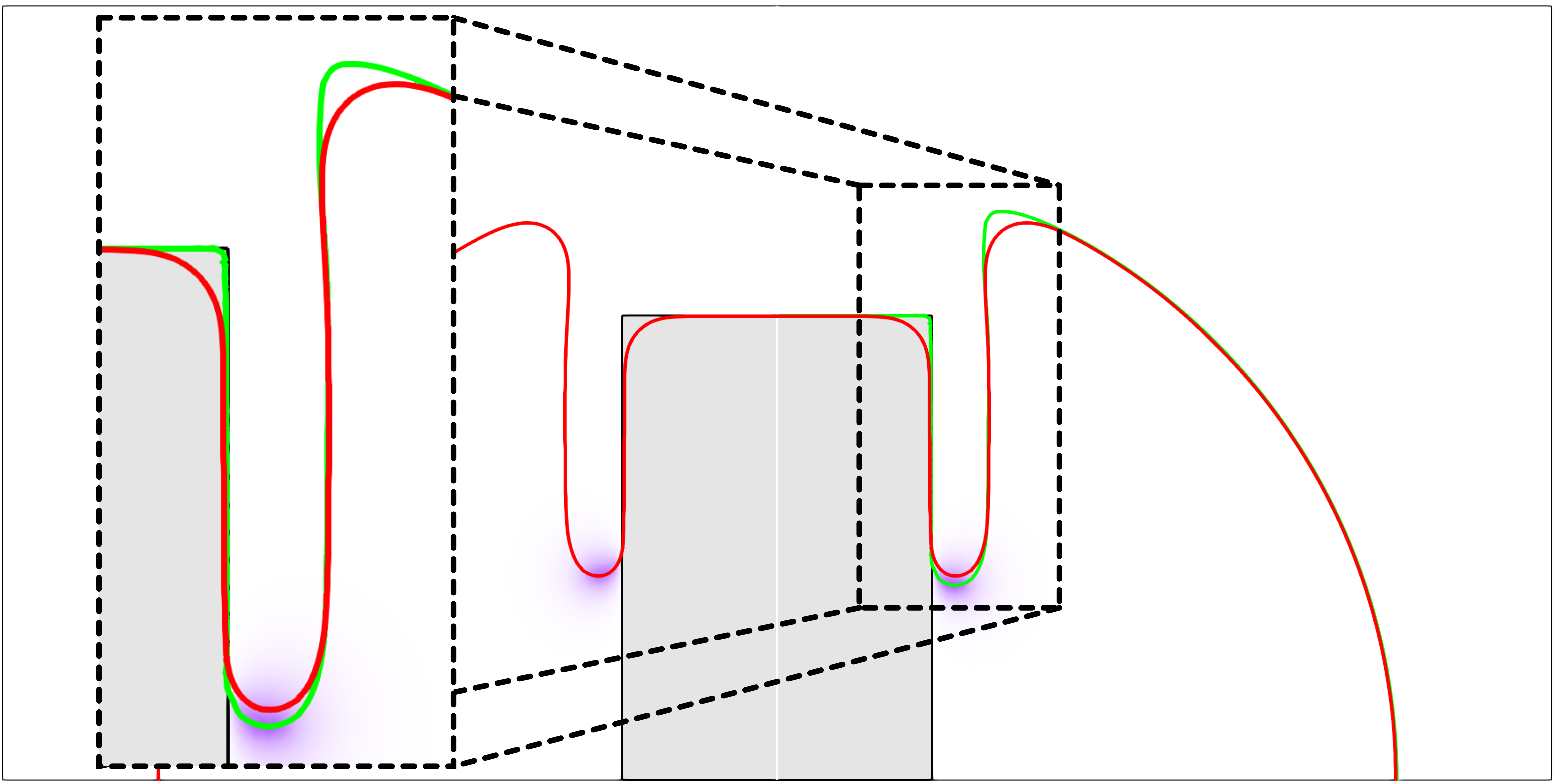}
        \caption{}
    \end{subfigure}
    \caption{The evolution of an interface visualised in Figure (a), (b), (c), and (d) towards a non-designable region highlighted in grey both with (green) and without (red) the $\lvert\boldsymbol{n}_F\cdot\boldsymbol{\beta}\rvert$ term in the evolution equation (Equation \ref{eqn: evolution wf}). In Figure (d) we overlay both curves on the right to show the slippage of the red interface into the non-designable region and the deviation from the green interface, which is particularly visible in the inset. We visualise the velocity field in purple that has been extended onto the background domain.}
    \label{fig:evolution_fig}
\end{figure}

\subsection{Reinitialisation}
Evolution of the boundary described above requires that the level-set function $\phi$ has well-behaved gradients. This is usually achieved by reinitialising $\phi$ to be an approximation of the signed distance function \citep{978-0-387-22746-7_2006} given by
\begin{equation}
\phi(\boldsymbol{x})=\begin{cases}
-d(\boldsymbol{x}, \partial \Omega)&\text{if }\boldsymbol{x} \in \Omega, \\
0&\text{if } \boldsymbol{x} \in \partial \Omega, \\
d(\boldsymbol{x}, \partial \Omega)&\text{if } \boldsymbol{x} \in D \backslash \bar{\Omega},
\end{cases}
\end{equation}
where $d(\boldsymbol{x}, \partial \Omega):=\min _{\boldsymbol{p} \in \partial \Omega}\lvert\boldsymbol{x}-\boldsymbol{p}\rvert$ is the minimum distance from $\boldsymbol{x}$ to the boundary $\partial \Omega$.

To achieve this, we solve the reinitialisation equation \cite{978-0-387-22746-7_2006,10.1006/jcph.1999.6345_1999} given by
\begin{align}\label{eqn: reinit}
\begin{cases}
    \pderiv{\phi}{t}(t,\boldsymbol{x}) + \operatorname{sign}(\phi_0(\boldsymbol{x}))\left(\lvert\boldsymbol{\nabla}\phi(t,\boldsymbol{x})\rvert-1\right) = 0,\\
    \phi(0,\boldsymbol{x})=\phi_0(\boldsymbol{x}),\\
    \boldsymbol{x}\in D, t>0.
\end{cases}
\end{align}
We solve this problem to steady state using an approach based on the one proposed by \citet{mallon2024neurallevelsettopology}. The weak formulation for this problem is given as: 
\begin{weak}
Find $\phi\in W^h$ such that
\begin{equation}
    \int_Dv\boldsymbol{w}\cdot\boldsymbol{\nabla}\phi-v\operatorname{sign}(\phi_0)~\mathrm{d}\boldsymbol{x}+\int_\Gamma \frac{\gamma_d}{h}\phi v~\mathrm{d}s+j(\phi,v)=0,~\forall v\in W^h,
\end{equation}
where $\boldsymbol{w}=\operatorname{sign}(\phi_0)\frac{\boldsymbol{\nabla}\phi}{\lVert\boldsymbol{\nabla}\phi\rVert}$, $\gamma_d$ is the surface penalty coefficient that we set to 20, and $h$ is the element size.
\end{weak}
In the above, the first term comes from multiplying Equation \eqref{eqn: reinit} by a test function and integrating over $D$, while the second term is a surface penalty term to reduce boundary movement under reinitialisation. The final term stabilises the finite element discretisation using an artificial viscosity given by
\begin{equation}\label{eqn: art visc reinit}
j(\phi,v)=\int_Dc_{r,1}h\Vert\boldsymbol{w}\rVert\boldsymbol{\nabla}\phi\cdot\boldsymbol{\nabla}v~\mathrm{d}\boldsymbol{x},
\end{equation}
where $h$ is the element diameter and $c_{r,1}$ is a stabilisation coefficient. In our examples, we take $c_{r,1}=0.5$. The above is then solved using Picard iterations, namely $\boldsymbol{w}$ is considered constant when computing the Jacobian in Newton-Raphson.

In our testing, we found that approximating the $\operatorname{sign}$ function with
\begin{equation}\label{eqn: approx sign}
    \operatorname{sign}(\phi)\approx\frac{\phi}{\sqrt{\phi^2+h^2\lvert\nabla\phi\rvert^2}},
\end{equation}
as proposed in \cite{10.1006/jcph.1999.6345_1999}, provides better convergence. In addition, we also considered replacing the artificial viscosity with a continuous interior penalty term, as this is more consistent with the strong formulation. This is given by
\begin{equation}\label{eqn: interior penalty reinit}
    j(\phi,v)=\sum_{F\in\mathscr{S}_h}\int_Fc_{r,2}h_F^2\llbracket\boldsymbol{\nabla}\phi\rrbracket\cdot\llbracket\boldsymbol{\nabla}v\rrbracket~\mathrm{d}s,
\end{equation}
where $c_{r,2}$ is a different stabilisation coefficient. We generally found that this provides better results than an artificial viscosity term, but requires that the initial level-set function $\phi_0$ is closer to a signed distance function for convergence. As such, for the results presented below, we solve the reinitialisation problem using the artificial viscosity term given by Equation \eqref{eqn: art visc reinit}. In future work, further exploration of these approaches will be considered.

\subsection{Isolated volume tagging}\label{sec: isolated vol tagging}
Isolated volumes of a domain $\Omega\subset D$ are regions of that domain within which solution fields are not sufficiently constrained (i.e., disconnected from all boundaries that have Dirichlet boundary conditions). In conventional level-set methods, where the void phase $D\setminus\Omega$ is filled with a weak ersatz material \cite[e.g.,][]{10.1016/j.jcp.2003.09.032_2004,10.1016/bs.hna.2020.10.004_978-0-444-64305-6_2021}, this situation is naturally remedied because the isolated volumes are connected by ersatz material and therefore fields are well constrained. Since we do not introduce such an approximation, the differential operator over disconnected volumes has a non-trivial kernel. As a result, the matrices resulting from a discretisation (e.g., finite element method) are singular. 

To remedy this situation, \citet{10.1016/j.cma.2017.03.007_2017} proposed constructing an indicator function by solving a thermal convection problem on $\Omega$ with heat dissipating towards a value of zero in regions connected to Dirichlet boundaries and accumulating to a desired value in isolated regions. A smoothed Heaviside function was then used to construct a binary indicator. However, in our tests, we found that this method fails to distinguish between volumes that are separated by a single cell layer. In these cases, the degrees of freedom within the two volumes are ``connected" through the ghost penalty term required by the CutFEM formulation. 

In this work, we propose a novel graph-based approach to detect isolated volumes as described below. This method relies on the creation of a conforming connectivity graph for the cut mesh, followed by a graph-colouring algorithm. We then mark all regions that are connected to the relevant Dirichlet boundaries. The negation of this map then identifies the isolated volumes. 

Once isolated volumes are identified, we follow the approach in \cite{10.1016/j.cma.2017.03.007_2017}. Namely, for a generic weak formulation $a(u,v)=l(v)$ over $\Omega$ with a non-trivial kernel owing to disconnected volumes marked by $\psi$, we add a penalty term
\begin{equation}
    k(u,v)=\int_{\Omega}\psi uv~\mathrm{d}\boldsymbol{x}.
\end{equation}
This constrains the average of $u$ to be zero in the regions marked by $\psi$. Concrete examples of this will be given in Section \ref{sec: Examples}.

\subsubsection{Generating the graph}

The first step is to construct a connectivity graph for the cut mesh, where each vertex corresponds to a cell within the mesh. The graph must be conforming, in the sense that two vertices are connected if and only if their associated cells share a face. 

Conforming meshes are generally not produced by typical cutting algorithms such as Delaunay or marching cubes, where cells are split locally without ensuring global mesh conformity. In two-dimensional settings, the meshes produced will always be conforming because two cells that share a cut edge will always split that edge in the same way. However, in three dimensions two cells sharing a cut face might generate non-matching sub-faces on the shared boundary (see Figure \ref{fig:delauney_nonconforming_split}). Generating a conforming mesh composed solely of tetrahedra and hexahedra for a complex cut geometry, particularly in higher dimensions, is generally challenging. This challenge motivates the development and use of embedded boundary methods, which do not require a conforming mesh. Therefore, the sub-triangulated cut mesh typically used for numerical integration cannot be directly employed to generate this conforming connectivity graph. 

Instead, we generate an auxiliary cut mesh, where cut cells are split into general polytopes or polyhedra. The resulting mesh is conforming, regardless of dimension. To create polytopal cuts, we adopt the framework proposed in \cite{badia_martorell_2022} and adapt the algorithms and their implementation in STLCutters.jl \cite{Martorell2025} for the case where the nodal values of the level set functions determine cuts. The result is the required conforming graph $G$, see Figure \ref{fig:polytopal_circle}. 

\begin{figure}
    \centering
    \begin{subfigure}{0.3\linewidth}
        \centering
        \includegraphics[height=0.9\linewidth]{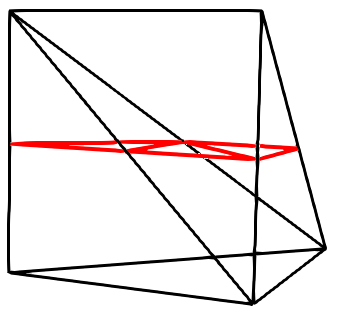}
        \caption{}
    \end{subfigure}
    \hspace{4em}
    \begin{subfigure}{0.3\linewidth}
        \centering
        \includegraphics[height=0.9\linewidth]{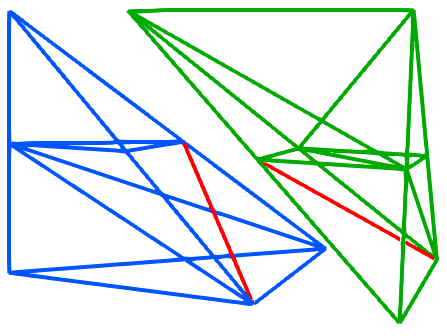}
        \caption{}
    \end{subfigure}
    \caption{Non-conforming cut meshes generated by Delaunay triangulation in 3D. Figure (a) shows the original mesh in black, given by two tetrahedra. The mesh is split through a plane, given by the red lines. Figure (b) shows the result of applying a Delaunay triangulation to both tetrahedra. The shared quadrilateral face is triangulated differently in each tetrahedra, creating a non-conforming edge (in red).}
    \label{fig:delauney_nonconforming_split}
\end{figure}

\begin{figure}
    \centering
    \includegraphics[width=0.5\linewidth]{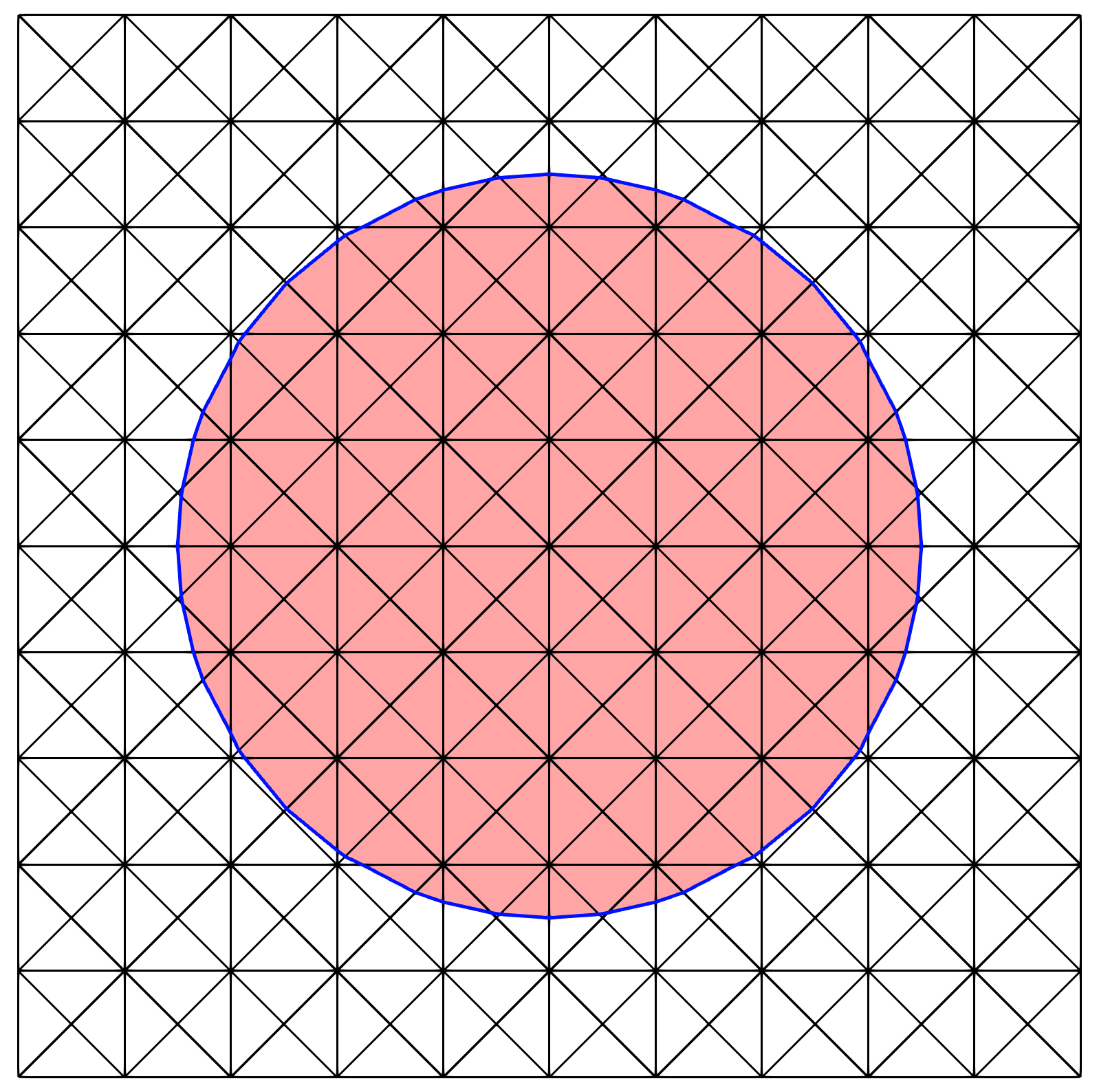}
    \caption{Polytopal cut mesh for the example in Figure \ref{fig:fig 3}. Note that, compared to the subtriangulated cut mesh, this mesh deals with cut cells using polytopes with an arbitrary number of faces.}
    \label{fig:polytopal_circle}
\end{figure}

\subsubsection{Colouring the graph}

The next step involves creating a colouring for the graph $G$, where each colour corresponds to a distinct disconnected volume. This is done by successfully applying a Breadth-First-Search (BFS) algorithm to colour connected parts of the graph that share the same IN/OUT state. Algorithms \ref{alg:color_graph} and \ref{alg:color_volume} detail the procedures used to colour the whole graph and each individual volume, respectively.

In the context of distributed meshes, each processor $p$ only stores a local portion of the graph, denoted $G_p$. Colouring each local graph partition $G_p$ locally (independently) is insufficient to reliably identify the global isolated volumes. Specifically, a single volume might span multiple processor domains but only be adjacent to constraining boundaries (e.g., Dirichlet boundaries) within a subset of these domains. In the worst-case scenario, a volume could theoretically span all processor domains within the communicator, while the relevant constraining boundaries reside in only a single processor's domain (see Figure \ref{fig:polytopal_snake}). This information has to be propagated to all processors that share the volume, which is inevitably a global operation. 

A global volume colouring algorithm is therefore required. This is done in two steps: First, we use Algorithm \ref{alg:color_graph} in each processor $p$ to create a local colouring of $G_p$. For each local volume, we then locally collect a) all neighbouring processors sharing this volume and b) the local colour ID of the volume within the neighbouring processor. This can be done through (relatively) inexpensive neighbour-to-neighbour communications, typical within an FEM context. Finally, this local information is gathered within a single processor that uses it to create a global numbering. The global colouring is then scattered back to all processors. 

\begin{figure}
    \centering
    \begin{subfigure}{0.4\textwidth}
        \centering
        \includegraphics[width=\linewidth]{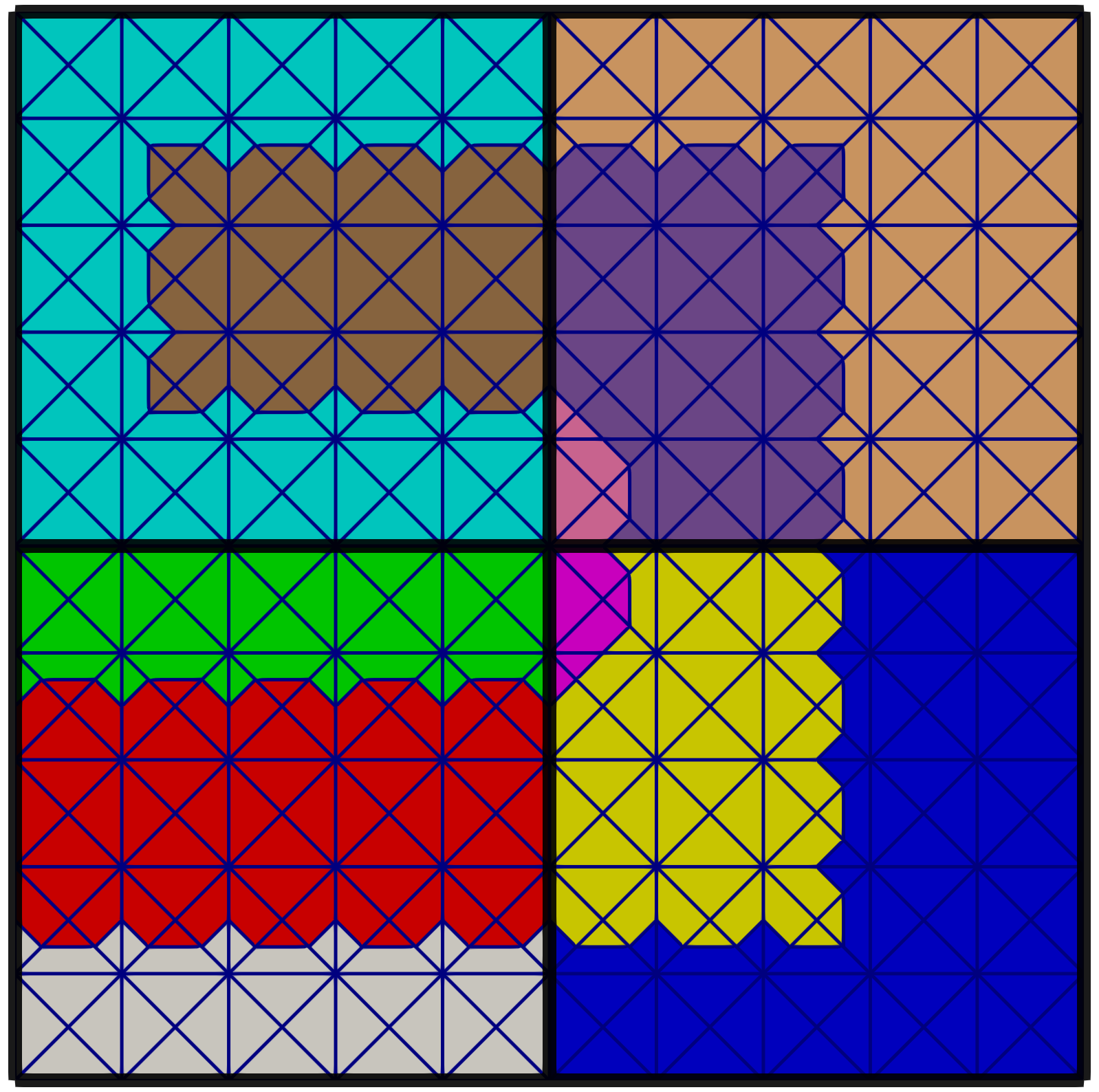}
        \caption{}
    \end{subfigure}
    \begin{subfigure}{0.4\textwidth}
        \centering
        \includegraphics[width=\linewidth]{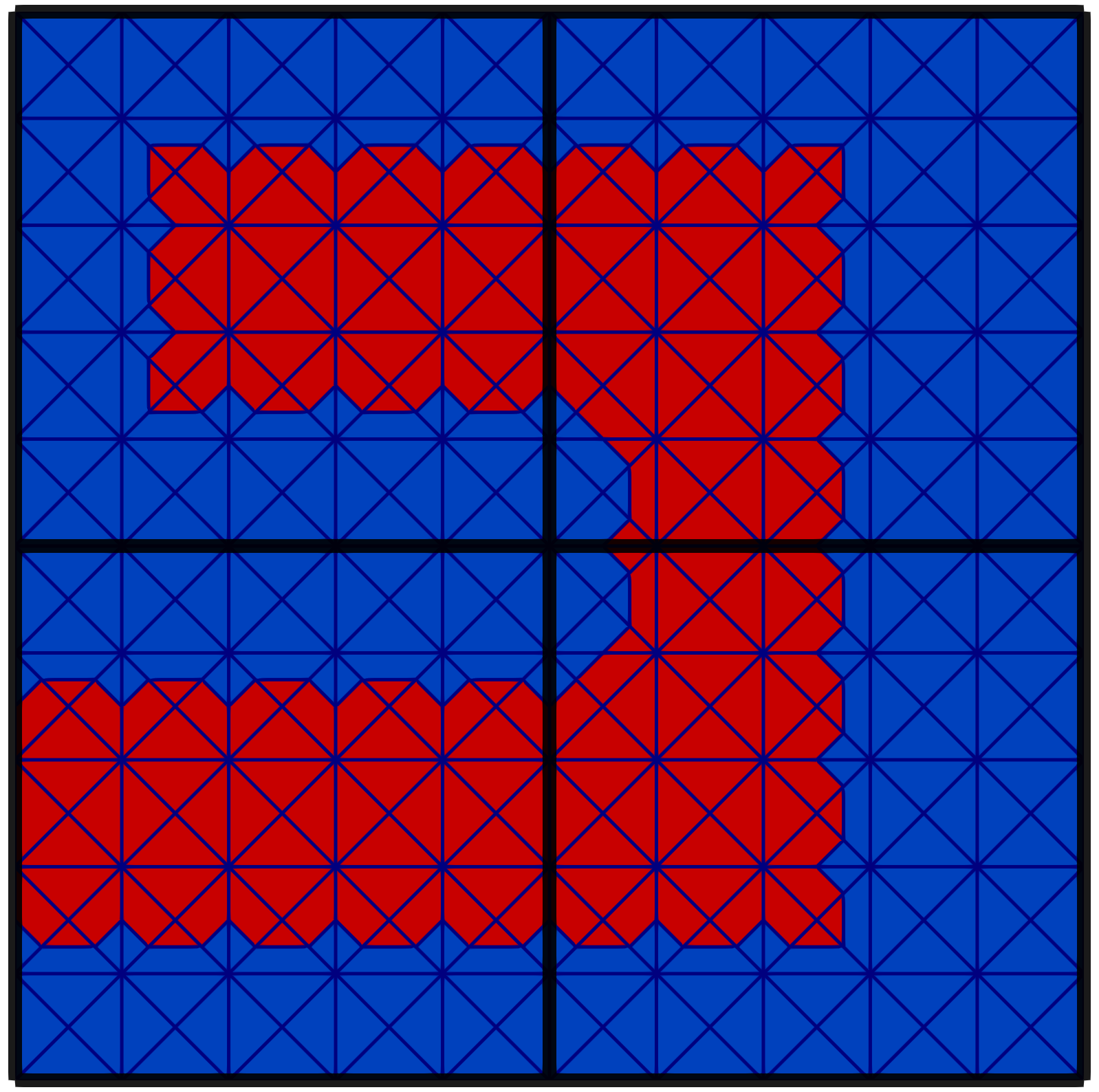}
        \caption{}
    \end{subfigure}
    \caption{Colouring of a ``snake" level set within a distributed background mesh. The background mesh is split between four processors. Processor boundaries are shown with black lines. The level set creates a single IN volume that loops around the whole background mesh while only being constrained by boundary conditions touching the boundary of the background mesh. Figure (a) shows the local colourings produced by each processor independently. Figure (b) shows the global colouring.}
    \label{fig:polytopal_snake}
\end{figure}

\begin{algorithm}
\caption{\textsc{ColourGraph}(G,S)}
\label{alg:color_graph}
\begin{algorithmic}[1]
    \Input
        \State $G$: Adjacency graph: $neighbours(v) = G[v]$
        \State $S$: Array of vertex IN/OUT states: $state(v) = S[v]$
    \Output
        \State $C$: Final colour assignment for each cell, where each colour identifies a connected component within a group.
        \State $S_c$: Mapping from final colour index to its state.\\

    \State Initialise $C$ as an empty vector of size $\text{num\_vertices}(G)$
    \State Initialise $S_c$ as an empty list
    \State $nc \leftarrow 0$
    \For{$v$ in $G$}
        \If{$v$ not coloured} \Comment{New volume found}
            \State $nc \leftarrow nc + 1$
            \State Append!$(S_c, state(v))$
            \State Call \textsc{ColourVolume}$(v, nc, G, S, C)$ \Comment{Propagate colour through volume.}
        \EndIf
    \EndFor
    \State \Return $C, S_c$
\end{algorithmic}
\end{algorithm}

\begin{algorithm}
\caption{\textsc{colourVolume}($v_0, c, G, S, C$)}
\label{alg:color_volume}
\begin{algorithmic}[1]
    \Input
        \State $v_0$: Starting vertex
        \State $c$: Colour value to assign to this volume
        \State $G$: Adjacency graph: $neighbours(v) = G[v]$
        \State $S$: Array of vertex IN/OUT states: $state(v) = S[v]$
        \State $C$: Array to vertex colours (output)
    \Output 
    \State Modifies $C$ in place, stops when all connected vertices to $v_0$ that share the same state have been coloured.\\
    \State Initialise empty queue $Q$
    \State Enqueue $v_0$ into $Q$
    \While{$Q$ is not empty}
        \State $v \leftarrow \text{Dequeue}(Q)$
        \State $C[v] \leftarrow c$ \Comment{Colour the vertex}
        \ForAll{$u$ in $neighbours(v)$} \Comment{Enqueue neighbours}
            \If{$u$ not coloured and $state(v) == state(u)$}
                \State Enqueue $u$ into $Q$
            \EndIf
        \EndFor
    \EndWhile
\end{algorithmic}
\end{algorithm}

\begin{algorithm}
\caption{\textsc{ColourDistributedGraph}$(G,S_p)$}
\label{alg:color_distributed_graph}
\begin{algorithmic}[1]
    \Input
        \State $G = \{G_p,g_G\}$: Global adjacency graph
        \State $S_p$: Local arrays of vertex IN/OUT states
    \Output
        \State $C = \{C_p,g_C\}$: Global colouring for each cell
        \State $S_c$: Mapping from each local colour to its state.\\

    \State $C_p, S_c \leftarrow  \textsc{ColourGraph}(G_p,S_p)$
    \State $C \leftarrow \textsc{GlobalColouring}(G,C_p)$
    \State \Return $C, S_c$
\end{algorithmic}
\end{algorithm}

\section{Examples}\label{sec: Examples}

In this section, we consider two examples that make use of the automatic differentiation techniques and the isolated volume algorithm discussed above. The first example is topology optimisation of a three-dimensional linear elastic wheel and the second is topology optimisation of the elastic part in a fluid-structure interaction problem.

\subsection{Elastic wheel}\label{sec: elastic wheel example}

\subsubsection{Formulation}
In this section, we consider the minimum compliance optimisation of a three-dimensional linear elastic wheel. 

The formulation of this problem is similar to the two-dimensional wheel considered in \cite{10.1016/j.cma.2017.09.005_2018}. Figure \ref{fig:wheel background} shows the background domain $D$ and the boundary conditions for this problem. We apply a homogeneous Dirichlet boundary condition to the outer shell and a non-homogeneous Neumann boundary condition $\boldsymbol{g}(x,y,z)=100(-y,x,0)$ to the inner shell.
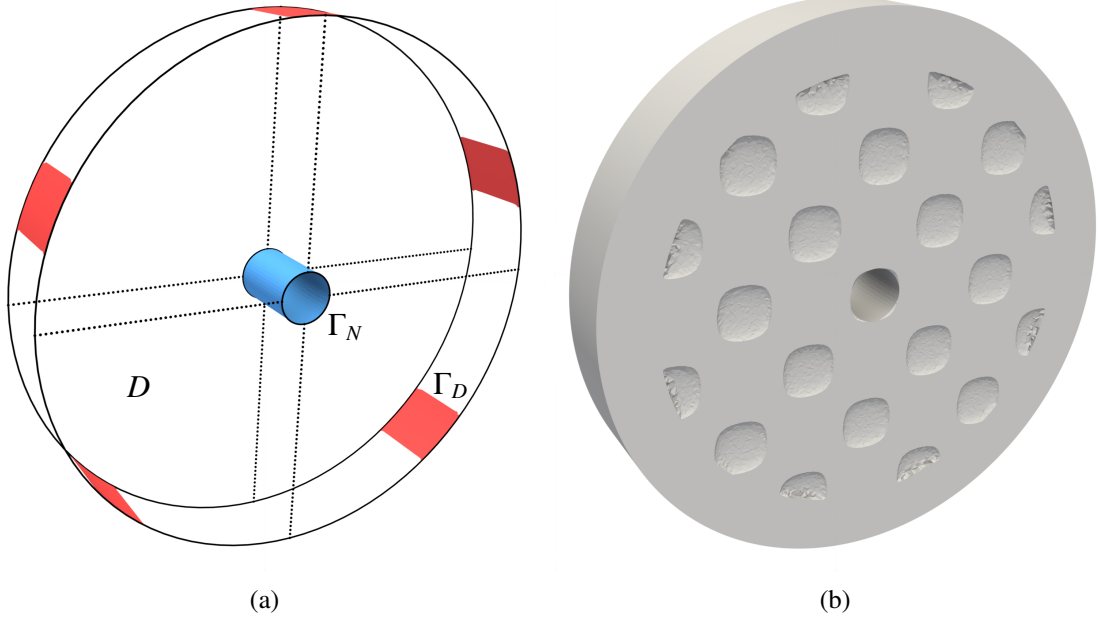
\begin{figure}[t]
    \centering
    \begin{subfigure}{0.45\textwidth}
        \centering
        \def\svgwidth{1\textwidth}
        {
\begingroup%
  \makeatletter%
  \providecommand\color[2][]{%
    \errmessage{(Inkscape) Color is used for the text in Inkscape, but the package 'color.sty' is not loaded}%
    \renewcommand\color[2][]{}%
  }%
  \providecommand\transparent[1]{%
    \errmessage{(Inkscape) Transparency is used (non-zero) for the text in Inkscape, but the package 'transparent.sty' is not loaded}%
    \renewcommand\transparent[1]{}%
  }%
  \providecommand\rotatebox[2]{#2}%
  \newcommand*\fsize{\dimexpr\f@size pt\relax}%
  \newcommand*\lineheight[1]{\fontsize{\fsize}{#1\fsize}\selectfont}%
  \ifx\svgwidth\undefined%
    \setlength{\unitlength}{1155.00002884bp}%
    \ifx\svgscale\undefined%
      \relax%
    \else%
      \setlength{\unitlength}{\unitlength * \real{\svgscale}}%
    \fi%
  \else%
    \setlength{\unitlength}{\svgwidth}%
  \fi%
  \global\let\svgwidth\undefined%
  \global\let\svgscale\undefined%
  \makeatother%
  \begin{picture}(1,1.03116879)%
    \lineheight{1}%
    \setlength\tabcolsep{0pt}%
    \put(0,0){\includegraphics[width=\unitlength,page=1]{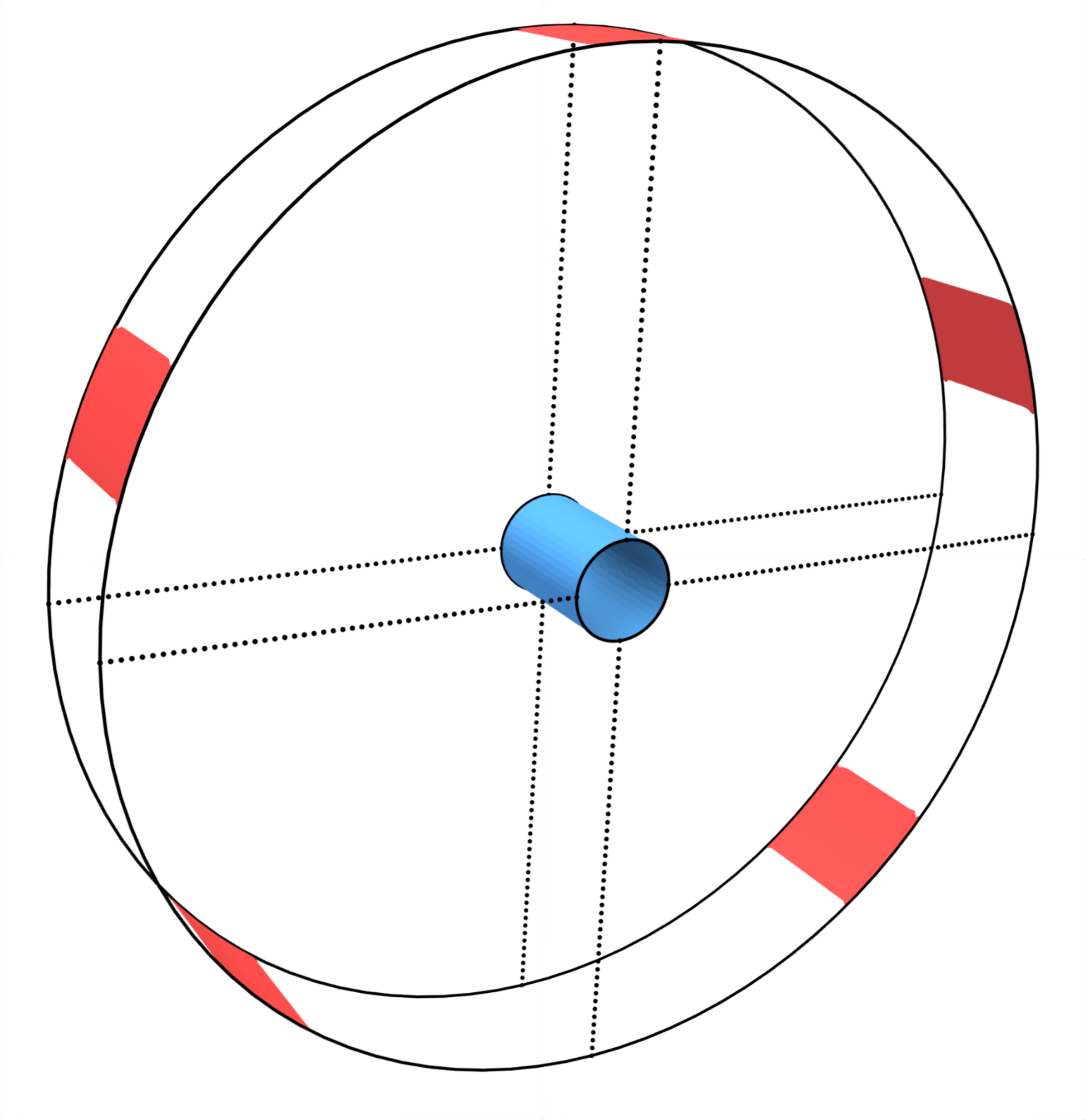}}%
    \put(0.6127094,0.42060323){\color[rgb]{0,0,0}\makebox(0,0)[lt]{\lineheight{1.25}\smash{\begin{tabular}[t]{l}$\Gamma_N$\end{tabular}}}}%
    \put(0.25413314,0.3104795){\color[rgb]{0,0,0}\makebox(0,0)[lt]{\lineheight{1.25}\smash{\begin{tabular}[t]{l}$D$\end{tabular}}}}%
    \put(0.8,0.31619346){\color[rgb]{0,0,0}\makebox(0,0)[lt]{\lineheight{1.25}\smash{\begin{tabular}[t]{l}$\Gamma_D$\end{tabular}}}}%
  \end{picture}%
\endgroup%
}
        \caption{}
        \label{fig:wheel background}
    \end{subfigure}
    \begin{subfigure}{0.45\textwidth}
        \centering
        \includegraphics[width=1\linewidth]{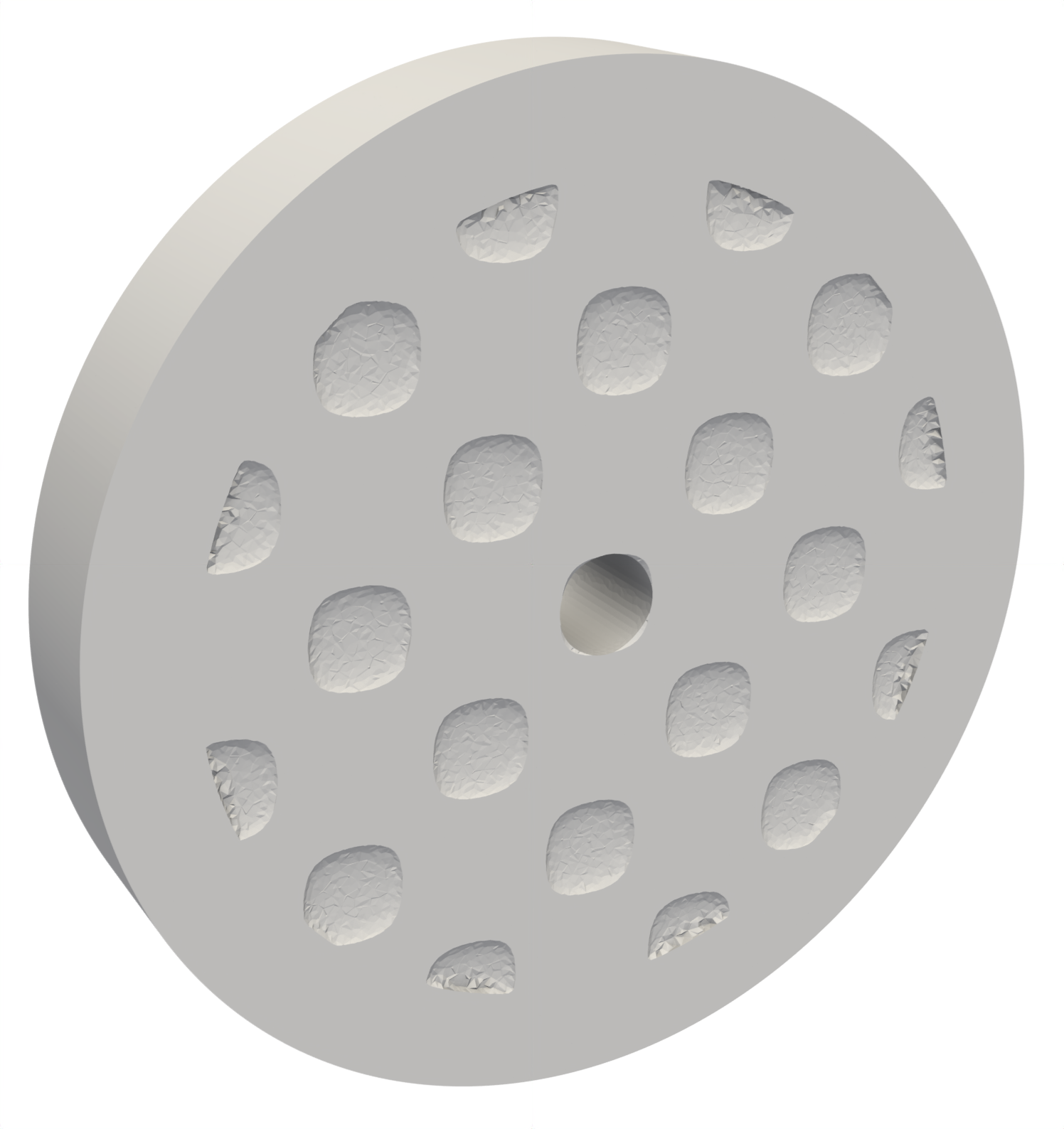}
        \caption{}
        \label{fig:wheel initial}
    \end{subfigure}
    \caption{A visualisation of the setup for the minimum compliance optimisation of a three-dimensional linear elastic wheel. Figure (a) shows the background domain $D$ and boundary conditions. The red boundary denotes a homogeneous Dirichlet boundary condition on $\Gamma_D$, while the blue boundary denotes a non-homogeneous Neumann boundary condition $\boldsymbol{g}(x,y,z)=100(-y,x,0)$ on $\Gamma_N$. Figure (b) shows the initial domain $\Omega(\phi_0)$. The inner and outer cylinders have a radii of 0.1 and 1, respectively, and the thickness of the wheel is 0.3.}
    \label{fig:wheel setup}
\end{figure}
The strong formulation of this problem is given by
\begin{equation}\label{eqn: strong form elast}
    \begin{cases}
        -\boldsymbol{\nabla}\cdot\boldsymbol{\sigma}(\boldsymbol{d})=\boldsymbol{0}&\text{in }\Omega(\phi),\\
        \boldsymbol{\sigma}(\boldsymbol{d})\cdot\boldsymbol{n}=\boldsymbol{0}&\text{on }\Gamma(\phi),\\
        \boldsymbol{\sigma}(\boldsymbol{d})\cdot\boldsymbol{n}=\boldsymbol{g}&\text{on }\Gamma_N,\\
        \boldsymbol{d}=\boldsymbol{0}&\text{on }\Gamma_D,\\
    \end{cases}
\end{equation}
where $\boldsymbol{\sigma}(\boldsymbol{d})=\lambda\operatorname{tr}(\boldsymbol{\varepsilon}(\boldsymbol{d}))\boldsymbol{I}+2\mu\boldsymbol{\varepsilon}(\boldsymbol{d})$ is the stress tensor, $\boldsymbol{\varepsilon}(\boldsymbol{d})=\frac{1}{2}\left(\boldsymbol{\nabla d}+(\boldsymbol{\nabla d})^\intercal\right)$ is the strain tensor, and $\lambda=0.5769$ and $\mu=0.3846$ are the Lam\'e constants of the material. To avoid ill-conditioning of the system resulting from the finite element discretisation of Equation \eqref{eqn: strong form elast}, we use CutFEM as described in \cite{10.1016/j.cma.2017.09.005_2018} and references therein. In particular, the weak formulation of this problem is: 
\begin{weak}
For $\Omega(\phi)\subset D$, find $\boldsymbol{d}\in R$ such that
\begin{equation}\label{eqn: elast weak form}
A(\boldsymbol{d},\boldsymbol{s})=a(\boldsymbol{d},\boldsymbol{s})+k_\psi(\boldsymbol{d},\boldsymbol{s})+j(\boldsymbol{d},\boldsymbol{s})= l(\boldsymbol{s}),~\forall \boldsymbol{s}\in T,
\end{equation}
with
\begin{align}
&a(\boldsymbol{d},\boldsymbol{s})=\int_{\Omega(\phi)}\boldsymbol{\sigma(d)}:\boldsymbol{\varepsilon(s)}~\mathrm{d}\boldsymbol{x},\label{eqn: elast a}\\
&l(\boldsymbol{s})=\int_{\Gamma_N}\boldsymbol{s}\cdot\boldsymbol{g}~\mathrm{d}s,\\
&k_\psi(\boldsymbol{d},\boldsymbol{s})=\int_{\Omega(\phi)}\psi\boldsymbol{d}\cdot\boldsymbol{s}~\mathrm{d}\boldsymbol{x},\label{eqn: elast k}\\
&j(\boldsymbol{d},\boldsymbol{s})=\gamma(\lambda+\mu)\sum_{F\in\mathscr{T}_G}\int_{F} h_F^3\llbracket\boldsymbol{n}_F\cdot\boldsymbol{\nabla d}\rrbracket\cdot\llbracket\boldsymbol{n}_F\cdot\boldsymbol{\nabla s}\rrbracket~\mathrm{d}s,\label{eqn: elast j}
\end{align}
where $R=\{\boldsymbol{d}\in [H^1(\Omega(\phi))]^3:\boldsymbol{d}\rvert_{\Gamma_D}=\boldsymbol{0},~\boldsymbol{d}\rvert_{\Gamma_N}=\boldsymbol{d}_0\}$ and $T=\{\boldsymbol{d}\in [H^1(\Omega(\phi))]^3:\boldsymbol{d}\rvert_{\Gamma_D}=\boldsymbol{0},~\boldsymbol{d}\rvert_{\Gamma_N}=\boldsymbol{0}\}$.
\end{weak}

In the above, $\mathscr{T}_G$ is known as the ghost skeleton and defined in \cite{10.1002/nme.4823_2015} as follows: for distinct elements $K_1\in\mathscr{S}_h$ and $K_2\in\mathscr{S}_h$, a facet $F\in\mathscr{T}_G$ is given by $F=K_1\cap K_2$ where at least one of $K_1$ or $K_2$ intersect the interface. In addition, the operation $:$ denotes double contraction, $a(\boldsymbol{d},\boldsymbol{s})$ is the standard bilinear form for elasticity, $k_\psi(\boldsymbol{d},\boldsymbol{s})$ enforces zero displacement within the isolated volume marked by $\psi$ (see Section \ref{sec: isolated vol tagging}), and $j(\boldsymbol{d},\boldsymbol{s})$ is the ghost penalty stabilisation term discussed in \cite{10.1016/j.cma.2017.09.005_2018} with stabilisation parameter $\gamma=10^{-7}$. Finally, we replace the test space $T$ and trial space $R$ with discrete spaces of continuous piecewise-linear vector fields denoted $T^h$ and $R^h$, respectively.

The optimisation problem is then
\begin{equation}\label{eqn: wheel opt prob}
    \begin{aligned}
    \min_{\phi\in W^h}&~J(\phi)\coloneqq\int_{\Omega(\phi)}\boldsymbol{\sigma(d)}:\boldsymbol{\varepsilon(d)}~\mathrm{d}\boldsymbol{x}\\
    \text{s.t. }&~C(\phi)=0,\\
    &~A(\boldsymbol{d},\boldsymbol{s})=l(\boldsymbol{s}),~\forall \boldsymbol{s}\in T^h,
    \end{aligned}
\end{equation}
where $C(\phi)\coloneqq\operatorname{Vol}(\Omega(\phi))-V_f\operatorname{Vol}(D)$ with $\operatorname{Vol}(\Omega(\phi))=\int_{\Omega(\phi)}1~\mathrm{d}\boldsymbol{x}$ and similarly for $\operatorname{Vol}(D)$, and $V_f=0.4$ is the required volume fraction.

The above is implemented using GridapTopOpt v0.2.0 \cite{GridapTopOpt}. Note that because Gridap and GridapTopOpt allow the user to specify arbitrary weak formulations and optimisation problems, other unfitted approaches, such as the finite cell method \cite{Parvizian_Duster_Rank_2012}, can also readily be used.

\subsubsection{Results}
We solve the optimisation problem in Equation \eqref{eqn: wheel opt prob} using augmented Lagrangian method \cite{978-0-387-30303-1_2006} as described by \citet{GridapTopOpt} with initial structure shown in Figure \ref{fig:wheel initial}. We use the automatic differentiation techniques discussed in Section \ref{sec: Automatic differentiation} along with C\'ea's method \cite{10.1051/m2an/1986200303711_1986} as provided by GridapTopOpt \citep{GridapTopOpt} to compute all derivatives. We generate the symplectic background triangulation $D$ using Gmsh \cite{10.1002/nme.2579} and load the resulting triangulation using GridapGmsh \cite{gridapgmsh}. The resulting mesh has roughly 1.32M elements and 219K nodes. We partition the background triangulation into 48 parts and solve the optimisation problem in an efficient memory-distributed framework using 48 processors. For a detailed review of the techniques used to improve computational efficiency we refer the reader to \citet{Badia2020}, \citet{Badia2022}, and \citet{GridapTopOpt}. To solve the Hilbertian extension-regularisation problem (Equation \ref{eqn: hilb extension wf}) we use an algebraic multigrid preconditioned conjugate gradient method as outlined in \cite{GridapTopOpt}. We use SuperLU\_DIST \cite{Demmel_Gilbert_Li_1999,Li_Demmel_2003,Li_Lin_Liu_Sao_2023} to solve the state equation for elasticity, the transport equation, and the reinitialisation equation. In the future, we will investigate iterative methods for solving these problems at larger scales. 

\begin{figure}[p]
    \centering
    \begin{subfigure}{0.32\textwidth}
        \centering
        \includegraphics[width=\textwidth]{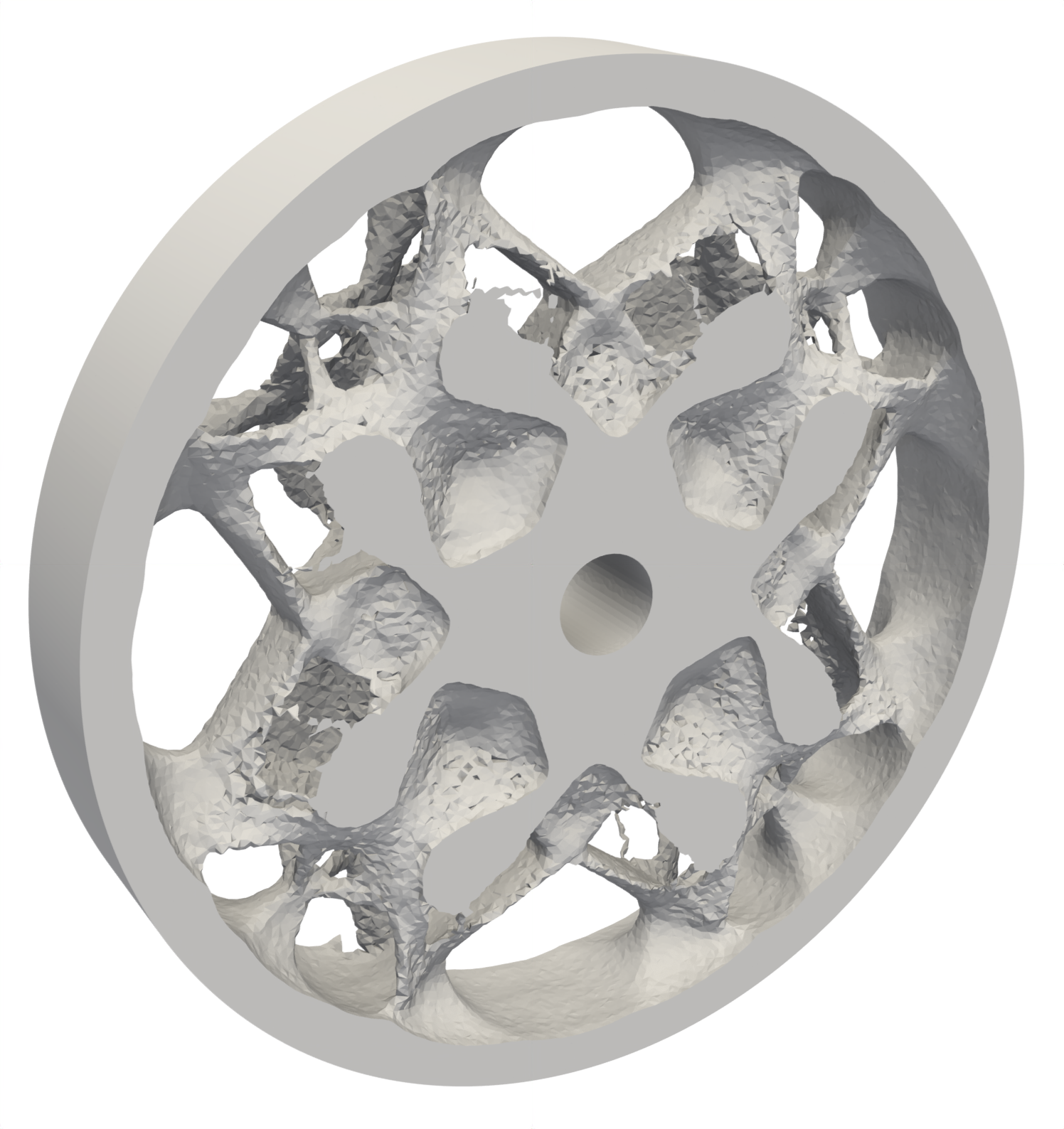}
        \caption{Iter. 50}
    \end{subfigure}
    \begin{subfigure}{0.32\textwidth}
        \centering
        \includegraphics[width=\textwidth]{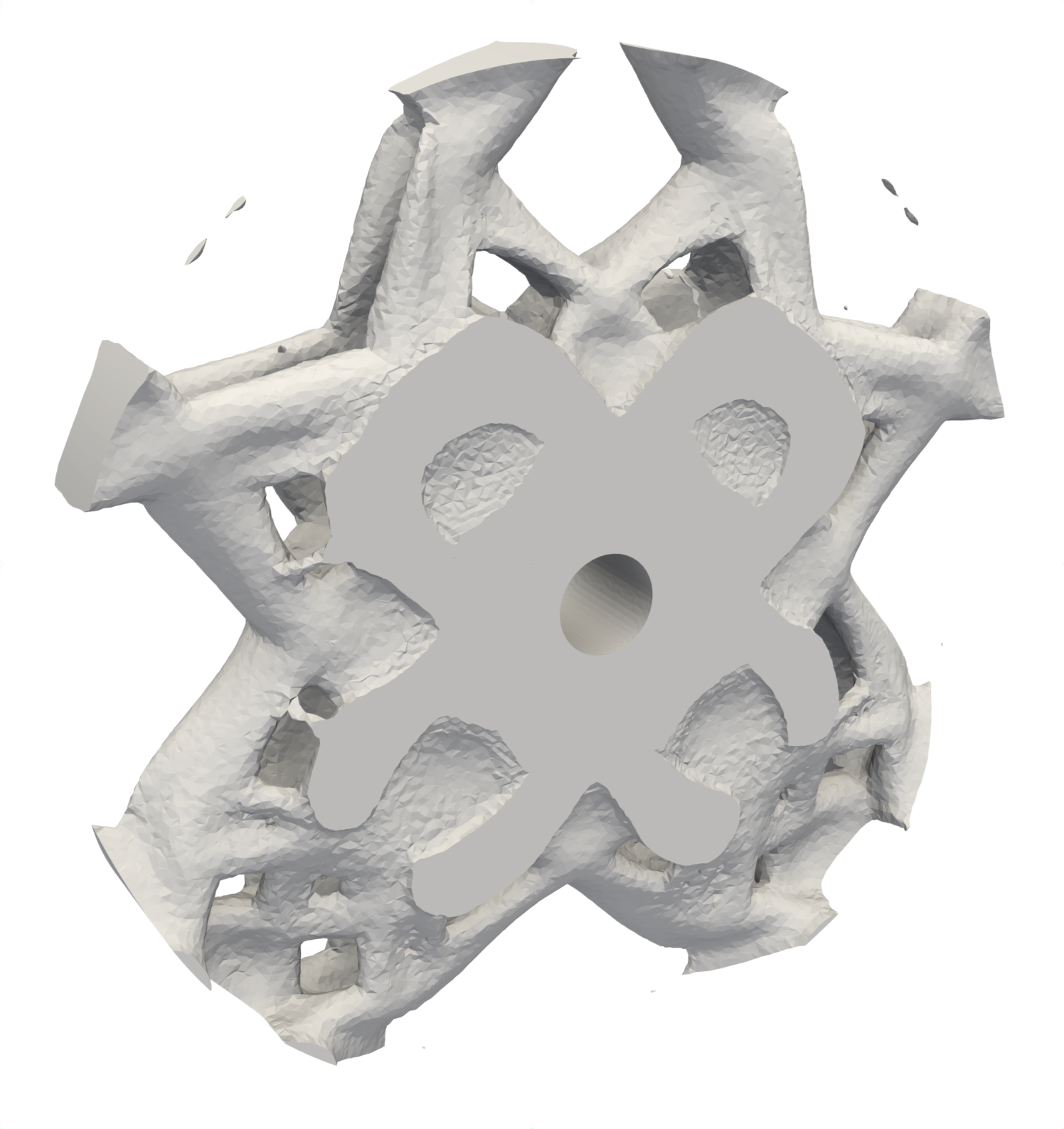}
        \caption{Iter. 100}
    \end{subfigure}\\
    \begin{subfigure}{0.32\textwidth}
        \centering
        \includegraphics[width=\textwidth]{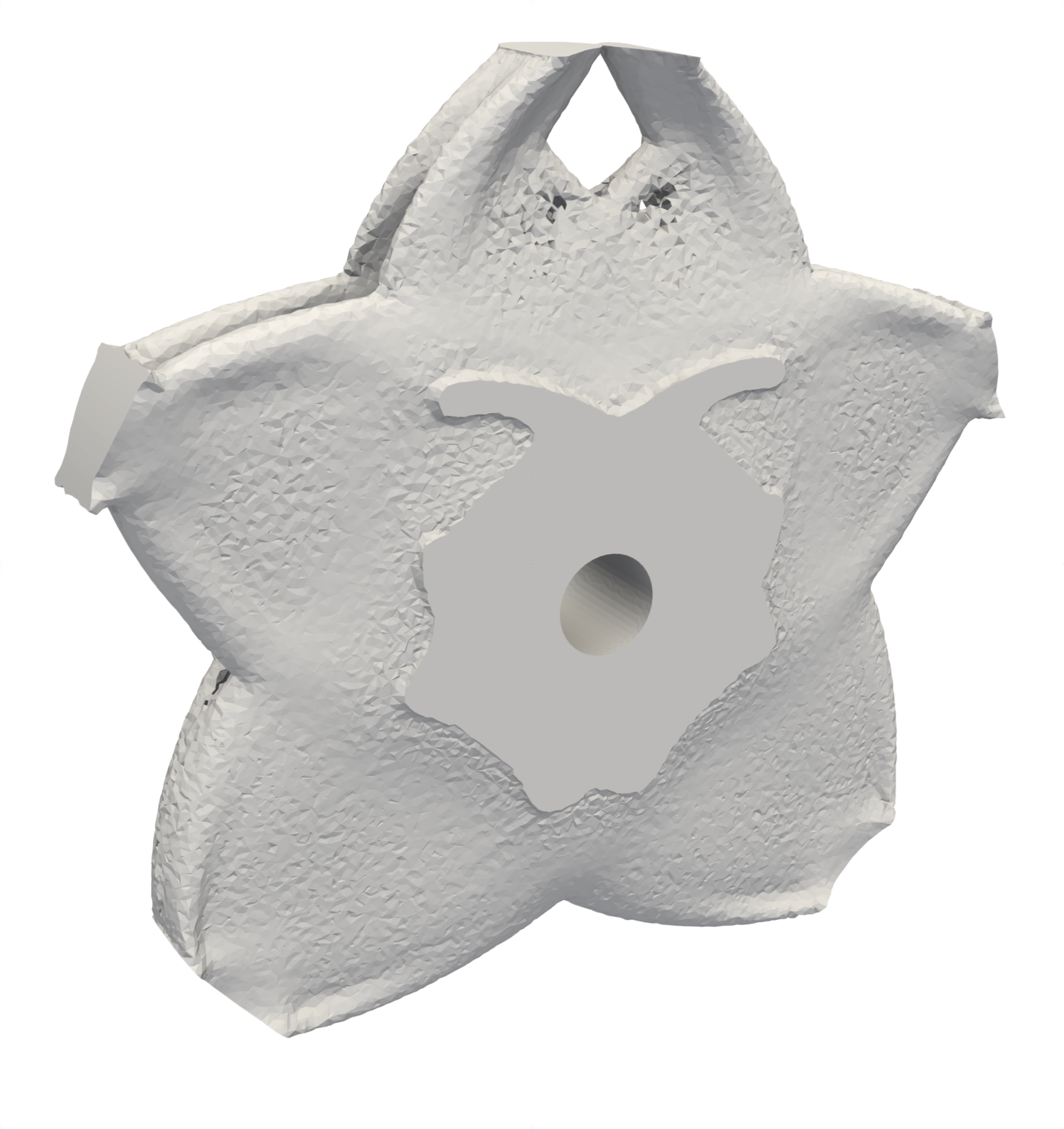}
        \caption{Iter. 150}
    \end{subfigure}
    \begin{subfigure}{0.32\textwidth}
        \centering
        \includegraphics[width=\textwidth]{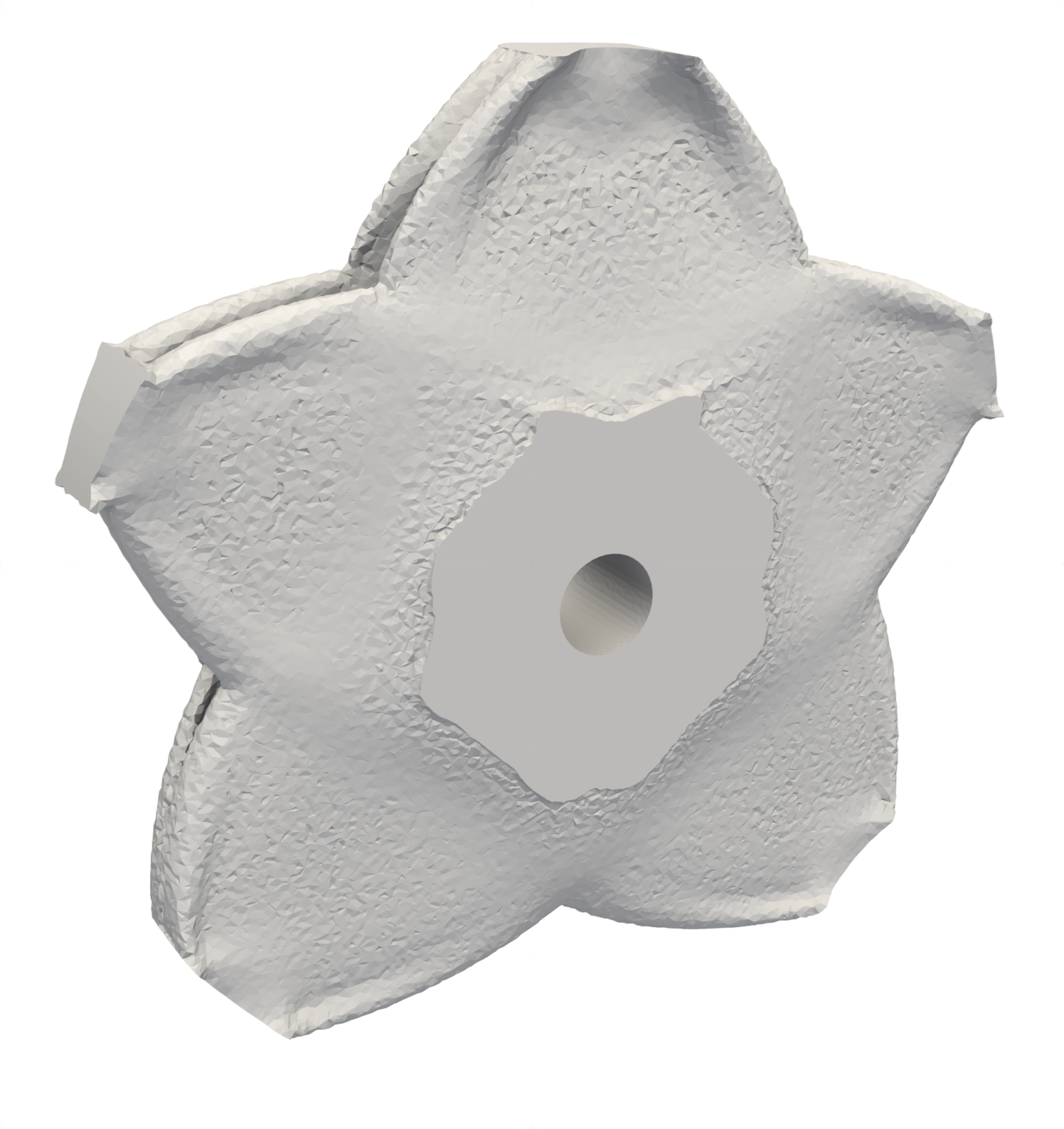}
        \caption{Iter. 240}
    \end{subfigure}
    \caption{Visualisations of the intermediate structures (Fig. a -- Fig. c) and the final optimised structure (Fig. d) for the elastic wheel problem. We outline the background domain using the black lines.}
    \label{fig: wheel results}
\end{figure}
\begin{figure}[p]
    \centering
    \begin{subfigure}{0.42\textwidth}
        \centering
        \includegraphics[width=\textwidth]{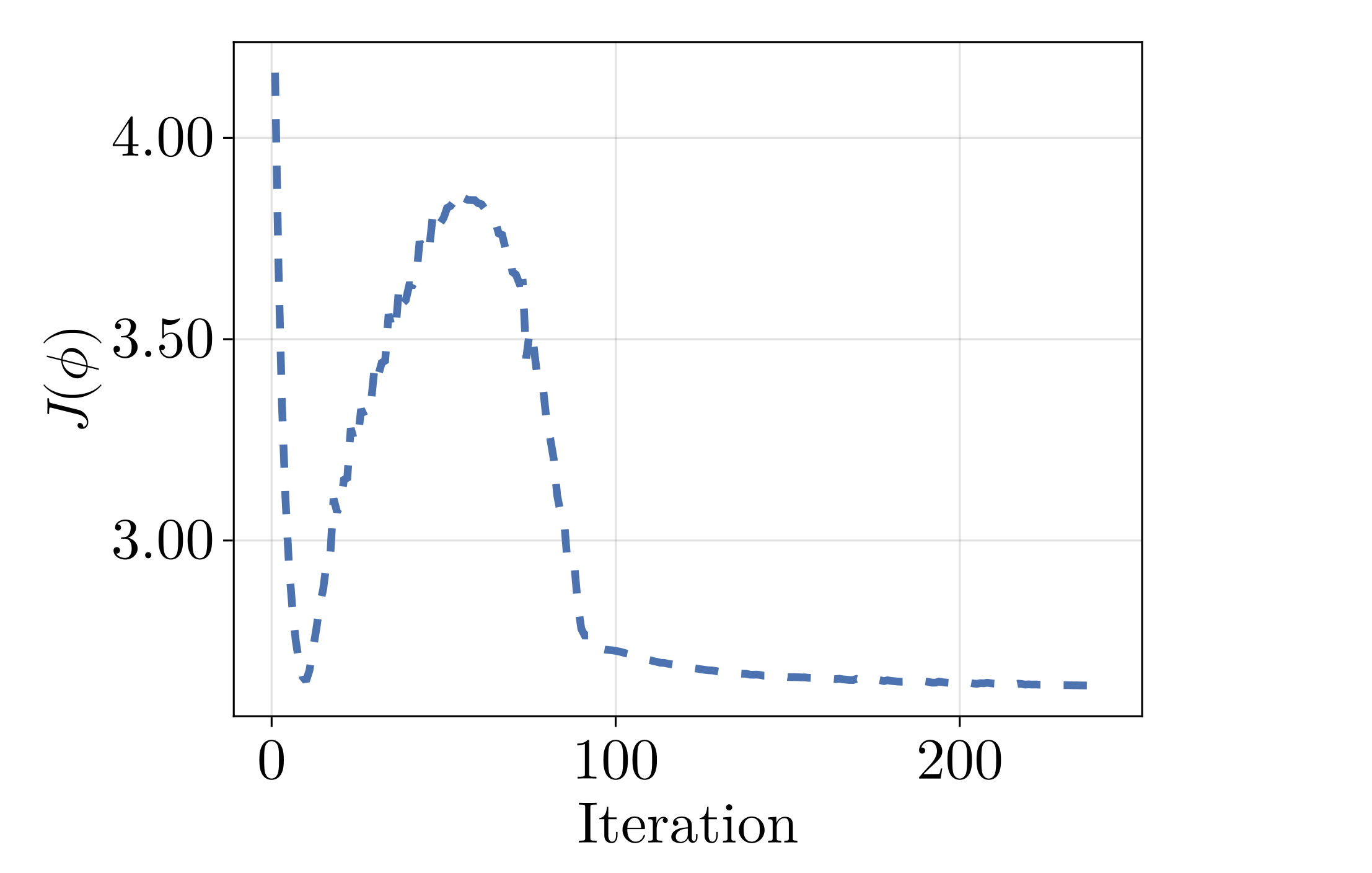}
        \caption{}
    \end{subfigure}
    \hspace{0.5cm}
    \begin{subfigure}{0.42\textwidth}
        \centering
        \includegraphics[width=\textwidth]{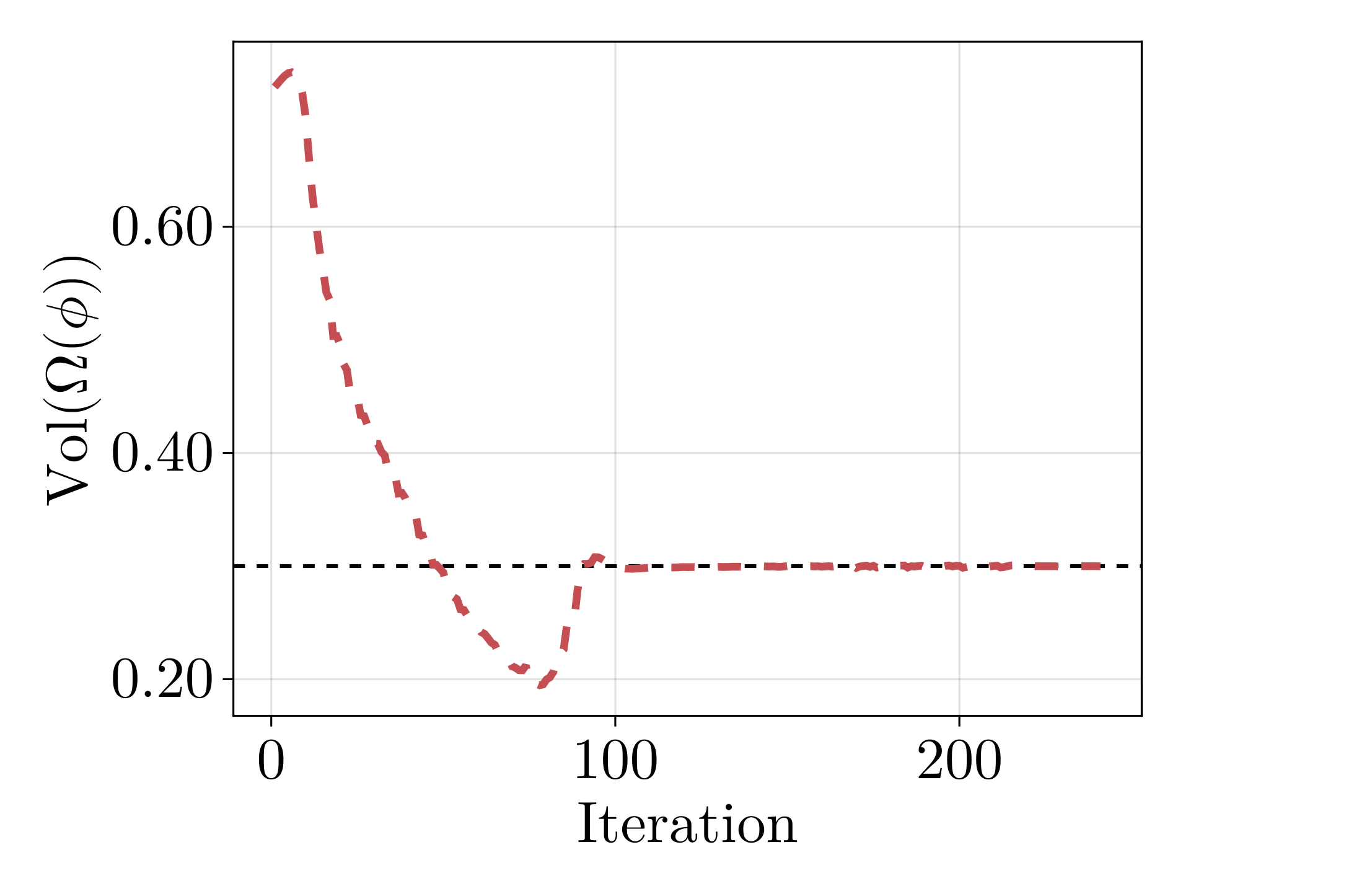}
        \caption{}
    \end{subfigure}
    \caption{Iteration history of the objective (Fig. a) and volume constraint (Fig. b) for the elastic wheel problem.}
    \label{fig: wheel iteration history}
\end{figure}

Figure \ref{fig: wheel results} shows visualisations of intermediate structures and the final structure, while Figure \ref{fig: wheel iteration history} gives the iteration history of the objective and volume constraint. The average iteration time is approximately 4 minutes. This could be significantly improved in future using scalable iterative solvers for the linear systems resulting from the finite element discretisation. In addition, the number of iterations required to reach convergence could be reduced by using a more advanced optimiser such as a projection/null-space method \citep{Wegert_2023b,10.1051/cocv/2020015_2020}.

The results demonstrate that using the unfitted approaches for level-set evolution and reinitialisation, we are able to solve topology optimisation problems on unstructured background domains using the automatic differentiation techniques discussed in Section \ref{sec: Automatic differentiation}.

\subsection{Fluid-structure interaction}

\subsubsection{Formulation}
In this section, we consider topology optimisation of a linear elastic structure in a fluid-structure interaction problem with Stokes flow. We assume that the displacement of the elastic structure is small, so that the coupling in the fluid-structure problem is one way. This allows the fluid and elastic problems to be de-coupled so that they can be solved in a sequential manner. We will refer to these as \textit{staggered} problems.

\begin{figure}
    \centering
    \def\svgwidth{0.8\textwidth}
    \input{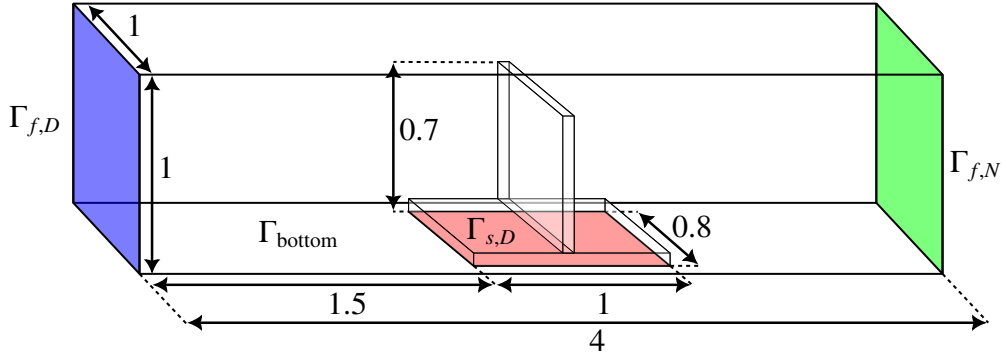}
    \caption{A visualisation of the background domain $D$ and boundary conditions for the fluid-structure interaction problem. The red boundary $\Gamma_{s,D}$ has a homogeneous Dirichlet boundary condition on the displacement, the blue boundary $\Gamma_{f,D}$ has an in-flow Dirichlet boundary condition on the flow, the green boundary $\Gamma_{f,N}$ has a zero stress condition on the flow exiting the domain, and $\Gamma_{\text{bottom}}$ has a no-slip condition. In addition, the sides of the domain $\Gamma_{\text{sides}}=\partial D\setminus(\Gamma_{\text{bottom}}\cup\Gamma_{s,D}\cup\Gamma_{f,D}\cup\Gamma_{f,N})$ have no-slip conditions on the flow. The plate at the centre of the domain is fixed and non-designable. The base of the plate and the width of the upright region both have a depth and width of 0.1, respectively.}
    \label{fig: fsi background}
\end{figure}

Figure \ref{fig: fsi background} shows the background domain and boundary conditions for this problem. The region shown in grey is a non-designable region. This setup is similar to the one in \cite{10.1016/j.jcp.2020.109574_2020}. The strong formulation of the Stokes flow is 
\begin{equation}\label{eqn: FSI strong form stokes}
    \begin{cases}
       -\boldsymbol{\nabla}\cdot\boldsymbol{\sigma}_f(\boldsymbol{u},p)=\boldsymbol{0}&\text{in }\Omega_{\text{out}}(\phi),\\
        \boldsymbol{\nabla}\cdot\boldsymbol{u}=\boldsymbol{0}&\text{in }\Omega_{\text{out}}(\phi),\\
        \boldsymbol{\sigma}_f(\boldsymbol{u},p)\cdot\boldsymbol{n}=\boldsymbol{0}&\text{on }\Gamma_{f,N},\\
        \boldsymbol{u}=\boldsymbol{u}_d&\text{on }\Gamma_{f,D},\\
        \boldsymbol{u}=\boldsymbol{0}&\text{on }\Gamma(\phi)\cup\Gamma_{\text{bottom}},\\
        \boldsymbol{u}\cdot\boldsymbol{n}={0}&\text{on }\Gamma_\text{sides},
    \end{cases}
\end{equation}
while the strong formulation governing the linear elastic structure is
\begin{equation}\label{eqn: FSI strong form elast}
    \begin{cases}
        -\boldsymbol{\nabla}\cdot\boldsymbol{\sigma}(\boldsymbol{d})=\boldsymbol{0}&\text{in }\Omega(\phi),\\
        \boldsymbol{\sigma}(\boldsymbol{d})\cdot\boldsymbol{n}=\boldsymbol{\sigma}_f(\boldsymbol{u},p)\cdot\boldsymbol{n}&\text{in }\Gamma(\phi),\\
        \boldsymbol{d}=\boldsymbol{0}&\text{on }\Gamma_{s,D},
    \end{cases}
\end{equation}
where $\Omega_\text{out}(\phi)=D\setminus\Omega(\phi)$ is the fluid domain, $\mu_f$ is the viscosity of the flow, $p$ is the pressure, $\boldsymbol{u}_d=y\boldsymbol{e}_1$ is the in-flow on $\Gamma_{f,D}$, and $\boldsymbol{\sigma}_f(\boldsymbol{u},p)=2\mu_f\boldsymbol{\varepsilon}(\boldsymbol{u})-p\boldsymbol{I}$ is the stress tensor for the fluid. The system above has one-way coupling as only the surface traction term in Equation \ref{eqn: FSI strong form elast} depends on the normal stress imparted by the fluid on $\Gamma(\phi)$. 

For the Stokes flow problem, we use the stabilised Nitsche fictitious domain method and a face ghost-penalty stabilisation as detailed in \cite{10.1002/nme.4823_2015} with continuous piecewise linear elements for the velocities and piecewise constant elements for the pressures. For the linear elastic problem, we use the same approach as in Section \ref{sec: elastic wheel example}. The weak formulations are given by: 
\begin{weak}\label{weakform: fsi fluid}
For $\Omega_{\text{out}}(\phi)\subset D$, find $(\boldsymbol{u},p)\in U\times Q$ such that
\begin{equation}
    R_1((\boldsymbol{u},p),(\boldsymbol{v},q))=a_f(\boldsymbol{u},\boldsymbol{v})+b_f(\boldsymbol{v},p)+b_f(\boldsymbol{u},q)+j_{f,u}(\boldsymbol{u},\boldsymbol{v})-j_{f,p}(p,q)+j_{f,\psi}(p,q)=0,~\forall (\boldsymbol{v},q)\in V\times Q,
\end{equation}
with
\begin{align}
    &a_{f}(\boldsymbol{u},\boldsymbol{v})=\int_{\Omega_\text{out}(\phi)}\mu_f\boldsymbol{\nabla u}:\boldsymbol{\nabla u}~\mathrm{d}\boldsymbol{x}-\int_{\Gamma(\phi)}\mu_f(\bn\cdot\boldsymbol{\nabla u})\cdot\boldsymbol{v}+\mu_f(\bn\cdot\boldsymbol{\nabla v})\cdot\boldsymbol{u}-\frac{\gamma_N}{h}\boldsymbol{u}\cdot\boldsymbol{v}~\mathrm{d}s,\\
    &b_{f}(\boldsymbol{v},p)=-\int_{\Omega_\text{out}(\phi)}p\boldsymbol{\nabla}\cdot\boldsymbol{v}~\mathrm{d}\boldsymbol{x}-\int_{\Gamma(\phi)}p\bn\cdot\boldsymbol{v}~\mathrm{d}s,\\
    &j_{f,u}(\boldsymbol{u},\boldsymbol{v})=\gamma_u\mu_f\sum_{F\in\mathscr{T}_G}\int_{F} h_F\llbracket\boldsymbol{n}_F\cdot\boldsymbol{\nabla d}\rrbracket\cdot\llbracket\boldsymbol{n}_F\cdot\boldsymbol{\nabla s}\rrbracket~\mathrm{d}s,\\
    &j_{f,p}(p,q)=\frac{\gamma_p}{\mu_f}\sum_{F\in\mathscr{T}_h}\int_{F}h_F\llbracket p\rrbracket\llbracket q\rrbracket~\mathrm{d}s,\\
    &j_{f,\psi}(p,q)=\int_{\Omega_\text{out}(\phi)}\psi_{f} pq~\mathrm{d}\boldsymbol{x},
\end{align}
where $U=\{\boldsymbol{u}\in [H^1(\Omega(\phi))]^3:\boldsymbol{u}\rvert_{\Gamma_{f,D}}=\boldsymbol{u}_d,~\boldsymbol{u}\rvert_{\Gamma(\phi)\cup\Gamma_{\text{bottom}}}=\boldsymbol{0},~(\boldsymbol{u}\cdot\boldsymbol{n})\rvert_{\Gamma_{sides}}=0\}$ and $V=\{\boldsymbol{u}\in [H^1(\Omega(\phi))]^3:\boldsymbol{u}\rvert_{\Gamma_{f,D}}=\boldsymbol{0},~\boldsymbol{u}\rvert_{\Gamma(\phi)\cup\Gamma_{\text{bottom}}}=\boldsymbol{0},~(\boldsymbol{u}\cdot\boldsymbol{n})\rvert_{\Gamma_{sides}}=0\}$ are the trial and test spaces for the velocity field, and $Q=L^2(\Omega)$ is pressure space.   
\end{weak}
\begin{weak}
    For $\Omega(\phi)\subset D$ and $(\boldsymbol{u},p)\in U\times Q$ such that $R_1((\boldsymbol{u},p),(\boldsymbol{v},q))=0,~\forall (\boldsymbol{v},q)\in V\times Q$, find $\boldsymbol{d}\in T$ such that
    \begin{equation}\label{eqn: elast weak form fsi}
        R_2((\boldsymbol{u},p),\boldsymbol{d},\boldsymbol{s})=a(\boldsymbol{d},\boldsymbol{s})+k(\boldsymbol{d},\boldsymbol{s})+j(\boldsymbol{d},\boldsymbol{s})- l_s((\boldsymbol{u},p),\boldsymbol{s})= 0,~\forall \boldsymbol{s}\in T,
    \end{equation}
    with $a(\boldsymbol{d},\boldsymbol{s})$, $k(\boldsymbol{d},\boldsymbol{s})$, and $j(\boldsymbol{d},\boldsymbol{s})$ as in Equation \eqref{eqn: elast a}, \eqref{eqn: elast k}, and \eqref{eqn: elast j}, and
    \begin{equation}
        l_s((\boldsymbol{u},p),\boldsymbol{s})=\int_{\Gamma(\phi)}(1-\psi_s)\boldsymbol{s}\cdot (\boldsymbol{\sigma}_f(\boldsymbol{u},p)\cdot\bn)~\mathrm{d}s
    \end{equation}
    where $T=\{\boldsymbol{d}\in [H^1(\Omega(\phi))]^3:\boldsymbol{d}\rvert_{\Gamma_{s,D}}=\boldsymbol{0}\}$.
\end{weak}
\noindent In Weak form \ref{weakform: fsi fluid}, the inclusion of $1-\psi_s$ in $l_s$ ensures that isolated volumes of the solid phase have zero displacement, $\gamma_N=100$ is the Nitsche parameter, $j_{f,u}(\boldsymbol{u},\boldsymbol{v})$ is the ghost penalty term with stabilisation parameter $\gamma_u=0.1$, $j_{f,p}(p,q)$ is the symmetric pressure stabilisation for piecewise constant pressures with stabilisation parameter $\gamma_p=0.25$, and $j_{f,\psi}(p,q)$ enforces an average pressure in isolated volumes of $\Omega_{\text{out}}(\phi)$ marked by $\psi_f$.

The optimisation problem is then given by
\begin{equation}\label{eqn: fsi opt prob}
    \begin{aligned}
    \min_{\phi\in W^h}&~J(\phi)\coloneqq\int_{\Omega(\phi)}\boldsymbol{\sigma(d)}:\boldsymbol{\varepsilon(d)}~\mathrm{d}\boldsymbol{x}+\int_{\Omega(\phi)}\gamma_{s}\psi_s~\mathrm{d}{x}\\
    \text{s.t. }&~C(\phi)=0,\\
    &~R_1((\boldsymbol{u},p),(\boldsymbol{v},q))=0,~\forall (\boldsymbol{v},q)\in V\times Q,\\
    &~R_2((\boldsymbol{u},p),\boldsymbol{d},\boldsymbol{s})=0,~\forall \boldsymbol{s}\in T,
    \end{aligned}
\end{equation}
where $C(\phi)$ is the volume constraint as previously with $V_f=0.06$. It should be noted that isolated volumes in the solid phase obstruct fluid flow due to the no-slip boundary conditions at the interface. As a result, these can artificially improve the objective. We therefore augment the minimum compliance objective with a term that penalises isolated volumes in the solid phase with a penalisation constant of $c=1000$. This ensures that isolated volumes in the solid phase are removed. Note that when no isolated volumes are present, the objective reduces to minimum compliance. 

To compute derivatives of $J$ with respect to $\phi$, we develop an adjoint method for staggered problems based on C\'ea's method \cite{10.1051/m2an/1986200303711_1986} and implement this in GridapTopOpt \citep{GridapTopOpt}. This is given in \ref{sec: Adjoint method for staggered problems} for general staggered problems of size $k$. Our particular example corresponds to the case $k=2$. Incorporating additional physics (e.g., thermostatics) increases the size of the staggered system.

\subsubsection{Results}
To solve the optimisation problem in Equation \eqref{eqn: fsi opt prob} we again use the augmented Lagrangian method. In addition, we generate an unstructured symplectic background triangulation with refinement in the region around the non-designable region to improve accuracy of the solutions near the interface of the elastic structure. The non-refined region has a maximum element size of 0.05 while the refined region has a maximum element size of 0.015. In addition, we use symmetry to halve the computational domain. As a result, the background triangulation has roughly 824K elements and 129K nodes. We partition the background triangulation into 48 parts and solve the optimisation problem in a memory-distributed framework using 48 processors. We use SuperLU\_DIST \cite{Demmel_Gilbert_Li_1999,Li_Demmel_2003,Li_Lin_Liu_Sao_2023} to solve all linear systems in parallel.

\begin{figure}[!t]
    \centering
    \begin{subfigure}{0.32\textwidth}
        \centering
        \includegraphics[width=\textwidth]{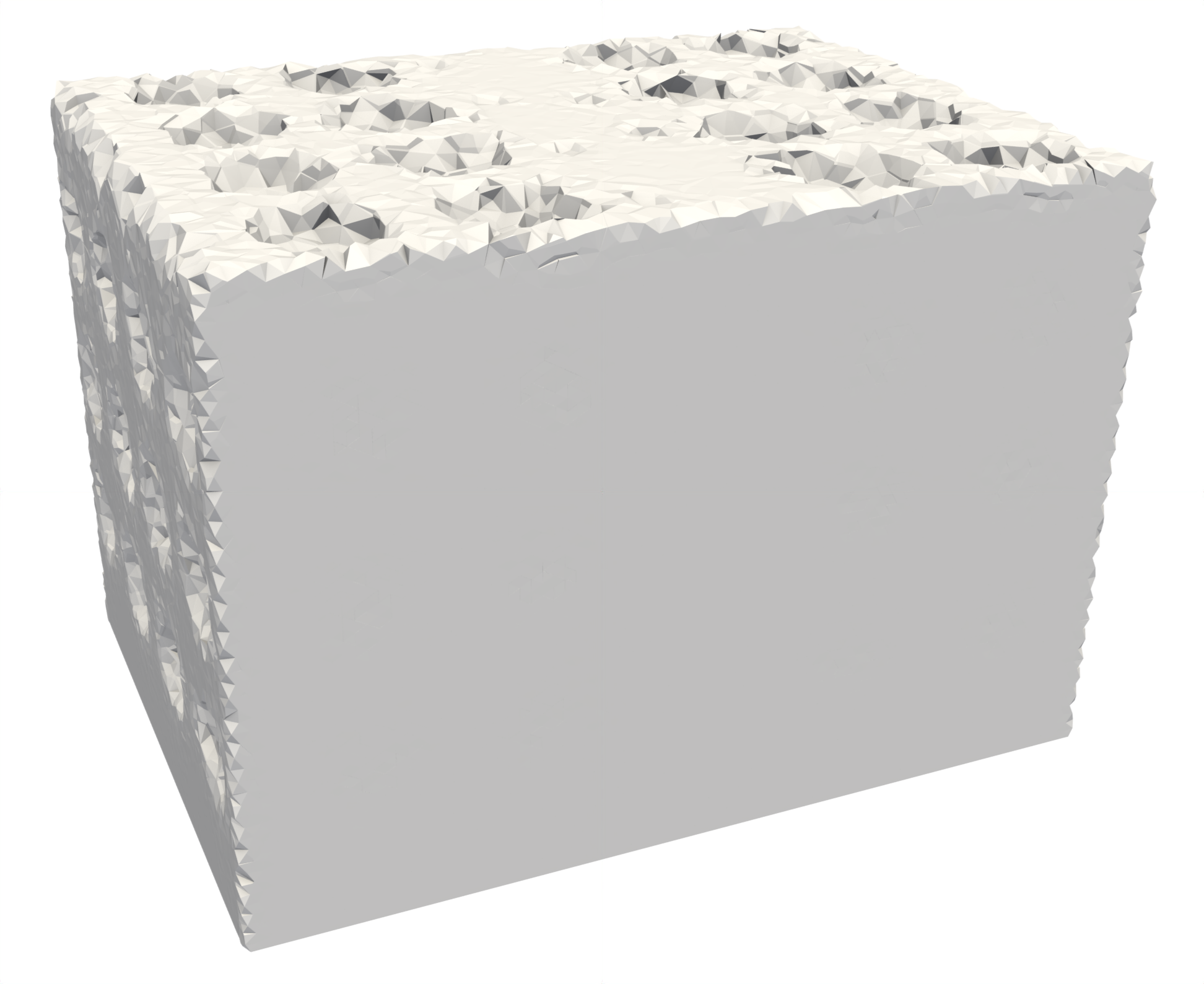}
        \caption{It. 0}
    \end{subfigure}
    \begin{subfigure}{0.32\textwidth}
        \centering
        \includegraphics[width=\textwidth]{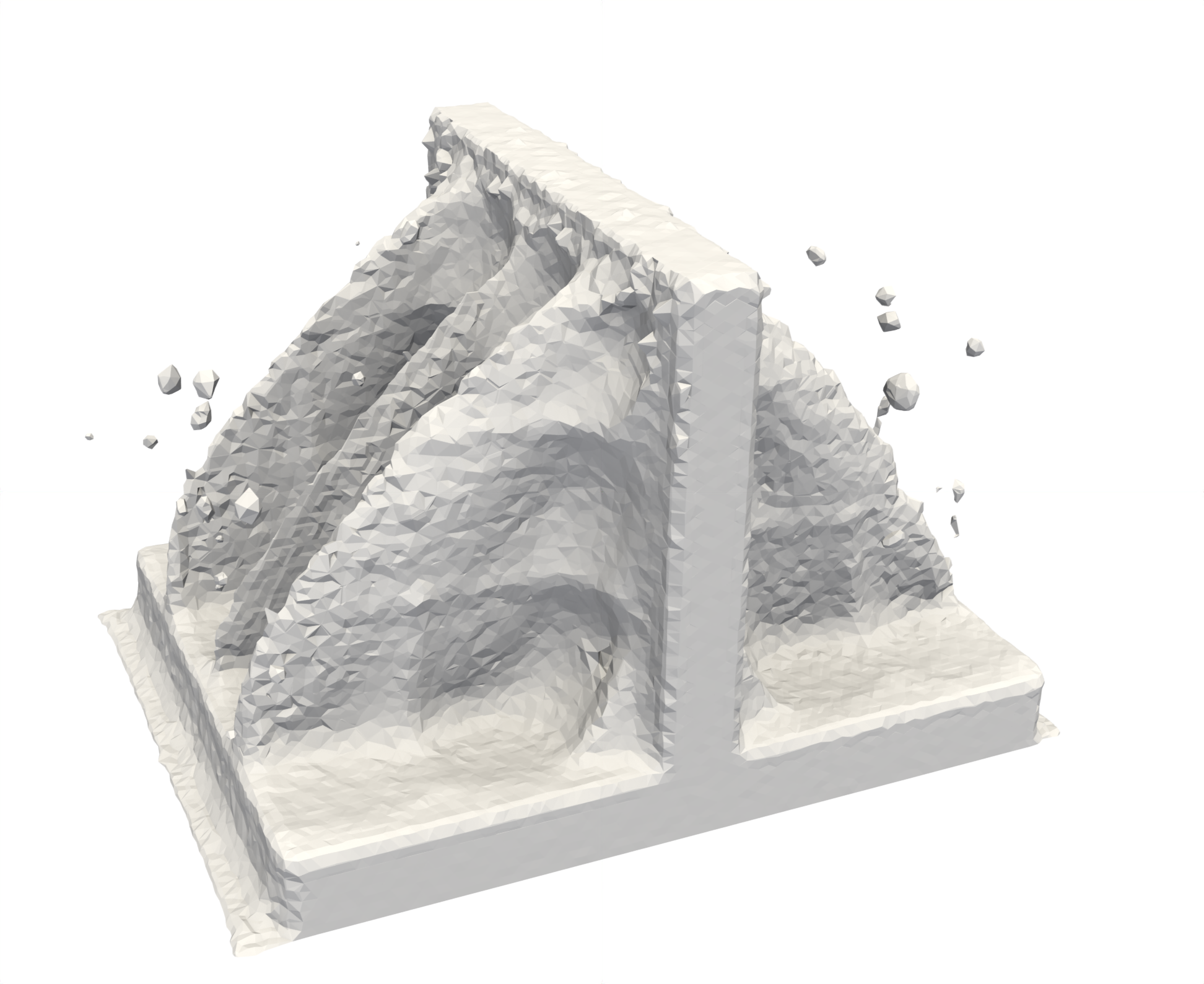}
        \caption{It. 50}
    \end{subfigure}
    \begin{subfigure}{0.32\textwidth}
        \centering
        \includegraphics[width=\textwidth]{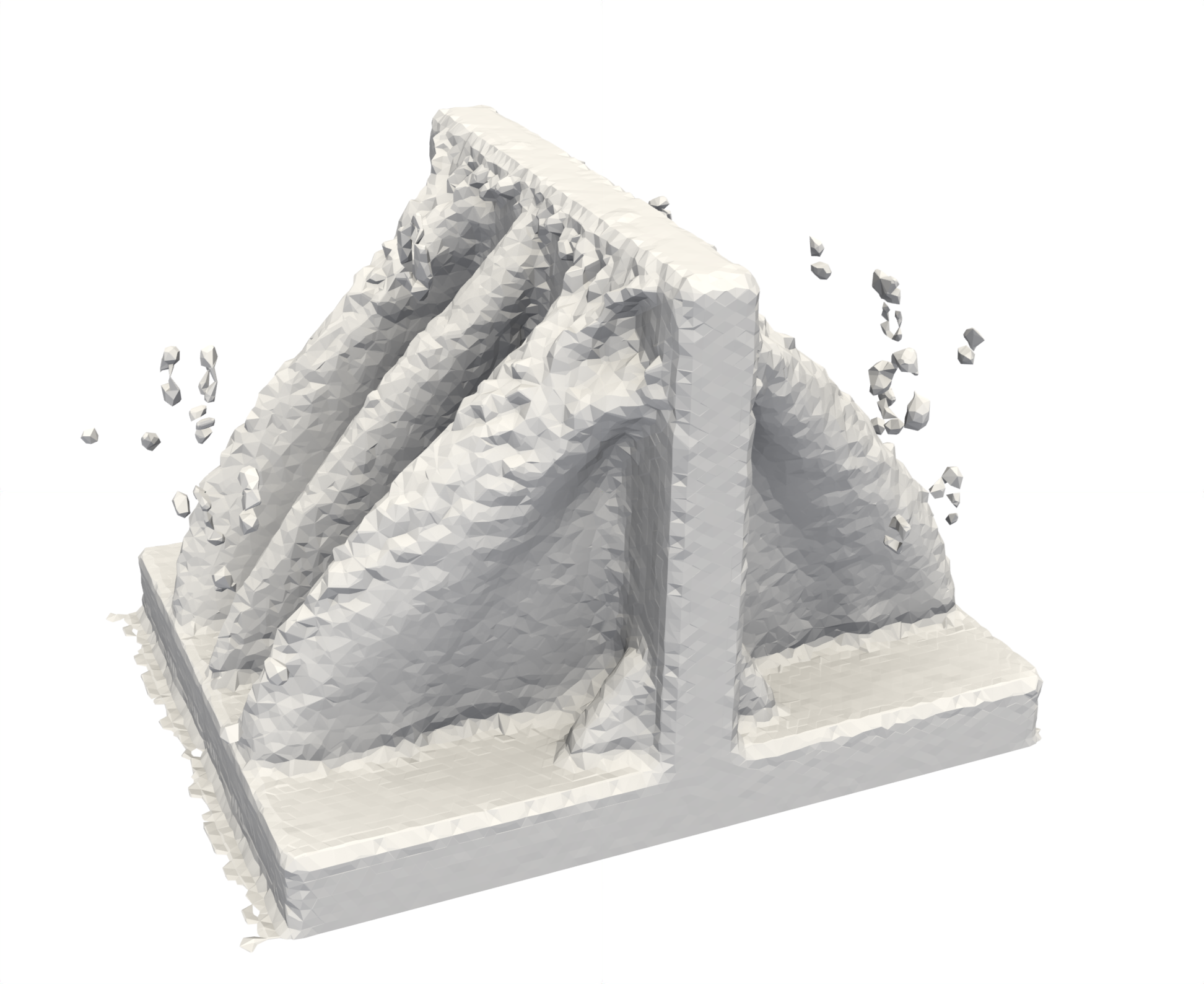}
        \caption{It. 100}
    \end{subfigure}
    \begin{subfigure}{0.32\textwidth}
        \centering
        \includegraphics[width=\textwidth]{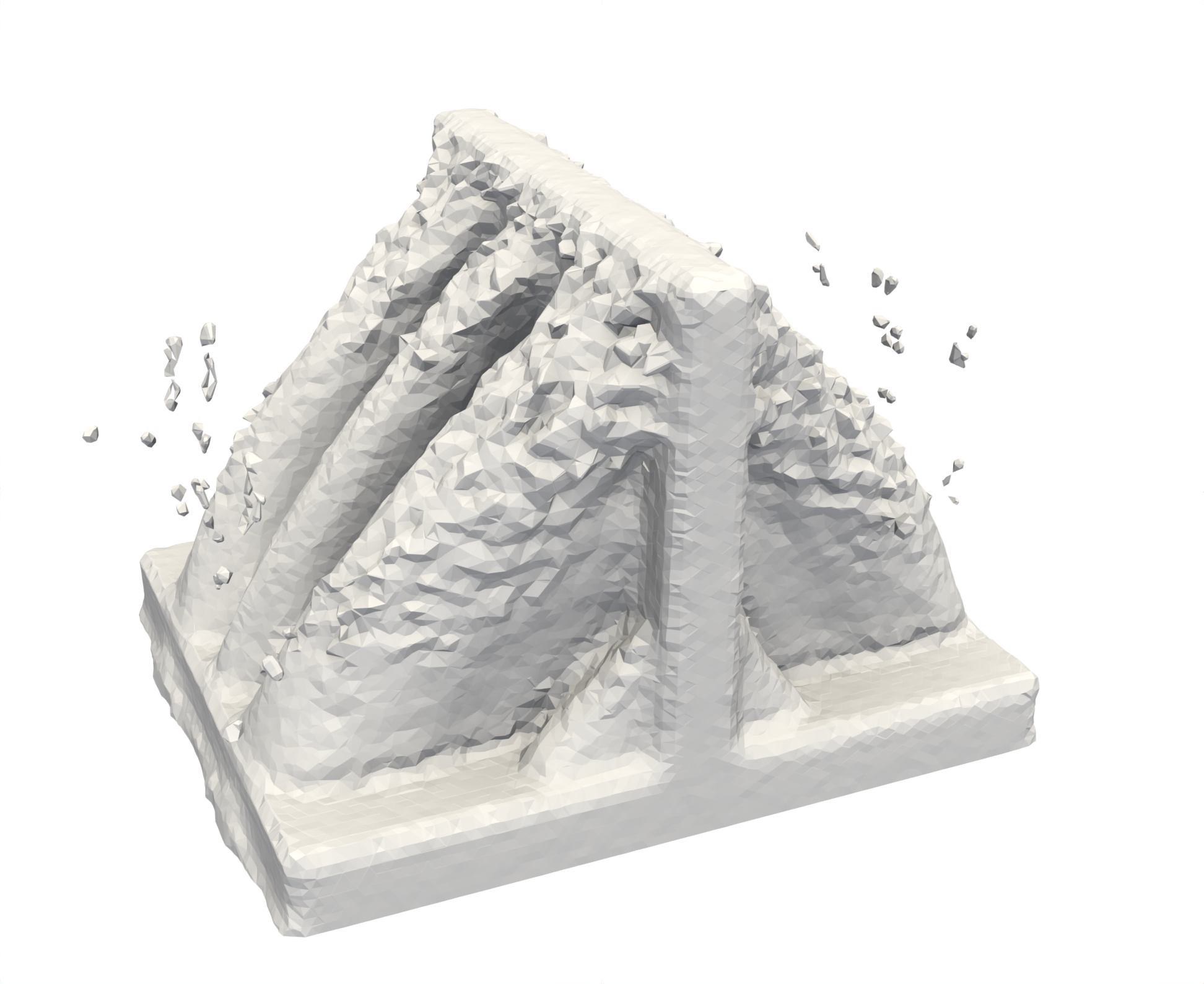}
        \caption{It. 200}
    \end{subfigure}
    \begin{subfigure}{0.32\textwidth}
        \centering
        \includegraphics[width=\textwidth]{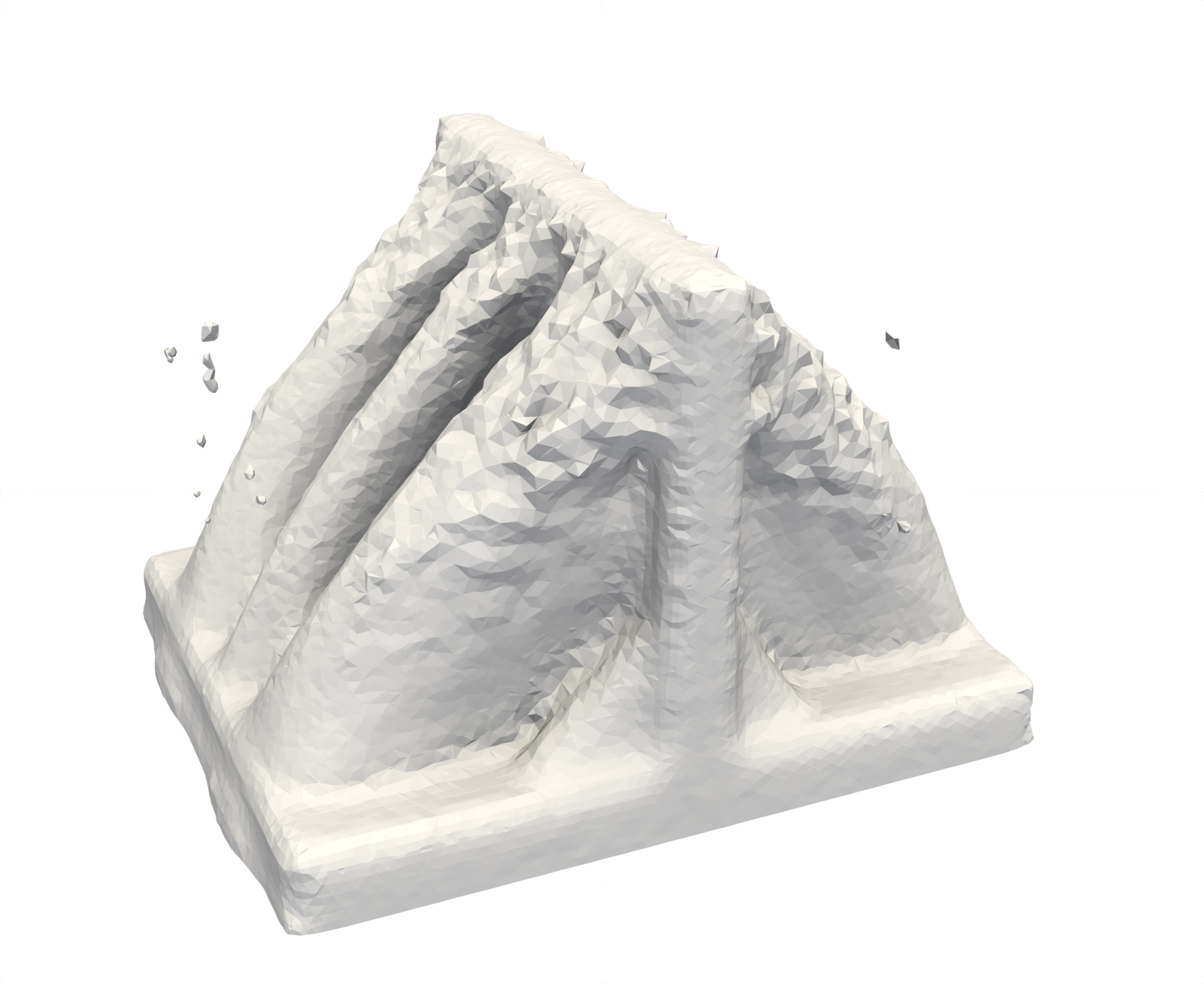}
        \caption{It. 400}
    \end{subfigure}
        \begin{subfigure}{0.32\textwidth}
        \centering
        \includegraphics[width=\textwidth]{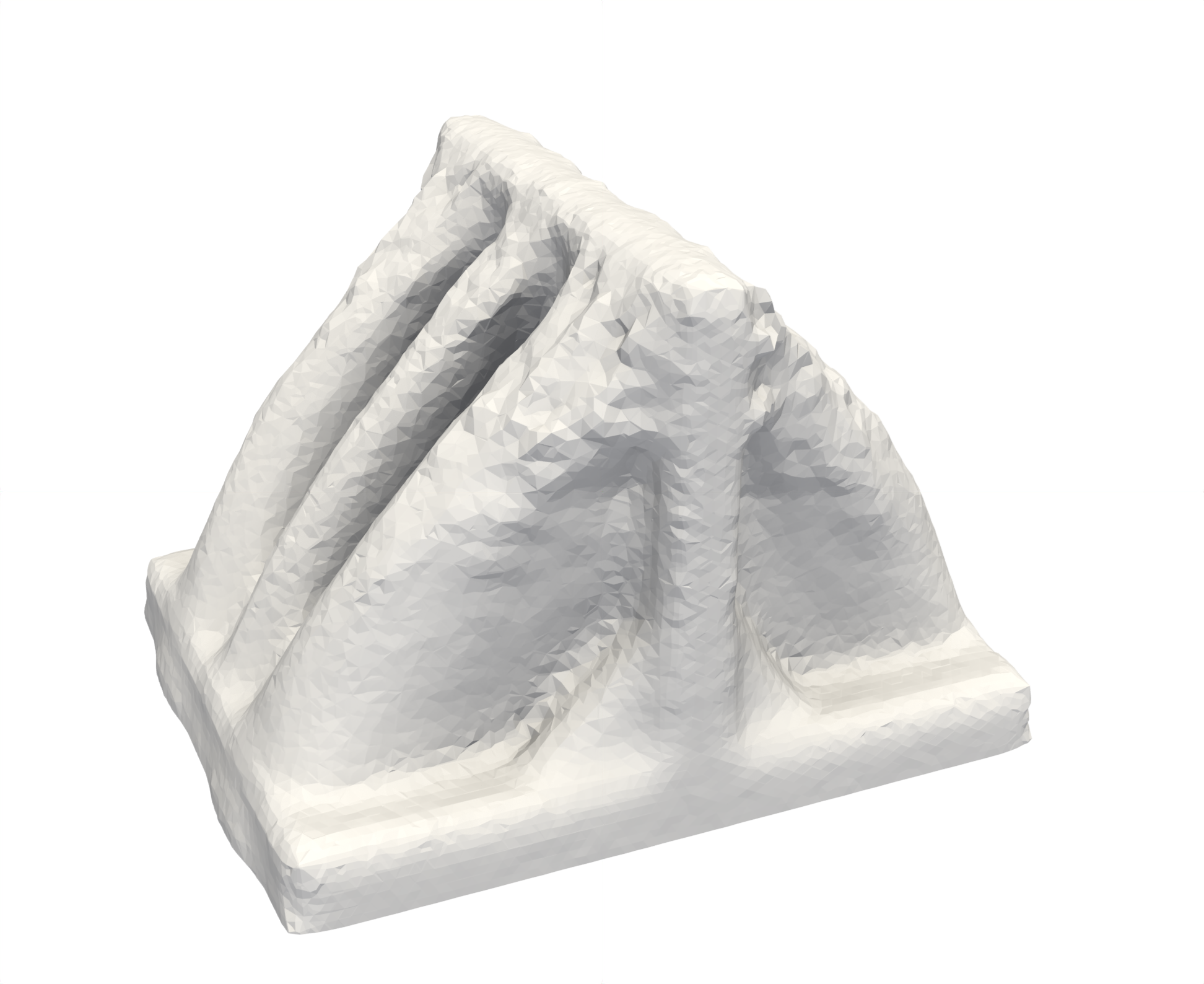}
        \caption{It. 519}
    \end{subfigure}
        \begin{subfigure}{1\textwidth}
        \centering
        \includegraphics[width=\textwidth]{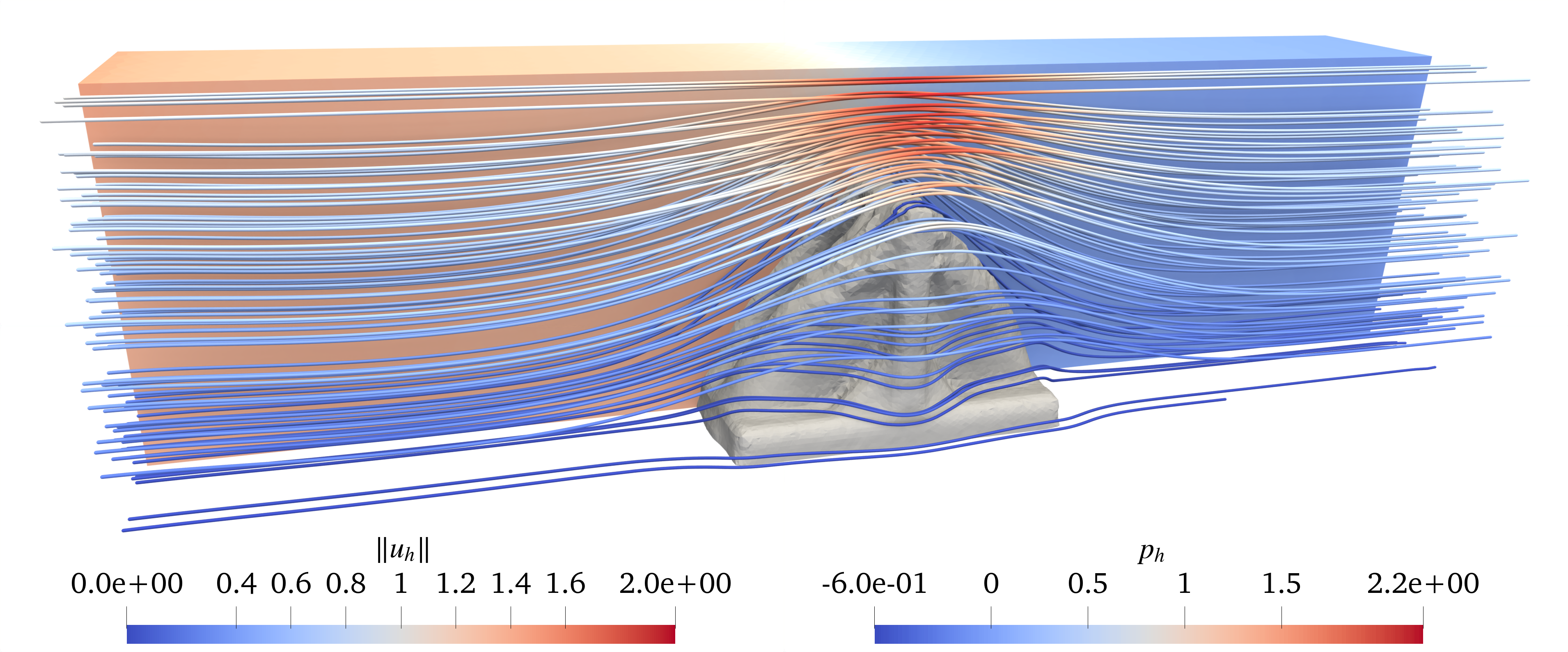}
        \caption{}
    \end{subfigure}
    \caption{Visualisations of the initial structure (Fig. a), intermediate structures (Fig. b -- Fig. e) and the final optimised structure (Fig. f) for the fluid-structure interaction problem. Figure (g) shows a visualisation of the velocity field and pressure field in the fluid.}
    \label{fig: FSI results}
\end{figure}
\begin{figure}[!t]
    \centering
    \begin{subfigure}{0.45\textwidth}
        \centering
        \includegraphics[width=\textwidth]{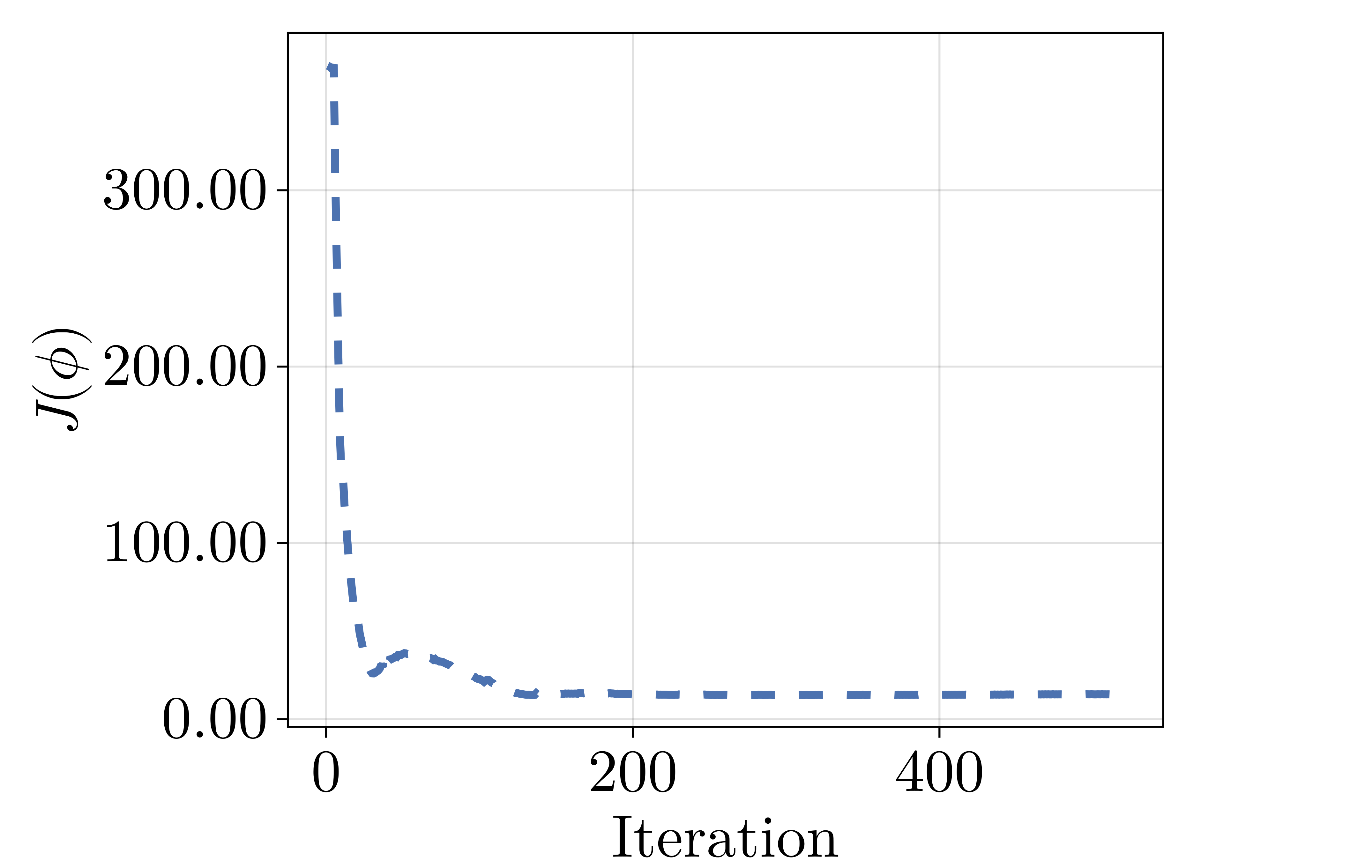}
        \caption{}
    \end{subfigure}
    \begin{subfigure}{0.45\textwidth}
        \centering
        \includegraphics[width=\textwidth]{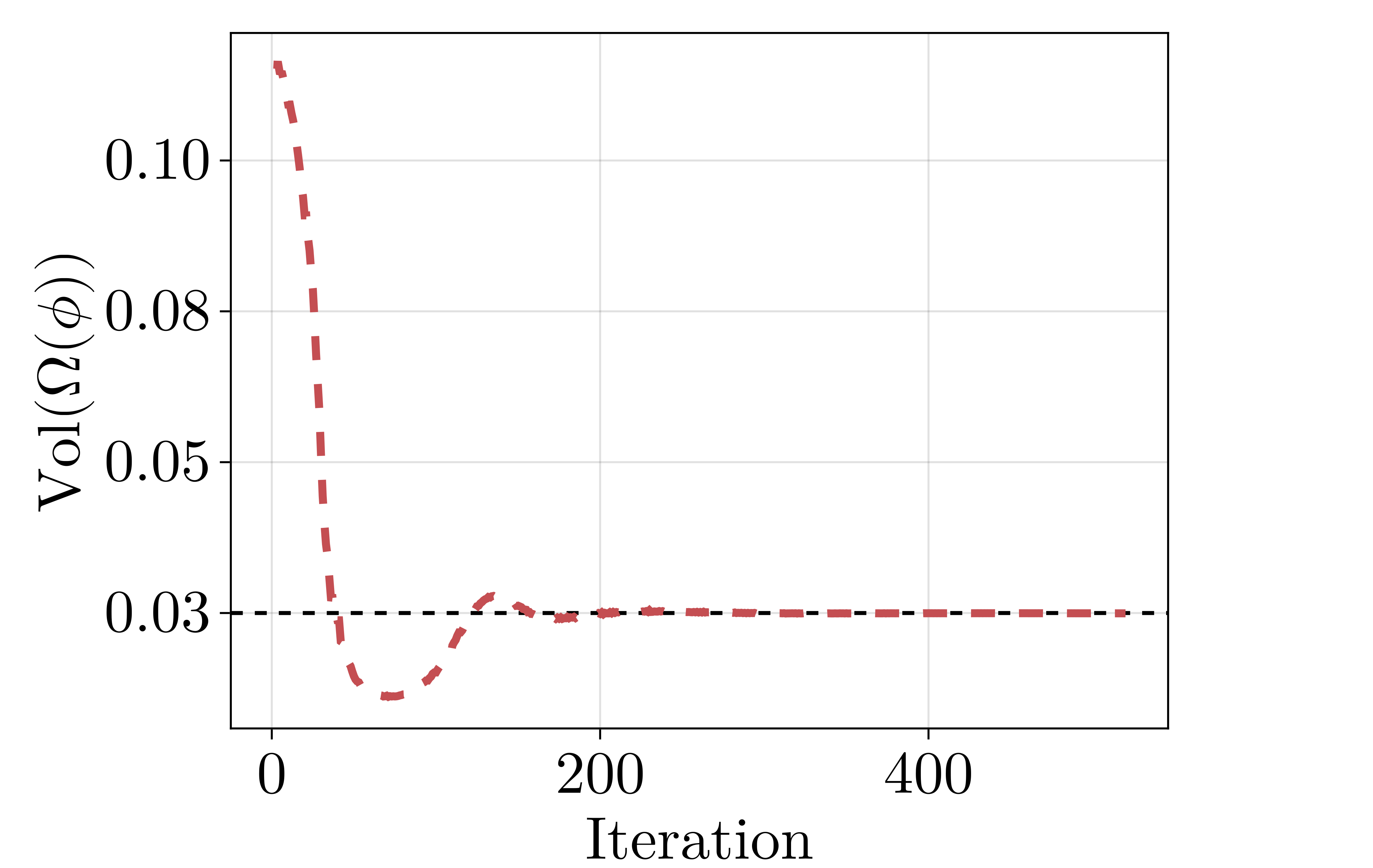}
        \caption{}
    \end{subfigure}
    \caption{Iteration history of the objective (Fig. a) and volume constraint (Fig. b) for the fluid-structure interaction problem.}
    \label{fig: FSI iteration history}
\end{figure}

Figure \ref{fig: FSI results} show visualisations of the initial structure, intermediate structures, and the final structure, while Figure \ref{fig: FSI iteration history} gives the iteration history of the objective and volume constraint. The average iteration time is approximately 8.5 minutes. As with the previous example, this could be significantly improved in future by using scalable iterative solvers and the number of iterations required to reach convergence could be reduced by using a more advanced optimiser.

These results demonstrate that automatic shape differentiation for unfitted discretisations can be used to solve more complicated multi-physics problems, thereby avoiding the difficult and error-prone calculation of shape derivatives. In addition, the unfitted finite element allows us to accurately represent the boundary condition for the flow field.

\section{Conclusions}\label{sec: Conclusions}
In this paper, we extended the shape calculus results developed in \cite{Berggren_2023} to the case where the domain boundary intersects the background domain boundary and linked these theoretical results to automatic shape differentiation for unfitted discretisations. We implemented analytic directional shape derivatives and unfitted automatic shape differentiation in the Julia package GridapTopOpt \cite{GridapTopOpt} and showed that the analytic directional shape derivatives given by \citet{Berggren_2023} and the extensions considered in this work can be recovered to machine precision regardless of mesh size. We proposed that unfitted automatic shape differentiation can be used to avoid computing intricate mesh-related quantities that appear in directional derivatives of surface integrals. Furthermore, automatic shape differentiation can compute derivatives of expressions for which there is not yet a rigorous mathematical counterpart in the framework of \citet{Berggren_2023}, such as shape Hessians. We showed that our implementation readily captures the shape Hessian and again verified the result with finite differences.

Our implementation in GridapTopOpt and the wider Gridap ecosystem is general and can be executed in both serial and distributed computing frameworks to solve arbitrary PDE-constrained optimisation problems using unfitted finite elements. We propose a novel graph-based approach to detect isolated volumes using a conforming connectivity graph for the cut mesh and a graph-colouring algorithm. These extensions allow GridapTopOpt to solve an even wider range of problems on unstructured background triangulations. Furthermore, by utilising unfitted methods, we avoid issues that arise from smoothing the material properties across the interface, thus providing higher accuracy solutions. To demonstrate the applicability of the unfitted automatic shape differentiation framework and our implementation, we considered two example optimisation problems: minimum compliance optimisation of a three-dimensional linear elastic wheel; and topology optimisation of a three-dimensional linear elastic structure in a fluid-structure interaction problem with Stokes flow. We solved these problems using CutFEM and distributed the computation over a computing cluster. We avoided computation of all analytic directional shape derivatives by leveraging the automatic shape differentiation techniques developed in this work.

In the future, we plan to extend the framework to multi-material problems involving several level-set functions and improve the scalability of algorithms by leveraging matrix-free methods.

\section*{CRediT authorship contribution statement}
\textbf{Zachary J Wegert:} Writing -- original draft, Writing -- review and editing, Conceptualization, Data curation, Formal analysis, Investigation, Methodology, Software, Validation, Visualization. \textbf{Jordi Manyer:} Writing -- original draft, Writing -- review \& editing, Software, Methodology. \textbf{Connor Mallon:} Writing -- review \& editing, Supervision. \textbf{Santiago Badia:} Supervision, Writing -- review \& editing, Resources. \textbf{Vivien J Challis:} Supervision, Writing -- review \& editing, Conceptualization, Project administration, Funding acquisition, Resources.

\section*{Declaration of Competing Interest}
The authors have no competing interests to declare that are relevant to the content of this article.

\section*{Data availability}
The source code and data for this work is available at \url{https://github.com/zjwegert/GridapTopOpt.jl}.

\section*{Acknowledgement}
The authors would like to thank Prof. Martin Berggren for his insightful and encouraging comments. This work was supported by the Australian Research Council through the Discovery Projects grant scheme (DP220102759). This research used computational resources provided by: the eResearch Office, Queensland University of Technology; the Queensland Cyber Infrastructure Foundation (QCIF); and the National Computational Infrastructure (NCI) Australia. The first author is supported by a QUT Postgraduate Research Award and a Supervisor Top-Up Scholarship. The above support is gratefully acknowledged.

\appendix

\section{Adjoint method for staggered problems}
\label{sec: Adjoint method for staggered problems}
In this appendix, we consider the directional derivative of a functional $J(u_1,\dots,u_k,\phi)$ at $\phi$ in the direction $\psi$ where $(u_1,\dots,u_k)\in U=U_1\times\dots\times U_k$ depend on $\phi\in\Phi$ through the residuals
\begin{equation}\label{appendix eqn: residuals}
    \begin{aligned}
    R_1(u_1,v_1,\phi)=0,&\quad\forall v_1\in V_1,\\
    R_2(u_1,u_2,v_2,\phi)=0,&\quad\forall v_2\in V_2,\\
    \vdots\quad~&\\
    R_k(u_1,\dots,u_k,v_k,\phi)=0,&\quad\forall v_k\in V_k,
\end{aligned}
\end{equation}
where $u_1,\dots,u_k$ have the following dependence
\begin{equation}\label{appendix eqn: ui dependence}
    \begin{aligned}
        &u_1=u_1(\phi),\\
        &u_2=u_2(u_1,\phi),\\
        &\quad~~\vdots\\
        &u_k=u_k(u_1,\dots,u_k,\phi),\\
    \end{aligned}
\end{equation}
and $R_i$ is linear in the test function $v_i$.

We assume that the above are sufficiently differentiable and proceed with C\'ea's formal method \cite{10.1051/m2an/1986200303711_1986}. For brevity of the construction, we drop the notational dependence of $J$ and $R_i$ on $u_1,\dots,u_k$ and $\phi$, excepting the definition of the Lagrangian below. In particular, we use the following shorthand for directional derivatives: the directional derivative of $J$ at $u_i$ in the direction $q_i$ is
\begin{equation}
    \partial_{u_i}J(q_i)\coloneqq\partial_{u_i}J(u_1,\dots,u_k,\phi)(q_i)\coloneqq\left.\D{}{t}\right\rvert_{t=0}J(u_1,\dots,u_i+tq_i,\dots,u_k,\phi),
\end{equation}
and similarly for other functionals and fields. 

Suppose that we define the Lagrangian $\mathcal{L}:U\times V\times\Phi$ to be
\begin{equation}\label{appendix eqn: lagrangian}
    \mathcal{L}(p_1,\dots,p_k,\lambda_1,\dots,\lambda_k,\phi)=J(p_1,\dots,p_k,\phi)-\sum_{i=1}^{k}R_i(p_1,\dots,p_i,\lambda_i,\phi).
\end{equation}
where $V=V_1\times\dots\times V_k$ and $(\lambda_1,\dots,\lambda_k)\in V$ are called the adjoint variables.

Clearly, at the solutions $u_1,\dots,u_k$ of the residuals in Equation \eqref{appendix eqn: residuals} we recover $J$ via
\begin{equation}\label{appendix eqn: recovery}
        \mathcal{L}\lvert_{(p_1,\dots,p_k)=(u_1,\dots,u_k)}=J(u_1,\dots,u_k,\phi).
\end{equation}

A derivative of $\mathcal{L}$ at $\lambda_i$ in the direction $q_i$ gives
\begin{equation}\label{appendix eqn: residual recovery}
    \partial_{\lambda_i}\mathcal{L}(q_i)=-R_i(p_1,\dots,p_i,q_i,\phi)
\end{equation}
by linearity in the test function. Requiring stationarity of Equation \eqref{appendix eqn: residual recovery} recovers the ith equation in \eqref{appendix eqn: residuals} for $(p_1,\dots,p_k)=(u_1,\dots,u_k)$.

Next, a derivative of $\mathcal{L}$ at $p_i$ in the direction $w_i$ gives
\begin{equation}\label{appendix eqn: deriv in lambda}
    \partial_{p_i}\mathcal{L}(w_i)=\partial_{p_i}J(w_i)-\sum_{j=i}^k\partial_{p_i}R_{j}(w_i)
\end{equation}
Requiring stationarity of Equation \eqref{appendix eqn: deriv in lambda} recovers an equation for the ith adjoint problem. Note that the ith equation in Equation \eqref{appendix eqn: deriv in lambda} depends on $\lambda_{i+1},\dots,\lambda_k$. This defines a staggered system of linear adjoint problems of the form: For $(p_1,\dots,p_k)=(u_1,\dots,u_k)$, find $(\lambda_1,\dots,\lambda_k)\in V$ such that
\begin{equation}\label{appendix eqn: adjoint}
    \begin{aligned}
        &\partial_{u_k}R_k(u_1,\dots,u_k,\lambda_k,\phi)(w_k) = \partial_{u_k}J(u_1,\dots,u_k,\phi)(w_k),\quad w_k\in U_k,\\
        &\partial_{u_{k-1}}R_{k-1}(u_1,\dots,u_{k-1},\lambda_{k-1},\phi)(w_{k-1}) = \partial_{u_{k-1}}J(u_1,\dots,u_{k-1},u_{k},\phi)(w_{k-1})\\&\hspace{5.5cm}-\partial_{u_{k-1}}R_{k}(u_1,\dots,u_{k},\lambda_{k},\phi)(w_{k-1}),\quad w_{k-1}\in U_{k-1},\\
        &\hspace{1.5cm}\vdots\\
        &\partial_{u_i}R_i(u_1,\dots,u_i,\lambda_i,\phi)(w_i) = \partial_{u_i}J(u_1,\dots,u_i,\dots,u_k,\phi)(w_i)-\sum_{j=i+1}^k\partial_{u_i}R_j(u_1,\dots,u_j,\lambda_j,\phi)(w_i),\quad w_i\in U_i,\\
        &\hspace{1.5cm}\vdots\\
        &\partial_{u_1}R_1(u_1,\lambda_1,\phi)(w_1) = \partial_{u_1}J(u_1,\dots,u_k,\phi)(w_1)-\sum_{j=2}^k\partial_{u_1}R_j(u_1,\dots,u_j,\lambda_j,\phi)(w_1),\quad w_1\in U_1.
    \end{aligned}
\end{equation}
For clarity, we explicitly show the dependence on the fields $u_i$ and $\lambda_i$. Note that the test and trial functions are swapped, so the transpose is assembled at discretisation. Each $\lambda_i$ can be found by first solving the problem for $\lambda_k$ then $\lambda_{k-1}$ through to $\lambda_1$.

Finally, at the solution $(p_1,\dots,p_k)=(u_1,\dots,u_k)$ that satisfies the residuals in Equation \eqref{appendix eqn: residuals}, and $(\lambda_1,\dots,\lambda_k)$ that that satisfies the staggered linear system in Equation \eqref{appendix eqn: adjoint}, a derivative of $\mathcal{L}$ at $\phi$ in the direction $\psi$ gives
\begin{equation}
    \mathrm{d}\mathcal{L}(\phi;\psi)=\partial_\phi J(\psi)-\sum_{i=1}^k\partial_\phi R_i(\psi) + \sum_{i=1}^k\partial_{\lambda_i} \mathcal{L}(\partial_\phi\lambda_i(\psi)) + \sum_{i=1}^k\partial_{u_i} \mathcal{L}(\partial_\phi u_i(\psi))
\end{equation}
where we have used the chainrule. The last two expressions vanish due to the stationarity of Equations \eqref{appendix eqn: residual recovery} and \eqref{appendix eqn: adjoint}. As we recover the $J$ from $\mathcal{L}$ in Equation \eqref{appendix eqn: recovery}, we have
\begin{equation}
    \mathrm{d}J(\phi;\psi)=\partial_\phi J(u_1,\dots,u_k,\phi)(\psi)-\sum_{i=1}^k\partial_\phi R_i(u_1,\dots,u_i,\lambda_i,\phi)(\psi),
\end{equation}
where we have written out the explicit dependence on fields $u_i$ and $\lambda_i$.

\bibliographystyle{elsarticle-num-names} 
\bibliography{main}

\newcommand{\noop}[1]{}
\begin{thebibliography}{47}
\expandafter\ifx\csname natexlab\endcsname\relax\def\natexlab#1{#1}\fi
\providecommand{\url}[1]{\texttt{#1}}
\providecommand{\href}[2]{#2}
\providecommand{\path}[1]{#1}
\providecommand{\DOIprefix}{doi:}
\providecommand{\ArXivprefix}{arXiv:}
\providecommand{\URLprefix}{URL: }
\providecommand{\Pubmedprefix}{pmid:}
\providecommand{\doi}[1]{\href{http://dx.doi.org/#1}{\path{#1}}}
\providecommand{\Pubmed}[1]{\href{pmid:#1}{\path{#1}}}
\providecommand{\bibinfo}[2]{#2}
\ifx\xfnm\relax \def\xfnm[#1]{\unskip,\space#1}\fi
\bibitem[{Bendsøe and Sigmund(2004)}]{TopOptMonograph}
\bibinfo{author}{M.~P. Bendsøe}, \bibinfo{author}{O.~Sigmund}, \bibinfo{title}{Topology Optimization Theory, Methods, and Applications}, \bibinfo{edition}{second} ed., \bibinfo{publisher}{Springer Berlin Heidelberg}, \bibinfo{address}{Berlin, Heidelberg}, \bibinfo{year}{2004}. \DOIprefix\doi{10.1007/978-3-662-05086-6}.
\bibitem[{Deaton and Grandhi(2013)}]{DeatonGrandhi2013}
\bibinfo{author}{J.~Deaton}, \bibinfo{author}{R.~Grandhi},
\newblock \bibinfo{title}{A survey of structural and multidisciplinary continuum topology optimization: post 2000},
\newblock \bibinfo{journal}{Structural and Multidisciplinary Optimization} \bibinfo{volume}{49} (\bibinfo{year}{2013}) \bibinfo{pages}{1--–38}. \DOIprefix\doi{10.1007/s00158-013-0956-z}.
\bibitem[{Sigmund and Maute(2013)}]{TopOptReviewSigmund}
\bibinfo{author}{O.~Sigmund}, \bibinfo{author}{K.~Maute},
\newblock \bibinfo{title}{Topology optimization approaches: A comparative review},
\newblock \bibinfo{journal}{Structural and Multidisciplinary Optimization} \bibinfo{volume}{48} (\bibinfo{year}{2013}) \bibinfo{pages}{1031--1055}. \DOIprefix\doi{10.1007/s00158-013-0978-6}.
\bibitem[{Allaire et~al.(2021)Allaire, Dapogny, and Jouve}]{10.1016/bs.hna.2020.10.004_978-0-444-64305-6_2021}
\bibinfo{author}{G.~Allaire}, \bibinfo{author}{C.~Dapogny}, \bibinfo{author}{F.~Jouve}, \bibinfo{title}{Shape and topology optimization}, volume~\bibinfo{volume}{22}, \bibinfo{publisher}{Elsevier}, \bibinfo{year}{2021}, p. \bibinfo{pages}{1–132}. \DOIprefix\doi{10.1016/bs.hna.2020.10.004}.
\bibitem[{Bendsøe(1989)}]{Bendsoe89}
\bibinfo{author}{M.~P. Bendsøe},
\newblock \bibinfo{title}{Optimal shape design as a material distribution problem},
\newblock \bibinfo{journal}{Structural Optimization} \bibinfo{volume}{1} (\bibinfo{year}{1989}) \bibinfo{pages}{193--202}. \DOIprefix\doi{10.1007/BF01650949}.
\bibitem[{Rozvany et~al.(1992)Rozvany, Zhou, and Birker}]{Rozvanyetal1992}
\bibinfo{author}{G.~I.~N. Rozvany}, \bibinfo{author}{M.~Zhou}, \bibinfo{author}{T.~Birker},
\newblock \bibinfo{title}{Generalized shape optimization without homogenization},
\newblock \bibinfo{journal}{Structural Optimization} \bibinfo{volume}{4} (\bibinfo{year}{1992}) \bibinfo{pages}{250--252}. \DOIprefix\doi{10.1007/BF01742754}.
\bibitem[{Wang et~al.(2003)Wang, Wang, and Guo}]{10.1016/S0045-7825(02)00559-5_2003}
\bibinfo{author}{M.~Y. Wang}, \bibinfo{author}{X.~Wang}, \bibinfo{author}{D.~Guo},
\newblock \bibinfo{title}{A level set method for structural topology optimization},
\newblock \bibinfo{journal}{Computer Methods in Applied Mechanics and Engineering} \bibinfo{volume}{192} (\bibinfo{year}{2003}) \bibinfo{pages}{227–246}. \DOIprefix\doi{10.1016/S0045-7825(02)00559-5}.
\bibitem[{Allaire et~al.(2004)Allaire, Jouve, and Toader}]{10.1016/j.jcp.2003.09.032_2004}
\bibinfo{author}{G.~Allaire}, \bibinfo{author}{F.~Jouve}, \bibinfo{author}{A.-M. Toader},
\newblock \bibinfo{title}{Structural optimization using sensitivity analysis and a level-set method},
\newblock \bibinfo{journal}{Journal of Computational Physics} \bibinfo{volume}{194} (\bibinfo{year}{2004}) \bibinfo{pages}{363–393}. \DOIprefix\doi{10.1016/j.jcp.2003.09.032}.
\bibitem[{Wegert et~al.(2025)Wegert, Manyer, Mallon, Badia, and Challis}]{GridapTopOpt}
\bibinfo{author}{Z.~J. Wegert}, \bibinfo{author}{J.~Manyer}, \bibinfo{author}{C.~N. Mallon}, \bibinfo{author}{S.~Badia}, \bibinfo{author}{V.~J. Challis},
\newblock \bibinfo{title}{{GridapTopOpt.jl: a scalable Julia toolbox for level set-based topology optimisation}},
\newblock \bibinfo{journal}{Structural and Multidisciplinary Optimization} \bibinfo{volume}{68} (\bibinfo{year}{2025}) \bibinfo{pages}{22}. \DOIprefix\doi{10.1007/s00158-024-03927-3}.
\bibitem[{Burman et~al.(2015)Burman, Claus, Hansbo, Larson, and Massing}]{10.1002/nme.4823_2015}
\bibinfo{author}{E.~Burman}, \bibinfo{author}{S.~Claus}, \bibinfo{author}{P.~Hansbo}, \bibinfo{author}{M.~G. Larson}, \bibinfo{author}{A.~Massing},
\newblock \bibinfo{title}{{CutFEM}: Discretizing geometry and partial differential equations},
\newblock \bibinfo{journal}{International Journal for Numerical Methods in Engineering} \bibinfo{volume}{104} (\bibinfo{year}{2015}) \bibinfo{pages}{472–501}. \DOIprefix\doi{10.1002/nme.4823}.
\bibitem[{Badia et~al.(2018)Badia, Verdugo, and Martín}]{Badia_Verdugo_Martin_2018}
\bibinfo{author}{S.~Badia}, \bibinfo{author}{F.~Verdugo}, \bibinfo{author}{A.~F. Martín},
\newblock \bibinfo{title}{The aggregated unfitted finite element method for elliptic problems},
\newblock \bibinfo{journal}{Computer Methods in Applied Mechanics and Engineering} \bibinfo{volume}{336} (\bibinfo{year}{2018}) \bibinfo{pages}{533–553}. \DOIprefix\doi{10.1016/j.cma.2018.03.022}.
\bibitem[{de~Prenter et~al.(2023)de~Prenter, Verhoosel, van Brummelen, Larson, and Badia}]{dePrenter2023}
\bibinfo{author}{F.~de~Prenter}, \bibinfo{author}{C.~V. Verhoosel}, \bibinfo{author}{E.~H. van Brummelen}, \bibinfo{author}{M.~G. Larson}, \bibinfo{author}{S.~Badia},
\newblock \bibinfo{title}{Stability and conditioning of immersed finite element methods: Analysis and remedies},
\newblock \bibinfo{journal}{Archives of Computational Methods in Engineering} \bibinfo{volume}{30} (\bibinfo{year}{2023}) \bibinfo{pages}{3617–3656}. \URLprefix \url{http://dx.doi.org/10.1007/s11831-023-09913-0}. \DOIprefix\doi{10.1007/s11831-023-09913-0}.
\bibitem[{Burman et~al.(2018)Burman, Elfverson, Hansbo, Larson, and Larsson}]{10.1016/j.cma.2017.09.005_2018}
\bibinfo{author}{E.~Burman}, \bibinfo{author}{D.~Elfverson}, \bibinfo{author}{P.~Hansbo}, \bibinfo{author}{M.~G. Larson}, \bibinfo{author}{K.~Larsson},
\newblock \bibinfo{title}{Shape optimization using the cut finite element method},
\newblock \bibinfo{journal}{Computer Methods in Applied Mechanics and Engineering} \bibinfo{volume}{328} (\bibinfo{year}{2018}) \bibinfo{pages}{242–261}. \DOIprefix\doi{10.1016/j.cma.2017.09.005}.
\bibitem[{Villanueva and Maute(2017)}]{10.1016/j.cma.2017.03.007_2017}
\bibinfo{author}{C.~H. Villanueva}, \bibinfo{author}{K.~Maute},
\newblock \bibinfo{title}{{CutFEM} topology optimization of 3d laminar incompressible flow problems},
\newblock \bibinfo{journal}{Computer Methods in Applied Mechanics and Engineering} \bibinfo{volume}{320} (\bibinfo{year}{2017}) \bibinfo{pages}{444–473}. \DOIprefix\doi{10.1016/j.cma.2017.03.007}.
\bibitem[{Bernland et~al.(2018)Bernland, Wadbro, and Berggren}]{10.1002/nme.5621_2018}
\bibinfo{author}{A.~Bernland}, \bibinfo{author}{E.~Wadbro}, \bibinfo{author}{M.~Berggren},
\newblock \bibinfo{title}{Acoustic shape optimization using cut finite elements: Acoustic shape optimization using cut finite elements},
\newblock \bibinfo{journal}{International Journal for Numerical Methods in Engineering} \bibinfo{volume}{113} (\bibinfo{year}{2018}) \bibinfo{pages}{432–449}. \DOIprefix\doi{10.1002/nme.5621}.
\bibitem[{Berggren(2023)}]{Berggren_2023}
\bibinfo{author}{M.~Berggren},
\newblock \bibinfo{title}{{Shape calculus for fitted and unfitted discretizations: Domain transformations vs. boundary-face dilations}},
\newblock \bibinfo{journal}{Communications in Optimization Theory} \bibinfo{volume}{2023} (\bibinfo{year}{2023}). \URLprefix \url{https://cot.mathres.org/archives/1568}. \DOIprefix\doi{10.23952/cot.2023.27}.
\bibitem[{Schmidt(2018)}]{Schmidt_2018}
\bibinfo{author}{S.~Schmidt},
\newblock \bibinfo{title}{Weak and strong form shape hessians and their automatic generation},
\newblock \bibinfo{journal}{SIAM Journal on Scientific Computing} \bibinfo{volume}{40} (\bibinfo{year}{2018}) \bibinfo{pages}{C210–C233}. \DOIprefix\doi{10.1137/16M1099972}.
\bibitem[{Schmidt et~al.(2018)Schmidt, Schütte, and Walther}]{Schmidt_Schutte_Walther_2018}
\bibinfo{author}{S.~Schmidt}, \bibinfo{author}{M.~Schütte}, \bibinfo{author}{A.~Walther},
\newblock \bibinfo{title}{Efficient numerical solution of geometric inverse problems involving maxwell’s equations using shape derivatives and automatic code generation},
\newblock \bibinfo{journal}{SIAM Journal on Scientific Computing} \bibinfo{volume}{40} (\bibinfo{year}{2018}) \bibinfo{pages}{B405–B428}. \DOIprefix\doi{10.1137/16M110602X}.
\bibitem[{Ham et~al.(2019)Ham, Mitchell, Paganini, and Wechsung}]{Ham_Mitchell_Paganini_Wechsung_2019}
\bibinfo{author}{D.~A. Ham}, \bibinfo{author}{L.~Mitchell}, \bibinfo{author}{A.~Paganini}, \bibinfo{author}{F.~Wechsung},
\newblock \bibinfo{title}{Automated shape differentiation in the unified form language},
\newblock \bibinfo{journal}{Structural and Multidisciplinary Optimization} \bibinfo{volume}{60} (\bibinfo{year}{2019}) \bibinfo{pages}{1813–1820}. \DOIprefix\doi{10.1007/s00158-019-02281-z}.
\bibitem[{Neofytou et~al.(2024)Neofytou, Rios, Bujny, Menzel, and Kim}]{Neofytou_Rios_Bujny_Menzel_Kim_2024}
\bibinfo{author}{A.~Neofytou}, \bibinfo{author}{T.~Rios}, \bibinfo{author}{M.~Bujny}, \bibinfo{author}{S.~Menzel}, \bibinfo{author}{H.~A. Kim},
\newblock \bibinfo{title}{Level set topology optimization with sparse automatic differentiation},
\newblock \bibinfo{journal}{Structural and Multidisciplinary Optimization} \bibinfo{volume}{67} (\bibinfo{year}{2024}) \bibinfo{pages}{178}. \DOIprefix\doi{10.1007/s00158-024-03894-9}.
\bibitem[{Dokken et~al.(2020)Dokken, Mitusch, and Funke}]{Dokken_Mitusch_Funke_2020}
\bibinfo{author}{J.~S. Dokken}, \bibinfo{author}{S.~K. Mitusch}, \bibinfo{author}{S.~W. Funke},
\newblock \bibinfo{title}{Automatic shape derivatives for transient pdes in fenics and firedrake}  (\bibinfo{year}{2020}). \URLprefix \url{http://arxiv.org/abs/2001.10058}. \DOIprefix\doi{10.48550/arXiv.2001.10058}, \bibinfo{note}{arXiv:2001.10058 [math]}.
\bibitem[{Gangl et~al.(2021)Gangl, Sturm, Neunteufel, and Schöberl}]{Gangl_Sturm_Neunteufel_Schöberl_2021}
\bibinfo{author}{P.~Gangl}, \bibinfo{author}{K.~Sturm}, \bibinfo{author}{M.~Neunteufel}, \bibinfo{author}{J.~Schöberl},
\newblock \bibinfo{title}{Fully and semi-automated shape differentiation in ngsolve},
\newblock \bibinfo{journal}{Structural and Multidisciplinary Optimization} \bibinfo{volume}{63} (\bibinfo{year}{2021}) \bibinfo{pages}{1579–1607}. \DOIprefix\doi{10.1007/s00158-020-02742-w}.
\bibitem[{Mallon et~al.(2025)Mallon, Thornton, Hill, and Badia}]{mallon2024neurallevelsettopology}
\bibinfo{author}{C.~N. Mallon}, \bibinfo{author}{A.~W. Thornton}, \bibinfo{author}{M.~R. Hill}, \bibinfo{author}{S.~Badia},
\newblock \bibinfo{title}{{Neural Level Set Topology Optimization Using Unfitted Finite Elements}},
\newblock \bibinfo{journal}{International Journal for Numerical Methods in Engineering} \bibinfo{volume}{126} (\bibinfo{year}{2025}) \bibinfo{pages}{e70004}. \URLprefix \url{https://onlinelibrary.wiley.com/doi/abs/10.1002/nme.70004}. \DOIprefix\doi{https://doi.org/10.1002/nme.70004}.
\bibitem[{Badia and Verdugo(2020)}]{Badia2020}
\bibinfo{author}{S.~Badia}, \bibinfo{author}{F.~Verdugo},
\newblock \bibinfo{title}{{Gridap:} an extensible finite element toolbox in {Julia}},
\newblock \bibinfo{journal}{Journal of Open Source Software} \bibinfo{volume}{5} (\bibinfo{year}{2020}) \bibinfo{pages}{2520}. \DOIprefix\doi{10.21105/joss.02520}.
\bibitem[{Verdugo and Badia(2022)}]{Verdugo2022}
\bibinfo{author}{F.~Verdugo}, \bibinfo{author}{S.~Badia},
\newblock \bibinfo{title}{The software design of {Gridap}: A finite element package based on the {Julia} {JIT} compiler},
\newblock \bibinfo{journal}{Computer Physics Communications} \bibinfo{volume}{276} (\bibinfo{year}{2022}) \bibinfo{pages}{108341}. \DOIprefix\doi{10.1016/j.cpc.2022.108341}.
\bibitem[{C\'ea(1986)}]{10.1051/m2an/1986200303711_1986}
\bibinfo{author}{J.~C\'ea},
\newblock \bibinfo{title}{Conception optimale ou identification de formes, calcul rapide de la dérivée directionnelle de la fonction coût},
\newblock \bibinfo{journal}{ESAIM: Mathematical Modelling and Numerical Analysis} \bibinfo{volume}{20} (\bibinfo{year}{1986}) \bibinfo{pages}{371–402}. \DOIprefix\doi{10.1051/m2an/1986200303711}.
\bibitem[{Badia et~al.(2022)Badia, Martín, and Verdugo}]{Badia2022}
\bibinfo{author}{S.~Badia}, \bibinfo{author}{A.~F. Martín}, \bibinfo{author}{F.~Verdugo},
\newblock \bibinfo{title}{{GridapDistributed}: a massively parallel finite element toolbox in julia},
\newblock \bibinfo{journal}{Journal of Open Source Software} \bibinfo{volume}{7} (\bibinfo{year}{2022}) \bibinfo{pages}{4157}. \URLprefix \url{https://doi.org/10.21105/joss.04157}. \DOIprefix\doi{10.21105/joss.04157}.
\bibitem[{Burman(2010)}]{Burman2010}
\bibinfo{author}{E.~Burman},
\newblock \bibinfo{title}{Ghost penalty},
\newblock \bibinfo{journal}{Comptes Rendus. Mathématique} \bibinfo{volume}{348} (\bibinfo{year}{2010}) \bibinfo{pages}{1217–1220}. \URLprefix \url{http://dx.doi.org/10.1016/j.crma.2010.10.006}. \DOIprefix\doi{10.1016/j.crma.2010.10.006}.
\bibitem[{Dapogny et~al.(2017)Dapogny, Faure, Michailidis, Allaire, Couvelas, and Estevez}]{10.1007/s00466-017-1383-6_2017}
\bibinfo{author}{C.~Dapogny}, \bibinfo{author}{A.~Faure}, \bibinfo{author}{G.~Michailidis}, \bibinfo{author}{G.~Allaire}, \bibinfo{author}{A.~Couvelas}, \bibinfo{author}{R.~Estevez},
\newblock \bibinfo{title}{Geometric constraints for shape and topology optimization in architectural design},
\newblock \bibinfo{journal}{Computational Mechanics} \bibinfo{volume}{59} (\bibinfo{year}{2017}) \bibinfo{pages}{933–965}. \DOIprefix\doi{10.1007/s00466-017-1383-6}.
\bibitem[{Baydin et~al.(2018)Baydin, Pearlmutter, Radul, and Siskind}]{Baydin_Pearlmutter_Radul_Siskind_2018}
\bibinfo{author}{A.~G. Baydin}, \bibinfo{author}{B.~A. Pearlmutter}, \bibinfo{author}{A.~A. Radul}, \bibinfo{author}{J.~M. Siskind},
\newblock \bibinfo{title}{{Automatic Differentiation in Machine Learning: a Survey}},
\newblock \bibinfo{journal}{Journal of Machine Learning Research} \bibinfo{volume}{18} (\bibinfo{year}{2018}) \bibinfo{pages}{1–43}.
\bibitem[{Verdugo et~al.(2020)Verdugo, Neiva, Martorell, and Badia}]{GridapEmbedded}
\bibinfo{author}{F.~Verdugo}, \bibinfo{author}{E.~Neiva}, \bibinfo{author}{P.~A. Martorell}, \bibinfo{author}{S.~Badia}, \bibinfo{title}{{GridapEmbedded}}, \bibinfo{howpublished}{\url{https://github.com/gridap/GridapEmbedded.jl}}, \bibinfo{year}{2020}.
\bibitem[{Burman et~al.(2017)Burman, Elfverson, Hansbo, Larson, and Larsson}]{Burman_Elfverson_Hansbo_Larson_Larsson_2017}
\bibinfo{author}{E.~Burman}, \bibinfo{author}{D.~Elfverson}, \bibinfo{author}{P.~Hansbo}, \bibinfo{author}{M.~G. Larson}, \bibinfo{author}{K.~Larsson},
\newblock \bibinfo{title}{{A cut finite element method for the Bernoulli free boundary value problem}},
\newblock \bibinfo{journal}{Computer Methods in Applied Mechanics and Engineering} \bibinfo{volume}{317} (\bibinfo{year}{2017}) \bibinfo{pages}{598–618}. \DOIprefix\doi{10.1016/j.cma.2016.12.021}.
\bibitem[{Burman and Fernández(2009)}]{Burman_Fernández_2009}
\bibinfo{author}{E.~Burman}, \bibinfo{author}{M.~A. Fernández},
\newblock \bibinfo{title}{Finite element methods with symmetric stabilization for the transient convection–diffusion–reaction equation},
\newblock \bibinfo{journal}{Computer Methods in Applied Mechanics and Engineering} \bibinfo{volume}{198} (\bibinfo{year}{2009}) \bibinfo{pages}{2508–2519}. \DOIprefix\doi{10.1016/j.cma.2009.02.011}.
\bibitem[{Osher and Fedkiw(2006)}]{978-0-387-22746-7_2006}
\bibinfo{author}{S.~Osher}, \bibinfo{author}{R.~Fedkiw}, \bibinfo{title}{Level Set Methods and Dynamic Implicit Surfaces}, Applied Mathematical Sciences, \bibinfo{edition}{1} ed., \bibinfo{publisher}{Springer Science \& Business Media}, \bibinfo{year}{2006}.
\bibitem[{Peng et~al.(1999)Peng, Merriman, Osher, Zhao, and Kang}]{10.1006/jcph.1999.6345_1999}
\bibinfo{author}{D.~Peng}, \bibinfo{author}{B.~Merriman}, \bibinfo{author}{S.~Osher}, \bibinfo{author}{H.~Zhao}, \bibinfo{author}{M.~Kang},
\newblock \bibinfo{title}{A pde-based fast local level set method},
\newblock \bibinfo{journal}{Journal of Computational Physics} \bibinfo{volume}{155} (\bibinfo{year}{1999}) \bibinfo{pages}{410–438}. \DOIprefix\doi{10.1006/jcph.1999.6345}.
\bibitem[{Badia et~al.(2022)Badia, Martorell, and Verdugo}]{badia_martorell_2022}
\bibinfo{author}{S.~Badia}, \bibinfo{author}{P.~A. Martorell}, \bibinfo{author}{F.~Verdugo},
\newblock \bibinfo{title}{Geometrical discretisations for unfitted finite elements on explicit boundary representations},
\newblock \bibinfo{journal}{Journal of Computational Physics} \bibinfo{volume}{460} (\bibinfo{year}{2022}) \bibinfo{pages}{111162}. \URLprefix \url{https://www.sciencedirect.com/science/article/pii/S0021999122002248}. \DOIprefix\doi{https://doi.org/10.1016/j.jcp.2022.111162}.
\bibitem[{Martorell and Badia(2025)}]{Martorell2025}
\bibinfo{author}{P.~A. Martorell}, \bibinfo{author}{S.~Badia},
\newblock \bibinfo{title}{Stlcutters.jl: A scalable geometrical framework library for unfitted finite element discretisations},
\newblock \bibinfo{journal}{Computer Physics Communications} \bibinfo{volume}{309} (\bibinfo{year}{2025}) \bibinfo{pages}{109479}. \URLprefix \url{http://dx.doi.org/10.1016/j.cpc.2024.109479}. \DOIprefix\doi{10.1016/j.cpc.2024.109479}.
\bibitem[{Parvizian et~al.(2012)Parvizian, Düster, and Rank}]{Parvizian_Duster_Rank_2012}
\bibinfo{author}{J.~Parvizian}, \bibinfo{author}{A.~Düster}, \bibinfo{author}{E.~Rank},
\newblock \bibinfo{title}{Topology optimization using the finite cell method},
\newblock \bibinfo{journal}{Optimization and Engineering} \bibinfo{volume}{13} (\bibinfo{year}{2012}) \bibinfo{pages}{57–78}. \DOIprefix\doi{10.1007/s11081-011-9159-x}.
\bibitem[{Nocedal and Wright(2006)}]{978-0-387-30303-1_2006}
\bibinfo{author}{J.~Nocedal}, \bibinfo{author}{S.~J. Wright}, \bibinfo{title}{Numerical optimization}, Springer series in operations research, \bibinfo{edition}{2nd ed} ed., \bibinfo{publisher}{Springer}, \bibinfo{address}{New York}, \bibinfo{year}{2006}.
\bibitem[{Geuzaine and Remacle(2009)}]{10.1002/nme.2579}
\bibinfo{author}{C.~Geuzaine}, \bibinfo{author}{J.-F. Remacle},
\newblock \bibinfo{title}{Gmsh: A 3-d finite element mesh generator with built-in pre- and post-processing facilities},
\newblock \bibinfo{journal}{International Journal for Numerical Methods in Engineering} \bibinfo{volume}{79} (\bibinfo{year}{2009}) \bibinfo{pages}{1309--1331}. \URLprefix \url{https://onlinelibrary.wiley.com/doi/abs/10.1002/nme.2579}. \DOIprefix\doi{doi.org/10.1002/nme.2579}. \href{http://arxiv.org/abs/https://onlinelibrary.wiley.com/doi/pdf/10.1002/nme.2579}{{\tt arXiv:https://onlinelibrary.wiley.com/doi/pdf/10.1002/nme.2579}}.
\bibitem[{Badia and Verdugo(2019)}]{gridapgmsh}
\bibinfo{author}{S.~Badia}, \bibinfo{author}{F.~Verdugo}, \bibinfo{title}{{GridapGmsh}}, \bibinfo{howpublished}{\url{https://github.com/gridap/GridapGmsh.jl}}, \bibinfo{year}{2019}.
\bibitem[{Demmel et~al.(1999)Demmel, Gilbert, and Li}]{Demmel_Gilbert_Li_1999}
\bibinfo{author}{J.~W. Demmel}, \bibinfo{author}{J.~R. Gilbert}, \bibinfo{author}{X.~S. Li}, \bibinfo{title}{{SuperLU Users’ Guide}}, \bibinfo{type}{Technical Report} \bibinfo{number}{LBNL-44289}, Lawrence Berkeley National Laboratory, \bibinfo{year}{1999}. \URLprefix \url{http://www.osti.gov/servlets/purl/751785-85h6HO/webviewable/}. \DOIprefix\doi{10.2172/751785}.
\bibitem[{Li and Demmel(2003)}]{Li_Demmel_2003}
\bibinfo{author}{X.~S. Li}, \bibinfo{author}{J.~W. Demmel},
\newblock \bibinfo{title}{{SuperLU\_DIST: A scalable distributed-memory sparse direct solver for unsymmetric linear systems}},
\newblock \bibinfo{journal}{ACM Trans. Math. Softw.} \bibinfo{volume}{29} (\bibinfo{year}{2003}) \bibinfo{pages}{110--140}. \DOIprefix\doi{10.1145/779359.779361}.
\bibitem[{Li et~al.(2023)Li, Lin, Liu, and Sao}]{Li_Lin_Liu_Sao_2023}
\bibinfo{author}{X.~S. Li}, \bibinfo{author}{P.~Lin}, \bibinfo{author}{Y.~Liu}, \bibinfo{author}{P.~Sao},
\newblock \bibinfo{title}{{Newly Released Capabilities in the Distributed-Memory SuperLU Sparse Direct Solver}},
\newblock \bibinfo{journal}{ACM Trans. Math. Softw.} \bibinfo{volume}{49} (\bibinfo{year}{2023}) \bibinfo{pages}{10:1--10:20}. \DOIprefix\doi{10.1145/3577197}.
\bibitem[{Wegert et~al.(2023)Wegert, Roberts, and Challis}]{Wegert_2023b}
\bibinfo{author}{Z.~J. Wegert}, \bibinfo{author}{A.~P. Roberts}, \bibinfo{author}{V.~J. Challis},
\newblock \bibinfo{title}{A {H}ilbertian projection method for constrained level set-based topology optimisation},
\newblock \bibinfo{journal}{Structural and Multidisciplinary Optimization} \bibinfo{volume}{66} (\bibinfo{year}{2023}) \bibinfo{pages}{204}. \DOIprefix\doi{10.1007/s00158-023-03663-0}.
\bibitem[{Feppon et~al.(2020{\natexlab{a}})Feppon, Allaire, and Dapogny}]{10.1051/cocv/2020015_2020}
\bibinfo{author}{F.~Feppon}, \bibinfo{author}{G.~Allaire}, \bibinfo{author}{C.~Dapogny},
\newblock \bibinfo{title}{Null space gradient flows for constrained optimization with applications to shape optimization},
\newblock \bibinfo{journal}{ESAIM: Control, Optimisation and Calculus of Variations} \bibinfo{volume}{26} (\bibinfo{year}{2020}{\natexlab{a}}) \bibinfo{pages}{90}. \DOIprefix\doi{10.1051/cocv/2020015}.
\bibitem[{Feppon et~al.(2020{\natexlab{b}})Feppon, Allaire, Dapogny, and Jolivet}]{10.1016/j.jcp.2020.109574_2020}
\bibinfo{author}{F.~Feppon}, \bibinfo{author}{G.~Allaire}, \bibinfo{author}{C.~Dapogny}, \bibinfo{author}{P.~Jolivet},
\newblock \bibinfo{title}{Topology optimization of thermal fluid–structure systems using body-fitted meshes and parallel computing},
\newblock \bibinfo{journal}{Journal of Computational Physics} \bibinfo{volume}{417} (\bibinfo{year}{2020}{\natexlab{b}}) \bibinfo{pages}{109574}. \DOIprefix\doi{10.1016/j.jcp.2020.109574}.

\end{thebibliography}





\end{document}